\documentclass[11pt,twoside]{article}

\usepackage{geometry}
\input{commands}
\def\<{\langle}
\def\>{\rangle}
\def\reals{{\mathbb R}}
\def\P{{\mathbb P}}

\def\hbt{\widehat{\boldsymbol \theta}}

\def\pzn{\partial_{Z}}
\def\px{\partial_{X_{k\ell}}}
\def\pxn{\partial_{X}}

\def\wtln{\mathcal{L}_r}

\def\hbtloos{\widehat{\bt}^{(\backslash 1)}}

\def\br{\boldsymbol r}

\def\pproj{\proj^{\perp}_{\bt_0}}

\def\wx{\widebar{x}}
\def\wbx{\widebar{\bx}}
\def\wz{\widebar{z}}
\def\wbz{\widebar{\bz}}

\allowdisplaybreaks

\begin{document}
\begin{center}
	
	{{\LARGE{ \mbox{High-dimensional logistic regression with missing data:}\\ \mbox{Imputation, regularization, and universality}}}}
	
	\vspace*{.2in}
	
	{\large{
			\begin{tabular}{ccc}
				Kabir Aladin Verchand$^{\star,\circ}$ and Andrea Montanari$^{\dagger,\ddagger}$
			\end{tabular}
	}}
	\vspace*{.2in}
	
	\begin{tabular}{c}
		$^{\star}$Statistical Laboratory, University of Cambridge\\
		$^{\circ}$Schools of Industrial and Systems Engineering, Georgia Institute of Technology\\
		$^{\dagger}$Department of Mathematics, Stanford University\\
		$^{\ddagger}$Department of Statistics, Stanford University
	\end{tabular}
	
	\vspace*{.2in}

	\today
	
	\vspace*{.2in}

\begin{abstract}
	We study high-dimensional, ridge-regularized logistic regression in a setting in which the covariates may be missing or corrupted by additive noise.  When both the covariates and the additive corruptions are independent and normally distributed, we provide exact characterizations of both the prediction error as well as the estimation error.  Moreover, we show that these characterizations are universal: as long as the entries of the data matrix satisfy a set of independence and moment conditions, our guarantees continue to hold.  Universality, in turn, enables the detailed study of several imputation-based strategies when the covariates are missing completely at random.  We ground our study by comparing the performance of these strategies with the conjectured performance---stemming from replica theory in statistical physics---of the Bayes optimal procedure.  Our analysis yields several insights including: \emph{(i)} a distinction between single imputation and a simple variant of multiple imputation and \emph{(ii)} that adding a simple ridge regularization term to single-imputed logistic regression can yield an estimator whose prediction error is nearly indistinguishable from the Bayes optimal prediction error.  We supplement our findings with extensive numerical experiments.
\end{abstract}
\end{center}

%

\section{Introduction}
\label{sec:introduction}
Statistical methodology is developed under the assumption that the data is fully observed.  In practice, however, this is often not the case.  For instance, data can be missing due to non-response in surveys~\citep{rubin2004multiple}, instruments malfunctioning in scientific investigations~\citep{do2018characterization}, or the integration of multi-modal data~\citep{du2022robust}, to name a few.  

The most prevalent strategy for dealing with missing data consists of an appealing two stage approach in which:  first, the statistician imputes, or fills in, the data to generate one or multiple complete data sets from the observed data; and second, the statistician uses their preferred complete-data method on the imputed datasets.  Such imputation-based methods form the most popular approaches to missing data and software packages that implement these methods are ubiquitous in statistical practice~\citep[see, e.g.,][]{buuren2010mice,su2011multiple}.  The reason for this prevalence is clear: imputation decouples the problem of handling the missing data from the downstream task of estimation or prediction.  Moreover, when the dimension of the parameters $p$ is fixed and the number of samples $n$ tends to infinity,~\citet{wang1998large} developed asymptotic normality theory to enable rigorous comparisons between various imputation-based procedures.  In regression settings, these methods typically rely on an initial estimate of the regression coefficients which is consistent and asymptotically linear, such as \emph{(i)} an estimate obtained from a complete case analysis---in which all samples with missing data are discarded---or  \emph{(ii)} the maximum likelihood estimator.

Unfortunately, in modern large scale and high-dimensional applications, such estimates are difficult to obtain in general.  For instance, when the data is high-dimensional, most of the samples will contain missing data and a complete-case analysis proves untenable.  Moreover, performing maximum likelihood estimation in the presence of missing data typically involves optimizing a non-concave log-likelihood and can suffer from the curse of dimensionality.  Consequently, it is exceedingly important in high dimensions to develop regression procedures which are computationally efficient and simultaneously yield statistically useful results.  Motivated by these issues, researchers have devoted significant effort to developing methods and theory to cope with missing data in the context of high-dimensional sparse linear regression~\citep[e.g.][]{rosenbaum2010sparse, loh2012high, datta2017cocolasso}.  
These methods, while yielding theoretical guarantees in high-dimensions, do not generalize in a straightforward manner to problems with categorical responses, e.g. in generalized linear models, in which different phenomena appear (see Section~\ref{sec:motivating-example} to follow).  
Towards understanding the effect of missing data in generalized linear models with categorical responses, we consider a simplified, analytically tractable setting in which the data matrix is random---namely, consisting of i.i.d. entries---and the data is missing completely at random.  We next describe this setting in detail.

\subsection{Problem set-up} \label{sec:problem-set-up}
We consider $n$ covariate response pairs $\{(\bx_i,y_i)\}_{i = 1}^{n}$, where $\bx_i\in \reals^p$ and $y_i\in\{+1,-1\}$ and form the data matrix $\bX \in \mathbb{R}^{n \times p}$ and response vector $\by \in \mathbb{R}^n$ as
\[
\bX = [\bx_1 \;\; \bx_2 \;\; \dots \;\; \bx_n]^{\top} \qquad \text{ and } \qquad \by = [y_1 \;\; y_2 \;\; \dots \;\; y_n]^{\top}.
\]
Throughout, we will assume that the covariates $\bx_i$ consist of i.i.d., zero-mean entries; that is the entries of the data matrix $(X_{ij})_{1 \leq i \leq n, 1 \leq j \leq p}$ are i.i.d. and zero-mean.  The conditional distribution of the response $y_i$ given the covariate vector $\bx_i$ is modeled as 
\begin{align}
	\label{eq:data_gen}
	\P\bigl(y_i = 1 \mid  \langle \bx_i, \bt_0 \rangle \bigr) = \rho'\bigl(\langle \bx_i, \bt_0 \rangle\bigr),
\end{align}
where $\rho: t \mapsto \log(1 + e^t)$ denotes the logistic link function and $\bt_0 \in \mathbb{R}^p$ is a vector of coefficients.  Rather than observing the pair $\bigl(\bX, \by\bigr)$, the statistician instead observes the pair \sloppy\mbox{$\bigl(\bX^{\mathsf{obs}}, \by \bigr) \in \mathbb{R}^{n \times p} \times \mathbb{R}^n$}.  We next describe the different observed data settings considered in the paper. 

\subsubsection{Missing data and error-in-variables models}\label{sec:missing-setup}
We consider two possible observed data models: data which is missing completely at random (MCAR) and a Gaussian error-in-variables model.

\paragraph{Missing completely at random:}
This is the simplest possible missing data model and is parametrized by a scalar $\alpha \in [0, 1]$.  The observed data matrix $\bX^{\mathsf{obs}}$ consists of entries 
\begin{align}
	\label{eq:MCAR}
	X_{ij}^{\mathsf{obs}} = \begin{cases} 
		X_{ij} & \text{ with probability } \alpha, \\
		* & \text{ with probability } 1 - \alpha.
	\end{cases}\tag{MCAR($\alpha$)}
\end{align}
When the observed data is generated in this way, we say that the data is MCAR($\alpha$).  While MCAR forms a strong and often overly simplistic assumption~\citep[see, e.g.][for examples of more realistic settings]{mckennan2020estimation,rubin2004multiple}, we adopt it here in pursuit of studying phenomena that arise in logistic models with missing data.
\hfill $\diamondsuit$

\paragraph{Gaussian error-in-variables:}
Following~\citet{berkson1950there}, we consider an ensemble consisting of covariates corrupted by additive noise.  We parametrize the model by positive scalars $\alpha_c$ and $\alpha_2$ which satisfy the relation $\alpha_c \leq \sqrt{\alpha_2}$ and define the observed data as
\begin{align}\label{eq:intro-gaussian-eiv}
	\bX^{\mathsf{obs}} = \alpha_c \cdot \bX+ \sqrt{\alpha_2 - \alpha_c^2} \cdot \bW,
\end{align}
where the unobserved random matrix $\bW$ is independent of the data $\bX$ and consists of i.i.d. entries $W_{ij} \sim \mathsf{N}(0, 1/p)$.  We note that our interest in this model stems from our proof technique in which we reduce the study~\ref{eq:MCAR} data to an equivalent error-in-variables model.  Consequently, we restrict our treatment to the simple Gaussian model~\eqref{eq:intro-gaussian-eiv}.  We refer the interested reader to~\citet{bickel1987efficient} and~\citet{rudelson2017errors} for theoretical treatments in settings with more general assumptions on the data and noise (although incomparable to our results here).  \hfill $\diamondsuit$

\subsubsection{Imputation methods}
We now describe three common imputation methods: single imputation, prior imputation, and multiple imputation.  Our analysis, presented in Section~\ref{sec:main_res} will provide guarantees for the first two methods, leaving the high-dimensional analysis of multiple imputation to an interesting question for further exploration.  We include its description here for completeness and as a point of comparison.

\paragraph{Single imputation:} This method requires knowledge of the conditional mean of the missing covariate given the observed covariates.  Recalling that we consider data which consists of i.i.d., zero-mean entries, the imputed matrix $\bZ$ consists of the entries
\begin{align}
	\label{eq:single_imp}
	Z_{ij} = \begin{cases} 
		X_{ij} & \text{ if observed}, \\
		0 & \text{ else}.
	\end{cases}\tag{Single imputation}
\end{align}
We note that typical treatments of missing data methodology eschew the use of single imputation in favor of multiple imputation, to be described shortly.  In line with~\citet{josse2019consistency}, we demonstrate in the sequel that this rule of thumb may be misleading in high-dimensions, especially for downstream tasks such as prediction.
\hfill $\diamondsuit$

\paragraph{Prior imputation:} This method requires knowledge of the distribution of the covariates and consists of replacing each missing entry with a fresh draw from the covariates distribution.  In our setting, this consists of drawing a random matrix $\widetilde{\bX}$ from the same distribution as the data matrix $\bX$ and setting the entries of the imputed matrix $\bZ$ as 
\begin{align}
	\label{eq:prior_imp}
	Z_{ij} = \begin{cases} 
		X_{ij} & \text{ if observed}, \\
		\widetilde{X}_{ij} & \text{ else}.
	\end{cases} \tag{Prior imputation}
\end{align}
\hfill $\diamondsuit$

\paragraph{Multiple imputation:} This is the most complicated of the three methods and is the method advocated by traditional statistical methodology~\citep{little1992regression,murray2018multiple}.  Here, a parameter $M$ is fixed and several completed data sets $\bZ_1, \bZ_2, \dots, \bZ_M$ are generated by sampling
\begin{align}\label{eq:multiple-imp}
	\bZ_{k} \sim \P\bigl\{\bX \in \cdot \; \mid \bX^{\mathsf{obs}}, \by, \widehat{\bt}\bigr\}, \qquad \text{ for all } \qquad k \in \{1, 2, \dots, M\}, \tag{Multiple imputation}
\end{align}
where $\widehat{\bt}$ is an initial estimate: typically either a consistent asymptotically linear estimate or a single draw from the posterior given the observed data~\citep[\S 3.1]{wang1998large}.
When the downstream analysis is performed, the preferred regression method is used to generate estimates $\widehat{\bt}_1, \dots, \widehat{\bt}_M$ and the resulting estimates are aggregated to obtain the final estimate.  \hfill $\diamondsuit$

Once the imputed matrix $\bZ$ is set, we estimate the coefficients $\bt$ using ridge-regularized logistic regression.  That is, given a regularization parameter $\lambda$, we minimize the cost \sloppy\mbox{$\mathcal{L}_n:\mathbb{R}^p \rightarrow \mathbb{R}$}, defined as
\begin{subequations}
	\begin{align}
		\label{eq:loss}
		\mathcal{L}_n(\bt; \bZ, \lambda) = \frac{1}{n} \sum_{i=1}^{n} \rho(-y_i \langle \bz_i, \bt \rangle) + \frac{\lambda}{2p}\norm{\bt}_2^2,
	\end{align}
	to obtain the estimate
	\begin{align}
		\label{eq:estimate}
		\widehat{\bt}(\bZ, \lambda) = \argmin_{\bt \in \mathbb{R}^p}\; \mathcal{L}_n(\bt; \bZ, \lambda).
	\end{align}
\end{subequations}
We note that the dependence on the imputed matrix $\bZ$ is made explicit as we will consider this estimator for several different choices of the matrix $\bZ$.

\subsection{A motivating example} \label{sec:motivating-example}
In order to motivate our treatment, we consider a small simulation study.  First, we consider the situation in low dimensions.  In particular, we run a simulation in which the dimension is fixed as $p=2$, the probability with which an entry is observed is fixed as $\alpha = 0.85$, the data matrix $\bX$ consists of entries $(X_{ij})_{i \leq n, j \leq p} \overset{\mathsf{i.i.d.}}{\sim} \mathsf{N}(0, 1/2)$, and the ground truth is fixed as $\bt_0 = (1, 1)$.  The number of samples $n$ is varied from $n=20$ to $n=20000$.  We then simulate the observed data to be~\ref{eq:MCAR} and form the imputed data matrix $\bZ$ by using one of three methods:~\ref{eq:single_imp},~\ref{eq:prior_imp}, or complete cases.  We subsequently minimize the loss $\mathcal{L}_n$~\eqref{eq:loss} without regularization ($\lambda = 0$) to obtain estimates $\widehat{\bt}$ and measure the angular error $\angle\bigl(\widehat{\bt}, \bt_0 \bigr)$ as well as the mean squared error $\| \widehat{\bt} - \bt_0 \|_2^2$.  We repeat this sequence independently $1000$ times.  The results are plotted in Figure~\ref{fig:intro-low-dim}.  
\begin{figure*}[!h]
	\begin{subfigure}[b]{0.45\textwidth}
		\centerline{\begin{tikzpicture}
	\begin{axis}[
		xlabel={Sample size $n$},
		ylabel={Angle error},
		ymode=log,
		xmode=log,
		legend pos=north east,
		legend style={font=\small,fill=none},
		ymajorgrids=true,
		grid style=grid,
		enlargelimits=false,
		]
		
		\addplot[
		color=CadetBlue,
		mark=triangle,
		mark size=2pt,
		line width=1.25pt   
		]
		table[x=delta,y=avg_angle] {data/intro-low-dim.dat};
		\legend{Single imputation}
		
		\addplot[
		color=ForestGreen,
		mark=triangle,
		mark size=2pt,
		line width=1.25pt   
		]
		table[x=delta,y=avg_angle_mi] {data/intro-low-dim.dat};
		\addlegendentry{Complete cases}
		
		\addplot[
		color=RoyalBlue,
		mark=triangle,
		mark size=2pt,
		line width=1.25pt   
		]
		table[x=delta,y=avg_angle_mi_prior] {data/intro-low-dim.dat};
		\addlegendentry{Prior imputation}

		\addplot+[name path=A-mi-prior,RoyalBlue!20, no markers] table[x=delta,y=lower_angle_mi_prior] {data/intro-low-dim.dat};
		\addplot+[name path=B-mi-prior,RoyalBlue!20, no markers] table[x=delta,y=upper_angle_mi_prior] {data/intro-low-dim.dat};
		
		\addplot[RoyalBlue!20] fill between[of=A-mi-prior and B-mi-prior];
		
		\addplot+[name path=A-mi,ForestGreen!20, no markers] table[x=delta,y=lower_angle_mi] {data/intro-low-dim.dat};
		\addplot+[name path=B-mi,ForestGreen!20, no markers] table[x=delta,y=upper_angle_mi] {data/intro-low-dim.dat};
		
		\addplot[ForestGreen!20] fill between[of=A-mi and B-mi];

		\addplot+[name path=A-si,CadetBlue!20, no markers] table[x=delta,y=lower_angle] {data/intro-low-dim.dat};
		\addplot+[name path=B-si,CadetBlue!20, no markers] table[x=delta,y=upper_angle] {data/intro-low-dim.dat};
		
		\addplot[CadetBlue!20] fill between[of=A-si and B-si];
		
	\end{axis}
\end{tikzpicture}}
		\caption{Angle error as a function of $n$}    
	\end{subfigure}
	\hfill
	\begin{subfigure}[b]{0.45\textwidth}  
		\centerline{\begin{tikzpicture}
	\begin{axis}[
		xlabel={Sample size $n$},
		ylabel={Mean squared error},
		ymode=log,
		xmode=log,
		legend pos=north east,
		legend style={font=\small,fill=none},
		ymajorgrids=true,
		grid style=grid,
		enlargelimits=false,
		]
		
		\addplot[
		color=CadetBlue,
		mark=triangle,
		mark size=2pt,
		line width=1.25pt   
		]
		table[x=delta,y=avg_mse] {data/intro-low-dim.dat};
		\legend{Single imputation}
		
		\addplot[
		color=ForestGreen,
		mark=triangle,
		mark size=2pt,
		line width=1.25pt   
		]
		table[x=delta,y=avg_mse_mi] {data/intro-low-dim.dat};
		\addlegendentry{Complete cases}
		
		\addplot[
		color=RoyalBlue,
		mark=triangle,
		mark size=2pt,
		line width=1.25pt   
		]
		table[x=delta,y=avg_mse_mi_prior] {data/intro-low-dim.dat};
		\addlegendentry{Prior imputation}

		\addplot+[name path=A-mi-prior,RoyalBlue!20, no markers] table[x=delta,y=lower_mse_mi_prior] {data/intro-low-dim.dat};
		\addplot+[name path=B-mi-prior,RoyalBlue!20, no markers] table[x=delta,y=upper_mse_mi_prior] {data/intro-low-dim.dat};
		
		\addplot[RoyalBlue!20] fill between[of=A-mi-prior and B-mi-prior];
		
		\addplot+[name path=A-mi,ForestGreen!20, no markers] table[x=delta,y=lower_mse_mi] {data/intro-low-dim.dat};
		\addplot+[name path=B-mi,ForestGreen!20, no markers] table[x=delta,y=upper_mse_mi] {data/intro-low-dim.dat};
		
		\addplot[ForestGreen!20] fill between[of=A-mi and B-mi];

		\addplot+[name path=A-si,CadetBlue!20, no markers] table[x=delta,y=lower_mse] {data/intro-low-dim.dat};
		\addplot+[name path=B-si,CadetBlue!20, no markers] table[x=delta,y=upper_mse] {data/intro-low-dim.dat};
		
		\addplot[CadetBlue!20] fill between[of=A-si and B-si];
		
	\end{axis}
\end{tikzpicture}}
		\caption{MSE as a function of $n$}
	\end{subfigure}
	\caption{A comparison of several different imputation methods in low dimensions ($p = 2$).  Triangular marks denote the average over $1000$ independent trials and the shaded regions represent the inter-quartile range.  In contrast with the linear model, in which single imputation yields a consistent estimator~\citep{chandrasekher2020imputation}, in the logistic model, single imputation is only able to identify the subspace in which $\bt_0$ lies.}  
	\label{fig:intro-low-dim}
\end{figure*}
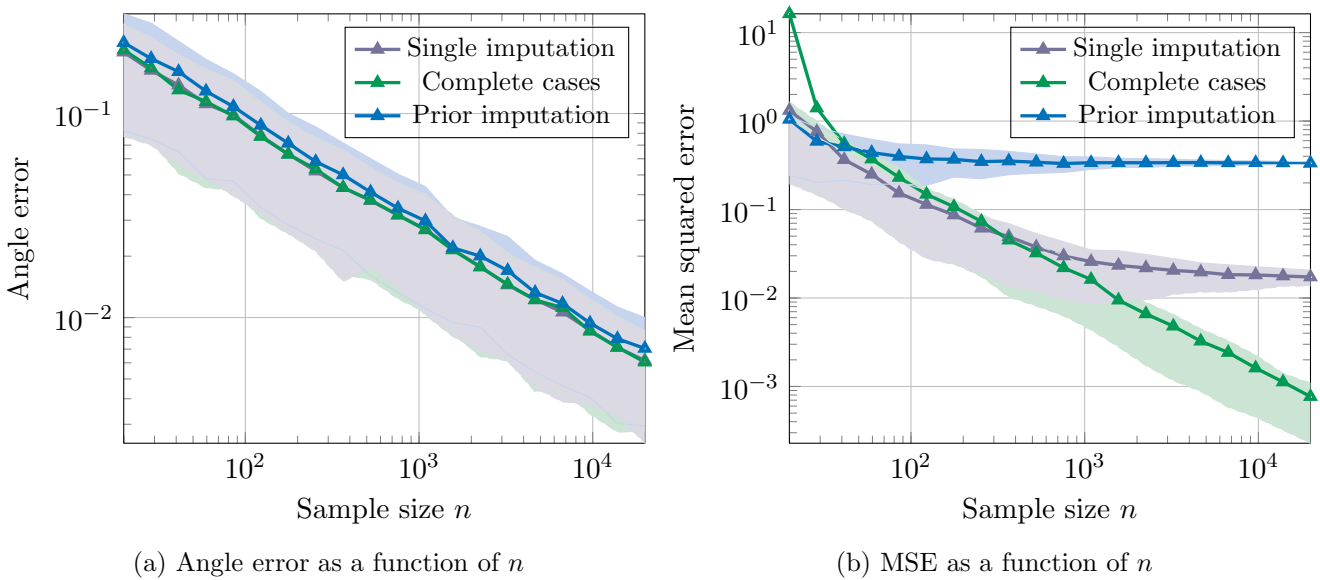

Figure~\ref{fig:intro-low-dim}(a) suggests that regardless of which of the three strategies is used, the logistic regression estimator is consistent in angle error.  By contrast, Figure~\ref{fig:intro-low-dim}(b) suggests that in mean square error, the complete cases estimator is consistent, whereas both imputation estimators are inconsistent.  This distinguishes the behavior in the logistic model from that of the linear model in which single imputation leads to a consistent estimator~\citep[see, e.g.,][]{chandrasekher2020imputation}.  On the other hand, for certain downstream tasks such as prediction, the error is governed by the angle; in such tasks, Figure~\ref{fig:intro-low-dim}(a) suggests that in low dimensions, the three strategies are interchangeable.

It is natural to wonder now whether the situation changes in high dimensions where the number of parameters may be comparable to the number of samples.  To this end, we consider the dimension $p=500$ and set the number of samples as $n=1500$.  We keep the probability of observing an entry as $\alpha=0.85$ and set the ground truth as $\bt_0 = \bm{1}$ (the all ones vector) to ensure that the the ratio $\| \bt_0 \|_2/\sqrt{p}$ remains the same in both experiments.  The data is simulated to be~\ref{eq:MCAR}.  Note that a complete case analysis is now completely infeasible as the probability that a sample contains no missing entries is $\simeq 5 \cdot 10^{-36}$.  Thus, we impute the matrix $\bZ$ using either single imputation or prior imputation.  We then vary the regularization parameter $\lambda$ and perform logistic regression~\eqref{eq:loss} to obtain an estimate~\eqref{eq:estimate}.  We again repeat this independently $1000$ times.  The results are plotted in Figure~\ref{fig:intro-high-dim}.  They are compared with the (conjectured) Bayes optimal errors (see Section~\ref{sec:numerical-illustration} for further detail).
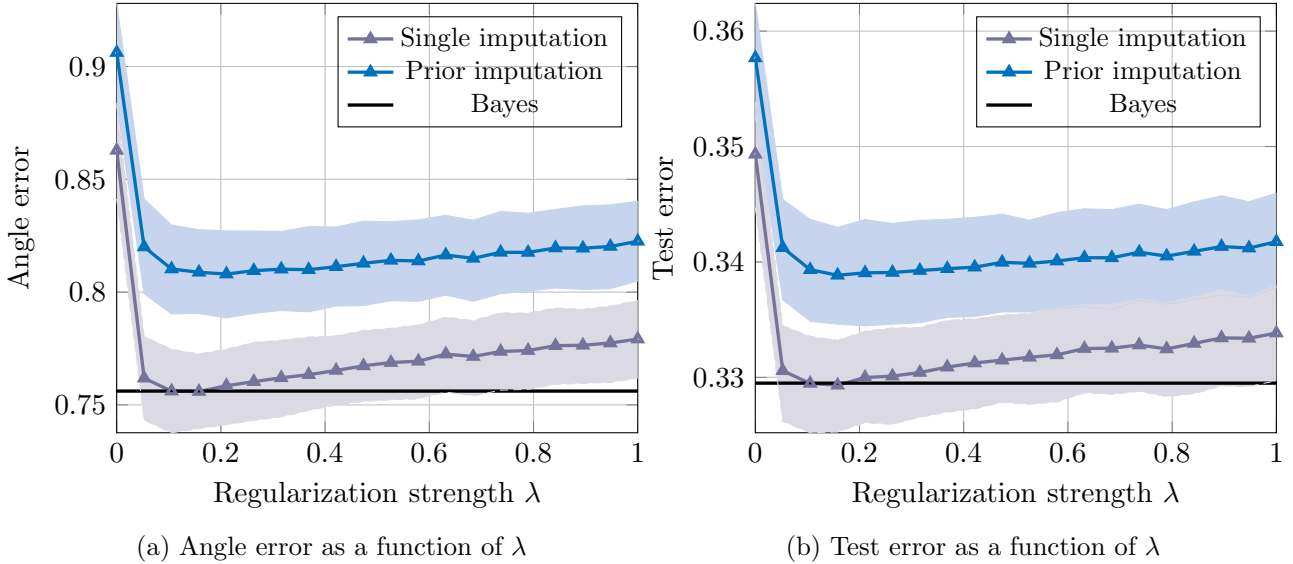
\begin{figure*}[!h]
	\begin{subfigure}[b]{0.47\textwidth}
		\centerline{\begin{tikzpicture}
	\begin{axis}[
		xlabel={Regularization strength $\lambda$},
		ylabel={Angle error},
		xmin=0, xmax=1,
		xtick={0, 0.2, 0.4, 0.6, 0.8,1},
		legend pos=north east,
		legend style={font=\small,fill=none},
		ymajorgrids=true,
		grid style=grid,
		enlargelimits=false,
		]
		
		\addplot[
		color=CadetBlue,
		mark=triangle,
		mark size=2pt,
		line width=1.25pt   
		]
		table[x=lam,y=avg_angle] {data/regularization-data-mi-vs-si-larger-lam.dat};
		\legend{Single imputation}
		
		\addplot[
		color=RoyalBlue,
		mark=triangle,
		mark size=2pt,
		line width=1.25pt   
		]
		table[x=lam,y=avg_angle_mi] {data/regularization-data-mi-vs-si-larger-lam.dat};
		\addlegendentry{Prior imputation}
		
%
		
		\addplot[mark=none, black, line width=1.25pt, samples=2] {0.7561208504075206};
		\addlegendentry{Bayes}

		\addplot+[name path=A-mi,RoyalBlue!20, no markers] table[x=lam,y=lower_angle_mi] {data/regularization-data-mi-vs-si-larger-lam.dat};
		\addplot+[name path=B-mi,RoyalBlue!20, no markers] table[x=lam,y=upper_angle_mi] {data/regularization-data-mi-vs-si-larger-lam.dat};
		
		\addplot[RoyalBlue!20] fill between[of=A-mi and B-mi];

		\addplot+[name path=A-si,CadetBlue!20, no markers] table[x=lam,y=lower_angle] {data/regularization-data-mi-vs-si-larger-lam.dat};
		\addplot+[name path=B-si,CadetBlue!20, no markers] table[x=lam,y=upper_angle] {data/regularization-data-mi-vs-si-larger-lam.dat};
		
		\addplot[CadetBlue!20] fill between[of=A-si and B-si];

		
		%
		
	\end{axis}
\end{tikzpicture}}
		\caption{Angle error as a function of $\lambda$}    
	\end{subfigure}
	\hfill
	\begin{subfigure}[b]{0.47\textwidth}  
		\centerline{\begin{tikzpicture}
	\begin{axis}[
		xlabel={Regularization strength $\lambda$},
		ylabel={Test error},
		xmin=0, xmax=1,
		xtick={0, 0.2, 0.4, 0.6, 0.8,1},
		legend pos=north east,
		legend style={font=\small,fill=none},
		ymajorgrids=true,
		grid style=grid,
		enlargelimits=false,
		]
		
		\addplot[
		color=CadetBlue,
		mark=triangle,
		mark size=2pt,
		line width = 1.25pt   
		]
		table[x=lam,y=avg_gen] {data/regularization-data-mi-vs-si-larger-lam.dat};
		\legend{Single imputation}
		
		\addplot[
		color=RoyalBlue,
		mark=triangle,
		mark size=2pt,
		line width = 1.25pt
		]
		table[x=lam,y=avg_gen_mi] {data/regularization-data-mi-vs-si-larger-lam.dat};
		\addlegendentry{Prior imputation}
		
%
		
		\addplot[mark=none, black, line width=1.25pt, samples=2] {0.32949744648618856};
		\addlegendentry{Bayes}

		\addplot+[name path=A-mi,RoyalBlue!20, no markers] table[x=lam,y=lower_gen_mi] {data/regularization-data-mi-vs-si-larger-lam.dat};
		\addplot+[name path=B-mi,RoyalBlue!20, no markers] table[x=lam,y=upper_gen_mi] {data/regularization-data-mi-vs-si-larger-lam.dat};
		
		\addplot[RoyalBlue!20] fill between[of=A-mi and B-mi];

		\addplot+[name path=A-si,CadetBlue!20, no markers] table[x=lam,y=lower_gen] {data/regularization-data-mi-vs-si-larger-lam.dat};
		\addplot+[name path=B-si,CadetBlue!20, no markers] table[x=lam,y=upper_gen] {data/regularization-data-mi-vs-si-larger-lam.dat};
		
		\addplot[CadetBlue!20] fill between[of=A-si and B-si];

		
		%
		
	\end{axis}
\end{tikzpicture}}
		\caption{Test error as a function of $\lambda$}
	\end{subfigure}
	\caption{A comparison of single imputation and prior imputation in high-dimensions ($p=500, n=1500$).  Both are compared with the conjectured Bayes optimal error (see Section~\ref{sec:numerical-illustration}).  Triangular marks denote the average over $1000$ independent trials and shaded regions represent the inter-quartile range.} 
	\label{fig:intro-high-dim}
\end{figure*}

Figure~\ref{fig:intro-high-dim} reveals a surprising phenomenon.  In low dimensions, the complete case analysis strictly dominates both imputation estimators in mean square error and the three methods perform equivalently with respect to the angle error.  By contrast, in high dimensions, the simple regularized single imputation estimator can achieve nearly the Bayes optimal angle and test errors.  Moreover, there is a separation between the two imputation estimators.  In Section~\ref{sec:main_res}, we will provide exact asymptotic expressions for the errors of both imputation estimators which match the empirical performance observed in Figure~\ref{fig:intro-high-dim}.  

\subsection{Contributions and paper outline}\label{sec:contributions}
We next summarize our results.
\begin{description}
	\item[Characterization of the risk in the error-in-variables model.] We obtain an asymptotically exact characterization of the risk of ridge-regularized logistic regression when the observed covariates stem from a Gaussian error-in-variables model.  Moreover, we complement this characterization with concentration inequalities which provide non-asymptotic bounds on the fluctuations around the asymptotic risk.  See Section~\ref{subsec:error-in-var} for precise statements.
	\item[Universality.]  We prove that for all pairs of matrices $\bX$ and $\bZ$ belonging to a certain universality class, the characterization of the risk in the error-in-variables model continues to hold.  This yields an asymptotic characterization of the risk for ridge-regularized logistic regression when covariates are missing completely at random and the imputation strategy follows either single imputation or prior imputation.  As a special case, our result implies universality of the risk of ridge-regularized logistic regression when the data is perfectly observed~\citep[cf.][]{sur2019modern,salehi2019impact}.  See Section~\ref{subsec:missing-data-main} for precise statements.
	\item[The effect of regularization.] We evaluate our predicted formulas and compare them with the conjectured Bayes optimal risk, defined precisely in Section~\ref{subsec:bayes}.  See Section~\ref{sec:numerical-illustration} for further details.
\end{description}

To expand on the last point, our results reveal a crucial role played by regularization, even in a moderate data regime when $n/p\gtrsim 10$:
\begin{itemize}
	\item If no regularization is employed, $\lambda=0$, then---in agreement with classical arguments---conditional mean imputation is underestimating the variability in the covariates. Consequently, the estimate of $\bt_0$ is overconfident (see Figure~\ref{fig:empirical-single-vs-prior}).
	\item If regularization is added, with optimally tuned $\lambda>0$, this problem is not only alleviated but nearly entirely eliminated:
	the resulting estimate is nearly as accurate as the (conjectured) Bayes-optimal estimate (see Figure~\ref{fig:fixalpha}).
	\item Prior imputation appears to alleviate the overconfidence problem as well (see Figure~\ref{fig:empirical-single-vs-prior}). One way to think about this effect is that the randomly drawn entries are effectively adding noise to the conditional-mean covariates matrix. Covariate noise is known to  have similar consequences as ridge   regularization~\citep{bishop1995training}.
	\item Finally, there is a separation in performance between different imputation strategies.  In particular, even though prior imputation utilizes more knowledge about the covariates' distribution than single imputation, it can lead to inferior prediction error.
\end{itemize}

The remainder of the paper is organized as follows.  In Section~\ref{sec:main_res}, we present our main results:  Section~\ref{subsec:error-in-var} presents sharp results for the Gaussian error-in-variables model and Section~\ref{subsec:missing-data-main} extends these results to a much larger universality class which contains both~\ref{eq:prior_imp} and~\ref{eq:single_imp}.  In Section~\ref{sec:numerical-illustration}, we provide extensive numerical illustrations as well as a conjectured characterization of the Bayes error.  In Section~\ref{sec:proofs}, we provide the proofs of our main results.  Finally, we provide discussion in Section~\ref{sec:discussion}.  Our appendices contain omitted proofs as well as additional numerical evidence.  

\subsection{Related work}

\paragraph{Missing data:} Regression with missing data has been studied for decades and a comprehensive review can be found in the book of~\citet{little2014statistical}.  In the low-dimensional setting,~\citet{wang1998large} provide asymptotic guarantees for multiple imputation.  Most relevant to our treatment are the specialized, high-dimensional methods designed for sparse linear regression.  A subset of these---all of which provide high dimensional consistency results---include~\citep{rosenbaum2010sparse, loh2012high, chen2013noisy, datta2017cocolasso, wang2017rate}.  Recently,~\citet{chandrasekher2020imputation} studied single-imputation for high-dimensional sparse linear regression and obtained optimal consistency rates of both the LASSO and the square-root LASSO without modification.  As the simulation study in Section~\ref{sec:motivating-example} suggests, single-imputation based estimators exhibit genuinely distinct behavior in the linear model and the logistic model we study here.  Moving beyond regression, several other models have been studied in higher dimension including PCA~\citep{zhu2019high,yan2021inference}, covariance estimation~\citep{lounici2014high}, changepoint detection~\citep{xie2012change,follain2021high}, and nonparametric classification~\citep{sell2024nonparametric}, to name a few.  However (to the best of our knowledge), existing theoretical treatments within the missing data literature have not covered estimation in generalized linear models.

More broadly, there has been a flurry of work on imputation methodology in recent years~\citep[see, e.g.,][]{zhao2020missing,you2020handling,bertsimas2018predictive}.  
As mentioned in Section~\ref{sec:missing-setup}, the~\ref{eq:MCAR} assumption is simplistic in nature and moving beyond it forms an important theoretical problem.  We refer the interested reader to two interesting papers in this direction.  First,~\citet{agarwal2021causal} focus on imputation methodology for low rank data and develop matrix completion--based imputation strategies, establishing guarantees under very general missingness mechanisms (including data which is missing not at random).  In a distinct direction,~\citet{berrett2022optimal} consider the problem of verifying the MCAR assumption and develop a test to determine whether data is missing at random or not.

\paragraph{Exact asymptotics with Gaussian data:} A substantial literature characterizes the asymptotic properties of
high-dimensional M-estimators in the proportional asymptotics in which both the number of parameters $p$ and the number of samples $n$
diverge. A subset of relevant papers include~\citep{bayati2011lasso,amelunxen2013,donoho2016high,el2018impact,reeves2016replica,thrampoulidis2015regularized,miolane2021distribution,sur2019modern,salehi2019impact}.  Our proofs rely on the CGMT (convex Gaussian min-max theorem), a tight application of Gordon's minimax theorem for Gaussian processes.  Gordon's original theorem~\citep{gordon1985some,Gordon1988} is a
Gaussian comparison inequality for the minimization-maximization of
two related Gaussian processes.
In a line of work initiated by~\citet{stojnic2013framework} and formalized by~\citet{thrampoulidis2015regularized}, the comparison inequality was shown to be tight when the underlying Gaussian process is convex-concave.
This observation has led to several works establishing exact asymptotics for high-dimensional convex procedures, 
including general penalized M-estimators in linear regression~\citep{thrampoulidis2015regularized,thrampoulidis2018} and binary classification~\citep{dengmodel,Montanari2019,Liang2020APH}.  Moving beyond studying specific procedures,~\citet{barbier2019optimal} provide an asymptotic characterization of the Bayes error in the proportional, asymptotic regime under a Gaussian data assumption (see Section~\ref{sec:numerical-illustration} for further discussion).   Utilizing the aforementioned characterization of the Bayes error in conjunction with the CGMT,~\citet{aubin2020generalization} consider using the logistic loss as a surrogate risk when the true labels were generated according to the perceptron and---similarly to what we show in Section~\ref{sec:numerical-illustration}---demonstrate that this procedure nearly achieves the Bayes optimal error (in a setting where the entries are completely observed).

\paragraph{Universality:}
Our proofs of universality rely on the Lindeberg principle~\citep{lindeberg1922neue}, which was formalized by~\citet{chatterjee2006generalization} and has proven extremely successful in deriving universality properties. It allows to prove universality for expectations of functions of independent random variables, as long as the functions are sufficiently smooth
(typically a bound on the third derivative is required).
Implementing this approach requires approximating the object of interest with such an expectation.  The works~\citep{korada2011applications, montanari2017universality} developed a technique to leverage the Lindeberg principle in the context of linear regression.  More recently,~\citet{han2022universality} leveraged the Lindeberg principle in conjunction with the CGMT to establish exact asymptotics and universality for a set of regularized regression estimates in the high dimensional linear model.  Related universality
results in high-dimensional statistics were proven in~\citet{bayati2015universality,oymak2018universality}.  Recently,~\citet{hu2022universality,montanari2022empirical} moved beyond the independent entries assumption and proved universality for empirical risk minimization.  We emphasize that our proofs follow the well established strategy developed and employed by the sequence of previous work listed above; our treatment departs from this line of work as we obtain universality when the data matrix $\bX$ used to generate the responses as well as the data matrix $\bZ$ used for inference can be different.  

\subsection{Notation}
We use bold-face lower-case letters to denote vectors $(\bw, \bv, \dots)$ and bold-face upper-case letters to denote matrices $(\bX, \bZ, \dots)$.  
We will make use of the Orlicz norm of a random variable \sloppy\mbox{$
\norm{X}_{\psi} := \inf\bigl\{t > 0: \E \psi(\lvert X \rvert/t) \leq 1\bigr\}$},
where $\psi:\mathbb{R}_{\geq 0} \rightarrow \mathbb{R}_{\geq 0}$ is a convex and strictly increasing function such that $\psi(0) = 0$.  We will make particular use of the sub-Gaussian norm, taking $\psi_2: u \mapsto \exp(u^2) - 1$ and the sub-exponential norm, taking $\psi_1: u \mapsto \exp(\lvert u \rvert) - 1$.  Throughout the paper, $c, C$ denote constants that may change line to line.  We will additionally use the asymptotic notation $a_n \gtrsim b_n$ to denote $a_n \geq Cb_n$, $a_n \lesssim b_n$ to denote $a_n \leq Cb_n$ and $a_n \asymp b_n$ to denote $a_n \lesssim b_n$ and $a_n \gtrsim b_n$, for sequences $\{a_n\}, \{b_n\}$.  
For a convex function $f$ and scalar $\gamma \geq 0$, we will denote the Moreau envelope of $f$ by $
M_f(x; \gamma) := \min_{u \in \mathbb{R}} \left\{f(u) + \frac{\gamma}{2}(u - x)^2\right\}$,
and denote the proximal operator by $
\prox_f(x; \gamma) := \argmin_{u \in \mathbb{R}}\left\{f(u) + \frac{\gamma}{2}(u - x)^2\right\}$.
For a linear subspace $S$ of $\mathbb{R}^p$ we will denote by $\proj_{S}$ the projection onto $S$ and by $\proj_{S}^{\perp}$ the projection onto the subspace orthogonal to $S$.  Finally, for a sequence of random variables $X_n$,  we will denote by $X_n \overset{P}{\rightarrow} X$ convergence in probability, in the limit where $n/p \rightarrow \delta$ and $n,p \rightarrow \infty$.

\section{Main results}
\label{sec:main_res}
We now describe our main results.  We begin by considering the Gaussian error-in-variables ensemble in Section~\ref{subsec:error-in-var}.  We then prove our main universality result in Section~\ref{subsec:missing-data-main} and detail several consequences for imputation based methods. 

\subsection{Logistic regression with error-in-variables}
\label{subsec:error-in-var}
We begin by defining the Gaussian error-in-variables ensemble, which was informally introduced in Section~\ref{sec:missing-setup}.
\begin{definition}[Gaussian error-in-variables]
	\label{def:gaussian-eiv}
	Let $\bX$ and $\bW$ be independent random matrices with entries $(X_{ij})_{i \leq n, j \leq p} \overset{\mathsf{i.i.d.}}{\sim} \mathsf{N}(0, 1/p)$ and $(W_{ij})_{i \leq n, j \leq p} \overset{\mathsf{i.i.d.}}{\sim} \mathsf{N}(0, 1/p)$.  Then, for positive constants $\alpha_c$ and $\alpha_2$ such that $\alpha_c \leq \sqrt{\alpha_2}$, the Gaussian error-in-variables matrix is defined as
	\[
	\bZ = \alpha_c \bX+ \sqrt{\alpha_2 - \alpha_c^2} \bW.
	\]
\end{definition}
As mentioned in the introduction, we use this model primarily as a theoretical tool to enable the study of missing data models.  We next require a regularity assumption governing the regularization parameter $\lambda$, the ratio of samples to dimensions $\delta = n/p$ and $R$, the norm of the re-scaled ground truth coefficients $\bt_0/\sqrt{p}$. 
\begin{assumption}[Parameter regularity]
	\label{asm:regularity}
	The regularization strength $\lambda$, ratio $\delta = n/p$, radius $R$, and covariance parameters $\alpha_c, \alpha_2$ are bounded below by an absolute, positive constant $K_1$ and above by an absolute, positive constant $K_2$.  Moreover, the ground truth coefficients $\bt_0 \in \mathbb{R}^p$ satisfy $\| \bt_0 \|_2/\sqrt{p} = R$.
\end{assumption}

We next define an asymptotic loss, which captures the asymptotic behavior of the the ridge-regularized loss $\mathcal{L}_n$~\eqref{eq:loss}.
\begin{definition}[Asymptotic loss]
	\label{def:asymploss}
	Consider problem parameters which satisfy Assumption~\ref{asm:regularity} and let $(Z_1, Z_2, G)$ denote a triple of i.i.d. standard Gaussian random variables.  Define the random variable $Y$, whose conditional distribution given $Z_1$ is 
	\begin{align}
		\label{def:Y_dist}
		Y \mid Z_1 = \begin{cases} +1 \qquad \text{ with probability } \qquad  \E_{G}\Bigl\{\rho'\Bigl(\frac{\alpha_c}{\sqrt{\alpha_2}}RZ_1 + \sqrt{1 - \frac{\alpha_c^2}{\alpha_2}}RG\Bigr)\Bigr\}\\
			-1 \qquad\text{ else}.
		\end{cases}
	\end{align}
	Additionally, for a pair of scalars $(\sigma, \xi) \in \mathbb{R}^2$, define the random variable $V(Z_1, Z_2)$ as
	\[
	V(Z_1, Z_2) := \xi R \sqrt{\alpha_2} Z_1 + \sigma \sqrt{\alpha_2} Z_2.
	\]
	The \emph{asymptotic loss} $L: \mathbb{R}_{\geq 0} \times \mathbb{R} \times \mathbb{R}_{\geq 0} \rightarrow \mathbb{R}$ is defined as 
	\begin{align}
		\label{eq:asymploss}
		L(\sigma, \xi, \gamma) = \frac{\lambda (\sigma^2 + \xi^2 R^2)}{2} - \frac{\alpha_2 \gamma \sigma^2}{2\delta} + \E\Bigl\{\min_{u \in \mathbb{R}}\Bigl[ \rho(-Yu) + \frac{\gamma}{2} \cdot \bigl(u - V(Z_1, Z_2)\bigr)^2\Bigr]\Bigr\}.
	\end{align}
\end{definition}

With these preliminaries in hand, we now collect several useful properties of the asymptotic loss in the next lemma, which will allow us to state the main result of this section.  We provide the proof of this lemma in Appendix~\ref{sec:properties-asymptotic}.  
\begin{lemma} 
	\label{lem:structural_L}
	Under Assumption~\ref{asm:regularity}, the asymptotic loss $L$~\eqref{eq:asymploss} satisfies the following properties.  
	\begin{enumerate} 
		\item[(a)] The map $\Psi: (\sigma, \xi) \mapsto \max_{\gamma \geq 0}\; L(\sigma, \xi, \gamma)$ is $\lambda \cdot (1 \wedge R^2)$--strongly convex on the domain $[0, \infty) \times \mathbb{R}$.
		\item[(b)] There exists a positive constant $\gamma_0$, depending only on $K_1, K_2$ such that
		\[
		\min_{\sigma \geq 0, \xi \in \mathbb{R}}\max_{\gamma \geq 0}\; L(\sigma, \xi, \gamma) = \min_{\sigma \geq 0, \xi \in \mathbb{R}}\max_{\gamma \geq \gamma_0}\; L(\sigma, \xi, \gamma)
		\]
		\item[(c)] There exists a unique triplet $(\sigma_{\star}, \xi_{\star}, \gamma_{\star})$ such that, for all $\sigma \in \mathbb{R}_{\geq 0}, \xi \in \mathbb{R}, \gamma \in \mathbb{R}_{\geq 0}$,
		\[
		L(\sigma, \xi, \gamma_{\star}) \leq L(\sigma_{\star}, \xi_{\star}, \gamma_{\star}) \leq L(\sigma_{\star}, \xi_{\star}, \gamma).
		\]
		Moreover, $(\sigma_{\star}, \xi_{\star}, \gamma_{\star})$ is identified as the unique solution to the following system of equations
		\begin{align}
			\label{eq:system}
			\sigma^2 &= \frac{\delta}{\alpha_2} \E\left\{(1 + Y)\cdot\Bigl[\prox_{\rho}\bigl(V(Z_1, Z_2); \gamma\bigr) - V(Z_1, Z_2)\Bigr]^2 \right\}\nonumber\\
			0 &= -\frac{\gamma \sigma \alpha_2}{\delta} + \sigma \lambda + \sqrt{\alpha_2} \cdot \E\Bigl\{(1 + Y)\cdot Z_2 \cdot \rho'\Bigl(\prox_{\rho}\bigl(V(Z_1, Z_2); \gamma\bigr)\Bigr)\Bigr\}\\
			0 &= \xi \lambda R^2 + R\sqrt{\alpha_2}\cdot \E\Bigl\{(1 + Y) \cdot Z_1 \cdot \rho'\Bigl(\prox_{\rho}\bigl(V(Z_1, Z_2); \gamma\bigr)\Bigr)\Bigr\}.\nonumber
		\end{align}
	\end{enumerate}
\end{lemma}

We emphasize that the triplet $(\sigma_{\star}, \xi_{\star}, \gamma_{\star})$ depends on the problem parameters $\lambda$, $\delta$, $\alpha_c$, $\alpha_2$, and $R$.  Taking $\alpha_c = \alpha_2 = 1$, we recover the system of equations derived by~\citet[Equation 16]{salehi2019impact}, which were derived in the context of high dimensional, ridge-regularized logistic regression with fully observed data.  We also note that setting the regularization strength $\lambda = 0$ (and keeping the setting $\alpha_c = \alpha_2 = 1$) recovers the system of equations derived by~\citet[Equation 5]{sur2019modern}.  

Before stating the main proposition, we define the following shorthand: For any vector $\bt \in \mathbb{R}^p$, we write its parallel and orthogonal components as
\begin{align}
	\label{def:shorthand-xi-sigma}
	\xi(\bt) = \frac{\< \bt, \bt_0 \>}{R^2 p} \qquad \text{ and } \qquad \sigma(\bt) = \frac{1}{\sqrt{p}} \| \pproj \bt \|_2.
\end{align}
We turn now to the main result on the error-in-variables model, whose proof we provide in Section~\ref{sec:proofs}.
\begin{proposition}
	\label{prop:sharp}
	Under Assumption~\ref{asm:regularity}, let the random matrices $\bX, \bG \in \mathbb{R}^{n \times p}$ belong to the $\eivdef$ ensemble with $(X_{ij})_{1 \leq i \leq n, 1 \leq j \leq p} \overset{\mathsf{i.i.d.}}{\sim} \mathsf{N}(0, 1/p)$, and assume that the labels $\by$ are generated from the data matrix $\bX$ and the ground truth $\bt_0$ according to the logistic model~\eqref{eq:data_gen}.  There exists a tuple of positive constants $(c_0, c, C)$, depending only on $K_1, K_2$ such that the following hold.
	\begin{itemize}
		\item[(a)] For every $0 < \epsilon \leq c_0$, the estimator $\widehat{\bt}(\bG; \lambda)$~\eqref{eq:estimate} satisfies
		\begin{align}
			\label{ineq:prop_gordon_phi}
			\pr\Bigl\{ \bigl \lvert \sigma\bigl(\widehat{\bt}(\bG; \lambda)\bigr) - \sigma_{\star} \bigr \rvert \vee \bigl \lvert \xi\bigl(\widehat{\bt}(\bG; \lambda)\bigr)- \xi_{\star} \bigr \rvert \geq \epsilon \Bigr\} \leq \frac{C}{\epsilon^{6}}\exp\Bigl\{-c\min(n\epsilon^{4}, n\epsilon^2)\Bigr\}.
		\end{align}
	\item[(b)] For every $\epsilon > 0$, the loss $\mathcal{L}_n(\cdot; \bG, \lambda)$~\eqref{eq:loss} satisfies
	\begin{align}
		\label{ineq:prop_gordon_loss}
		\pr\Bigl\{\bigl\lvert \min_{\bt \in \mathbb{R}^p} \mathcal{L}_n(\bt; \bG, \lambda) - L(\sigma_{\star}, \xi_{\star}, \gamma_{\star}) \bigr \rvert \geq \epsilon\Bigr\} \leq \frac{C}{\epsilon^{3}}\exp\Bigl\{-c\min(n\epsilon^2, n \epsilon)\Bigr\}.
	\end{align}
	\end{itemize}
%
%
%
%
%
%
%
%
%
\end{proposition}
Note that part (a) shows that both the `orthogonal' component $\sigma\bigl(\widehat{\bt}(\bG, \lambda)\bigr)$ as well as the `signal' component $\xi\bigl(\widehat{\bt}(\bG, \lambda)\bigr)$ deviate from the quantities $\sigma_{\star}$ and $\xi_{\star}$, respectively, with fluctuations on the order $\widetilde{\mathcal{O}}(n^{-1/4})$.  On the other hand, part (b) implies that the \emph{minimum} of the loss function $\mathcal{L}_n$ deviates from the minimum of the asymptotic loss $L$ with fluctuations on the order $\widetilde{\mathcal{O}}(n^{-1/2})$.  

Assuming the regularization strength $\lambda > 0$, the proposition improves upon~\citet[][Theorem 2]{sur2019modern} (who consider the unregularized case) and~\citet[][Theorem 2]{salehi2019impact} in two directions.  First, it provides guarantees in the situation when the data matrix $\bX$ used to generate the labels and the data matrix $\bZ$ used for estimation are different.  Second, it provides a non-asymptotic characterization of the error---such a characterization is necessary in order to provide a quantitative universality statement in the sequel.  We note that two recent papers~\citep{chandrasekher2021sharp,loureiro2021capturing} provide similar non-asymptotic guarantees for generalized linear models, although neither considers a mismatch in the data matrices used to generate the labels and to perform estimation.  Similarly to these works, our proof leverages the CGMT (convex Gaussian min-max theorem)~\citep{thrampoulidis2015regularized} and employs a strategy developed by~\citet{miolane2021distribution} to obtain non-asymptotic control.  

Having established guarantees for the error-in-variables model, we next describe our main results for models with missing data.

\subsection{Universality of the logistic regression error}
\label{subsec:missing-data-main}
In this section, we provide sharp performance guarantees of the logistic regression estimator under a significantly larger set of data matrices.  The central structure underlying this phenomenon is the $(\alpha_c, \alpha_2)$--universality class, defined presently. 
\begin{definition}[$(\alpha_c, \alpha_2)$--universality class]
	\label{def:universality} Let $\alpha_c, \alpha_2$ be positive scalars which satisfy the inequality $\alpha_c \leq \sqrt{\alpha_2}$.  Consider random matrices $\bX \in \mathbb{R}^{n \times p}$ and $\bZ\in \mathbb{R}^{n \times p}$.  We say that the pair of random matrices $(\bX, \bZ)$ belongs to the $(\alpha_c, \alpha_2)$--universality class if the pairs of random variables $\{(X_{ij}, Z_{ij})\}_{i \leq n, j \leq p}$ are mutually independent and further satisfy the following:
	\begin{align*}
		(i)\;\; \E\big\{ X_{ij} \bigr\} = \E\bigl\{Z_{ij} \bigr\} = 0, \quad (ii)\;\; \E\bigl\{Z_{ij}^2\bigr\} = \frac{\alpha_2}{p}, \quad (iii)\;\; \E\bigl\{X_{ij} Z_{ij}\bigr\} = \frac{\alpha_c}{p},\\
		(iv)\;\; \E\bigl\{X_{ij}^2\bigr\} = \frac{1}{p},\quad \text{ and } \quad (v)\;\; \max\Bigl( \norm{X_{ij}}_{\psi_2}, \norm{Z_{ij}}_{\psi_2}\Bigr) \leq \frac{K_3}{\sqrt{p}},
	\end{align*}
	where $K_3$ is a constant which may depend on $\alpha_c, \alpha_2$.
\end{definition}
Our results hold on this class of data matrices under one more regularity assumption on the ground truth coefficients $\bt_0$, which ensures that the true coefficients are not too concentrated in a small set of coordinates.  

\begin{assumption}[Spread]
	\label{asm:groundtruth}
	For positive parameters $K_4$ and $\tau < 1/6$, we have
	\[
	\| \bt_0 \|_{\infty} \leq K_4 \cdot n^{1/6 - \tau}.
	\]
\end{assumption}

Taken together, Assumptions~\ref{asm:regularity} and~\ref{asm:groundtruth} imply that our guarantees hold provided for parameters $\bt_0$ contained in the set $\Theta_{R, \tau, K} \subseteq \mathbb{R}^p$, defined as
\begin{align}\label{eq:parameter-set}
\bt_0 \in \Theta_{R, \tau, K} := \Bigl\{ \bt \in \mathbb{R}^{p}: \| \bt \|_2 = R \sqrt{p} \quad \text{ and } \quad \| \bt \|_{\infty} \leq K n^{1/6 - \tau}\Bigr\}.
\end{align}
We defer further commentary on Assumption~\ref{asm:groundtruth} until after the statement of our main theorem.  We are now poised to state our main theorem, whose proof we provide in Section~\ref{sec:proofs}.
\begin{theorem}
	\label{thm:main} 
	Under Assumptions~\ref{asm:regularity} and~\ref{asm:groundtruth}, let the pair of random matrices $\bX, \bZ \in \mathbb{R}^{n \times p}$ belong to the $\univclass$, and assume that the labels $\by$ are generated from the data matrix $\bX$ and the ground truth $\bt_0$ according to the logistic model~\eqref{eq:data_gen}.  The estimator $\widehat{\bt}(\bZ, \lambda)$~\eqref{eq:estimate} satisfies
	\[
	\sup_{\lambda \in [K_1, K_2]} \Bigl\{ \bigl \lvert \sigma\bigl(\widehat{\bt}(\bZ, \lambda)\bigr) - \sigma_{\star} \bigr \rvert \vee \bigl \lvert \xi\bigl(\widehat{\bt}(\bZ, \lambda)\bigr)- \xi_{\star} \bigr \rvert \Bigr\} \overset{P}{\rightarrow} 0.
	\]
%
\end{theorem}

In words, this theorem establishes an asymptotic equivalence of the error of the logistic regression estimator when performed with any pair of data matrices from the $(\alpha_c, \alpha_2)$--universality class.  This error is measured by the maximum deviation of the `parallel' and `orthogonal' components from their asymptotic counterparts.  As an immediate corollary, note that for any continuous function $\phi: \mathbb{R}^2 \rightarrow \mathbb{R}$, the following holds
\[
\sup_{\lambda \in [K_1, K_2]}\; \Bigl \lvert \phi\Bigl(\sigma\bigl(\widehat{\bt}(\lambda)\bigr), \xi\bigl(\widehat{\bt}(\lambda)\bigr)\Bigr) - \phi\bigl(\sigma_{\star}(\lambda), \xi_{\star}(\lambda)\bigr)\Bigr \rvert \overset{P}{\rightarrow} 0.
\] 

We turn now to consequences for imputation-based methods.  Suppose that the data matrix $X$ was used to generate the labels and that the mechanism by which data is missing is~\ref{eq:MCAR}.  If the imputed matrix $\bZ^{\mathsf{si}}$ is formed according to the~\ref{eq:single_imp} strategy, then the pair $(\bX, \bZ^{\mathsf{si}})$ belongs to the $(\alpha, \alpha)$--universality class. On the other hand, if the imputed matrix $\bZ^{\mathsf{pi}}$ is formed according to the~\ref{eq:prior_imp} strategy, then the pair $(\bX, \bZ^{\mathsf{pi}})$ belongs to the $(\alpha, 1)$--universality class.  As a concrete application of the formulas developed in Theorem~\ref{thm:main}, Figure~\ref{fig:empirical-single-vs-prior} (see Section~\ref{sec:numerical-illustration}) demonstrates a particular setting of $(\alpha, R, \delta)$ in which the single imputation-based estimator strictly outperforms the prior imputation-based estimator. 

Some remarks on specific aspects of the theorem are in order.  First, note that the $(\alpha_c, \alpha_2)$--universality class on which this theorem holds is defined with respect to a \emph{pair} of random matrices.  This is required as the data matrix used to generate the labels will be different from that used for inference.  By contrast, universality classes are typically defined with respect to a single random matrix and require only parts \emph{(i)}, \emph{(ii)}, and \emph{(v)} of Definition~\ref{def:universality}~\citep[see, e.g.][]{chatterjee2006generalization,tao2011random}.  Our definition generalizes these notions.  Indeed, specifying $\alpha_2 = \alpha_c = 1$, an immediate corollary of Theorem~\ref{thm:main} is universality for the error of the ridge-regularized logistic regression estimator.  

Second, we note that an assumption such as Assumption~\ref{asm:groundtruth}---which bounds the largest coordinate of the ground-truth $\bt_0$---is necessary.  To see this, set $\alpha_2 = \alpha_c = 1$ and consider the matrices $\bX$, consisting of i.i.d. Rademacher entries, and $\bZ$, consisting of i.i.d. Gaussian entries.  Further, let $\bt_0 = R \sqrt{p} \cdot \be_1$, where $\be_1$ is the first standard basis vector.  It can be seen in this scenario that the limits of $\sigma\bigl(\widehat{\bt}(\bX, \lambda)\bigr)$ and $\sigma\bigl(\widehat{\bt}(\bZ, \lambda)\bigr)$ do not coincide (and similarly for the respective quantities $\xi(\cdot)$)\footnote{See~\citet[\S 4]{montanari2022empirical} for related discussion.}.  While we have not attempted to obtain the sharpest possible scaling of the maximum coordinate,  Assumption~\ref{asm:groundtruth} suffices for many ground-truth vectors of interest.  For instance, Assumption~\ref{asm:groundtruth} is satisfied when the coordinates are i.i.d. from a light-tailed distribution (e.g. sub-Gaussian or sub-exponential) or even from a heavy-tailed distribution (e.g. Pareto with shape parameter $\alpha < 6$).  

Finally, our proof relies on the Lindeberg principle~\citep{lindeberg1922neue,chatterjee2006generalization}.  Our task differs from these prototypical applications as the quantity of interest is defined only implicitly as the minimizer of a convex function.  To overcome this obstacle, we employ a strategy developed by~\citet{montanari2017universality} to pass from the study of the minimizer of a convex function to the minimum of a convex function.  In turn, we approximate the minimum with an exponential smoothing and apply the Lindeberg principle to the smoothed minimum.  Carefully handling the approximation errors yields the result.  The proof is provided in detail in Section~\ref{sec:proofs}.  

\section{Numerical illustrations}\label{sec:numerical-illustration}
This section is organized as follows.  First, in Section~\ref{subsec:bayes}, we discuss Bayes estimation, providing a conjecture for the Bayes lower bound.  Then, in Section~\ref{subsec:bayes-numerical}, we provide a detailed numerical study comparing the Bayes prediction error lower bound with the characterization provided by Theorem~\ref{thm:main} for single imputation.  Finally, in Section~\ref{subsec:effect-of-reg}, we focus on a particular parameter setting and investigate the effect of regularization.

\subsection{Bayes estimation in generalized linear models}\label{subsec:bayes}
We begin by drawing a connection between our model of interest---in which the ground truth $\bt_0$ is deterministic---and the Bayesian setting in which a prior on the ground truth is assumed.  To this end, we note that in the Gaussian error-in-variables model, the rotational invariance of the Gaussian distribution implies that the error of any equivariant procedure is the same for any $\bt_0$ with the same norm $\| \bt_{0} \|_2 = R\sqrt{p}$.  Consequently, this error is the same for a random $\bt_{0}$ drawn uniformly on the sphere in dimension $p$ of radius $R\sqrt{p}$ and by the Hunt--Stein theorem~\citep[Theorem 9.2]{lehmann2006theory}, the optimal error for this prior provides a lower bound on the minimax error of \emph{any} (not necessarily equivariant) procedure.  In turn, since as the dimension $p$ grows, the uniform prior is well approximated by a (scaled) standard Gaussian prior, we compute the minimax risk as the Bayes risk with a Gaussian prior.  We refer the interested reader to~\citet[\S 3]{dicker2016ridge} for related discussion in the linear model.

With this connection in hand, we next recall known results on Bayes optimal procedures in high-dimensional generalized linear models before specializing to the Gaussian error-in-variables model in Section~\ref{sec:bayes-eiv}, where we additionally make a conjecture for the missing data model considered here.  

\subsubsection{Bayes estimation and the replica symmetric potential}\label{sec:bayes}
We now recall some known results concerning Bayes optimality in high-dimensional generalized linear models.  Our starting point is~\citet{barbier2019optimal}, whose results we specialize to our setting.  The authors consider the situation in which the ground truth $\bt_0$ consists of i.i.d. coordinates drawn from a distribution $P_{\theta}$ such that $\EE_{P_{\theta}} \Theta^2 = \varrho^2$.
The labels are then generated from the ground truth $\bt_0$ and a data matrix $\bX$ consisting of i.i.d. entries (not necessarily Gaussian) according to the conditional probability mass function $P_{Y}(\cdot \mid \langle \bx, \bt_0 \rangle)$.  With these in hand, the authors define the \emph{replica symmetric potential} $f_{RS}: \mathbb{R}^2 \rightarrow \mathbb{R}$ as 
\begin{subequations}
	\begin{align}\label{eq:RS-potential}
		f_{RS}(q, r; \varrho) = \psi(r) + \delta \cdot \Psi(q; \varrho)  - \frac{rq}{2},
	\end{align}
	where we recall $\delta = n/p$ and the functions $\psi: \mathbb{R} \rightarrow \mathbb{R}$ and $\Psi: \mathbb{R} \rightarrow \mathbb{R}$ are defined as
	\begin{align}
		\psi(r) &= \E_{\Theta, G} \log{ \E_{\Theta_1}\Bigl\{ \exp\Bigl(r \cdot \Theta \Theta_1 + \sqrt{r} \cdot \Theta_1 G - \frac{r \Theta_1^2}{2} \Bigr)\Bigr\}} \qquad \text{ and }\\
		\Psi(q; \varrho) &= \E_{V, W, \widetilde{Y}} \log{ \E_{W_1}\Bigl\{P_Y\bigl(\widetilde{Y} \mid \sqrt{q} \cdot V + \sqrt{\varrho^2 - q} \cdot W_1\bigr)\Bigr\} }, \label{eq:Psi-RS}
	\end{align}
\end{subequations}
where the tuple of random variables $(G, V, W, W_1)$ are i.i.d. standard Gaussian, the pair of random variables $(\Theta, \Theta_1)$ are i.i.d. draws from the distribution $P_{\theta}$ and the random variable $\widetilde{Y}$ is distributed as \sloppy \mbox{$\widetilde{Y} \sim P_Y(\cdot \mid \sqrt{q} \cdot V + \sqrt{\varrho^2- q} \cdot W)$}.  
Equipped with these preliminary notions,~\citet[Proposition 1]{barbier2019optimal} show that the variational problem 
\begin{align} \label{eq:variational-RS}
	\inf_{q \in [0, \varrho^2]}\sup_{r \geq 0}\; f_{RS}(q, r; \varrho),
\end{align}
admits a unique minimizer $q_{\star}$---denoted the optimal overlap.  This value derives its name from the fact that if $\bt$ denotes a sample from the posterior distribution, the quantity \sloppy\mbox{$\lvert \frac{1}{p}\langle \bt_0, \bt \rangle - q_{\star} \rvert \overset{P}{\rightarrow} 0$}~\citep[see, e.g.,][Theorem 4]{barbier2019optimal}.  In turn,~\citet[Theorem 4]{barbier2019optimal} implies that the Bayes prediction error can be computed from the overlap $q_{\star}$.  

\subsubsection{Error-in-variables model}\label{sec:bayes-eiv}
Unfortunately, the extensive results of~\citet{barbier2019optimal} do not cover the case studied here as the noise in the model depends on the norm of the estimator.  As is the case when we studied the ridge-regularized logistic regression estimator, we consider the Gaussian error-in-variables model of Definition~\ref{def:gaussian-eiv}.  Since we are mostly interested in the connection with missing data, we set the parameters $\alpha_c = \alpha_2 = \alpha$.  We will use the normalized data $\widetilde{\bz} = \bz/\sqrt{\alpha}$.  Exploiting orthogonality of the minimimum mean square estimator and its error, we write the conditional distribution of a label as
\begin{align} \label{eq:eiv-cond-dist}
P_Y\bigl(y \mid \widetilde{\bz}, \bt \bigr) = \E_G\Bigl\{ \rho'\bigl(y \cdot \sqrt{\alpha} \cdot \langle \widetilde{\bz}, \bt \rangle  + y \cdot \| \bt \|_2 \cdot \sqrt{\frac{1 - \alpha}{p}} \cdot G \bigr)\Bigr\},
\end{align}
where $G \sim \mathsf{N}(0, 1)$, and is independent of both $\bt$ as well as $\widetilde{\bz}$.  We thus define the conditional distribution $P_Y^{(\alpha)}$ as 
\begin{align}\label{eq:missing-cond-dist}
	P_Y^{(\alpha)}\bigl(y \mid \langle \widetilde{\bz}, \bt \rangle) =  \E_G\Bigl\{ \rho'\bigl(y \cdot \sqrt{\alpha} \cdot  \langle \widetilde{\bz}, \bt \rangle + y \cdot \varrho \cdot \sqrt{1 - \alpha} \cdot G \bigr)\Bigr\},
\end{align}
where we recall that $\varrho^2$ is the second moment of the distribution $P_{\theta}$.  We have the following corollary of the results in~\citet{barbier2019optimal}.
\begin{corollary}
		\label{conj:bayes-conj}
		Let $\bt_0 \in \mathbb{R}^p$ consist of coordinates drawn i.i.d. from the distribution $P_{\theta}$.  Assume the pair of matrices $(\bX, \bZ)$ belong to the $(\alpha, \alpha)$-universality class and use the data matrix $\bX$ as well as the ground truth $\bt_0$ to generate the labels $\by$ according to the logistic model~\eqref{eq:data_gen}.  Use the conditional distribution $P_Y^{(\alpha)}$~\eqref{eq:missing-cond-dist} to define the function $\Psi$~\eqref{eq:Psi-RS}.  Let $\pi(\mathrm{d} \bt_0 \mid \bZ, \by)$ denote the conditional distribution of $\bt_0$ given the observed data. The following hold.  
		\begin{itemize}
			\item[(a)] The variational problem~\eqref{eq:variational-RS} with $P_Y \equiv P_Y^{(\alpha)}$ admits a unique minimizer $q_{\star}$.
			\item[(b)] Let $\bt \sim \pi(\cdot \mid \bZ, \by)$.  Then, $
			\frac{1}{p} \lvert \langle \bt, \bt_0 \rangle \rvert \overset{\mathsf{P}}{\rightarrow} q_{\star}$.
			\item[(c)] Let $\widehat{\bt} = \int_{\mathbb{R}^d} \bt\, \pi(\mathrm{d}\bt \mid \bZ, \by)$.  Then, $\frac{1}{p}\mathbb{E}\bigl\{ \| \bt_0 - \widehat{\bt} \|_2^2\bigr\} \rightarrow \varrho  - q_{\star}$. 
		\end{itemize}
	\end{corollary}
	We use the corollary to compute the Bayes prediction error as well as angular error, beginning with the prediction error.  To this end, let $(\bx^{\mathsf{new}}, \bz^{\mathsf{new}})$ denote a new pair from a distribution in the $(\alpha, \alpha)$-universality class and let $\widetilde{\bz}^{\mathsf{new}} = \bz^{\mathsf{new}}/\sqrt{\alpha}$.  With $P_Y$ as in~\eqref{eq:eiv-cond-dist}, the estimator with optimal prediction error is given by the maximizer of the posterior marginals
			\[
			\widehat{Y}^{\mathsf{bayes}}(\widetilde{\bz}^{\mathsf{new}}) = \argmax_{y \in \{ \pm 1\}}\; \int_{\mathbb{R}^d} P_Y(y \mid \widetilde{\bz}^{\mathsf{new}}, \bt) \, \pi(\mathrm{d} \bt \mid \bZ, \by).
			\]
			The estimator $	\widehat{Y}^{\mathsf{bayes}}$ achieves prediction error 
			\[
			\pr\bigl\{ Y^{\mathsf{new}} \neq \widehat{Y}^{\mathsf{bayes}}(\widetilde{\bz}^{\mathsf{new}}) \bigr\} \overset{\mathsf{P}}{\rightarrow} 2\E\biggl\{ \rho'(\varrho G) \cdot \Phi\biggl(-\frac{\alpha q_{\star} G}{\sqrt{\alpha q_{\star}\varrho^2 - \alpha^2 q_{\star}^2}}\biggr)\biggr\},
			\]
			where $Y^{\mathsf{new}}$ denotes a fresh label generated from the ground truth $\bt_0$ and new data $\bx^{\mathsf{new}}$ according to the logistic model~\eqref{eq:data_gen}.
	The estimator with optimal error in the angular metric is
			\[
			\widehat{\bt}^{\mathsf{angle}} = \argmin_{\bt \in \mathbb{R}^d}\; \int_{\mathbb{R}^d} \cos^{-1}\biggl(\frac{\langle \bt, \bt' \rangle}{\| \bt \|_2 \| \bt' \|_2}  \biggr)\, \pi(\mathrm{d} \bt' \mid \bZ, \by),
			\]
			and it achieves angular error
			\[
			\cos^{-1}\biggl(\frac{\langle \widehat{\bt}^{\mathsf{angle}}, \bt_0 \rangle}{\| \widehat{\bt}^{\mathsf{angle}} \|_2 \| \bt_0 \|_2}  \biggr) \overset{\mathsf{P}}{\rightarrow} \cos^{-1}\Bigl(\sqrt{q_{\star}}/\varrho\Bigr).
			\]
		We emphasize that while the corollary above characterizes the Bayes error in the error-in-variables model, it does not directly characterize the Bayes error in the missing data model of interest as the conditional distribution differs from that in~\eqref{eq:eiv-cond-dist}.  Nonetheless, we conjecture that the same characterization holds for the missing data model considered here.  In line with the discussion at the beginning of the section, we remark that this conjecture equivalently forms a conjecture on the exact minimax risk over the set $\Theta_{R, \tau, K}$~\eqref{eq:parameter-set}.  In the next section, we will set $\varrho = R$ and compare these conjectured formulas with the performance of optimally regularized logistic regression.

\subsection{Comparison of optimal regularized logistic regression with Bayes lower bound} \label{subsec:bayes-numerical}
Given a fresh sample $\bz^{\mathsf{new}}$ and an estimator $\bt \in \mathbb{R}^p$, we compute the prediction error of the estimator using the two dimensional state variables $\sigma(\bt)$ and $\xi(\bt)$ as 
\[
\pr\bigl\{Y^{\mathsf{new}} \neq \mathsf{sgn}(\langle \bz^{\mathsf{new}}, \bt \rangle)\bigr\} = \phi_{\mathsf{test}}\bigl(\sigma(\bt), \xi(\bt)\bigr),
\]
where we have defined the function $\phi_{\mathsf{test}}: \mathbb{R}^2 \rightarrow \mathbb{R}$ as 
\begin{align}
	\label{eq:phi-test}\phi_{\mathsf{test}}(\sigma, \xi) := 2\E\biggl\{\rho'(RG) \cdot \Phi\biggl(-\frac{\xi R \alpha_c G}{\sqrt{\sigma^2\alpha_2 + (\alpha_2 - \alpha_c^2) \xi^2 R^2}}\biggr)\biggr\}.
\end{align}
Recalling the ridge-regularized logistic regression estimator $\widehat{\bt}(\bZ, \lambda)$~\eqref{eq:estimate}, we apply Theorem~\ref{thm:main} to obtain
\[
\phi_{\mathsf{test}}\bigl( \sigma(\widehat{\bt}(\bZ, \lambda)), \xi(\widehat{\bt}(\bZ, \lambda) \bigr) \overset{P}{\rightarrow} \phi_{\mathsf{test}}(\sigma_{\star}(\lambda), \xi_{\star}(\lambda)).
\]

In order to compare the optimal test error of single-imputed, ridge-regularized logistic regression with the Bayes lower bound, we first specify a triple of parameters $(\alpha, \delta, R)$, where $\alpha$ denotes the probability with which an entry is missing, $\delta = n/p$ denotes the ratio of samples to dimensions, and $R$ denotes the re-scaled norm of the ground truth $\| \bt_0 \|_2/\sqrt{p}$.  Additionally, we specify the distribution $P_{\theta}$ (as in Section~\ref{sec:bayes}) to denote the Gaussian distribution with mean zero and variance $R^2$.  This allows us to specify the replica symmetric potential~\eqref{eq:RS-potential} and optimize the replica symmetric potential to obtain an asymptotic overlap $q_{\star}$.  We subsequently compute the Bayes optimal test error via Corollary~\ref{conj:bayes-conj}.

On the other hand, in order to compute the optimal test error of ridge-regularized logistic regression, we define the function $T: \mathbb{R}_{+} \rightarrow \mathbb{R}$ as
\begin{align} \label{eq:T-func}
T(\lambda) = \phi_{\mathsf{test}}\bigl(\sigma_{\star}(\lambda), \xi_{\star}(\lambda)\bigr),
\end{align}
and compute the optimal test error as $\min_{\lambda \in \mathbb{R}_+} T(\lambda)$.  In order to evaluate the function $T$, it is necessary to compute the quantities $\sigma_{\star}(\lambda)$ and $\xi_{\star}(\lambda)$, which we do by solving the system of equations~\eqref{eq:system}.  

Figure~\ref{fig:fixalpha} fixes the probability of observing an entry $\alpha = 0.704$ and evaluates the Bayes optimal test error as well as the optimally regularized test error of logistic regression for several different values of the parameters $R$ and $\delta$.  Figure~\ref{fig:fixalpha}(a) plots contour lines of the two quantities overlayed.  As is evident from the plot, the two values are nearly indistinguishable visually.  Indeed, Figure~\ref{fig:fixalpha}(b) zooms in and plots contour lines of the difference between the two, which is of the order $10^{-5}$.  In Appendix~\ref{sec:additional-numerical-experiments}, we provide several more plots in different parameter regimes to further corroborate these observations.  We remark also that such an in depth comparison is made possible by the exact expressions in Theorem~\ref{thm:main} and Conjecture~\ref{conj:bayes-conj}, which can be evaluated quickly.  

\begin{figure*}[!h]
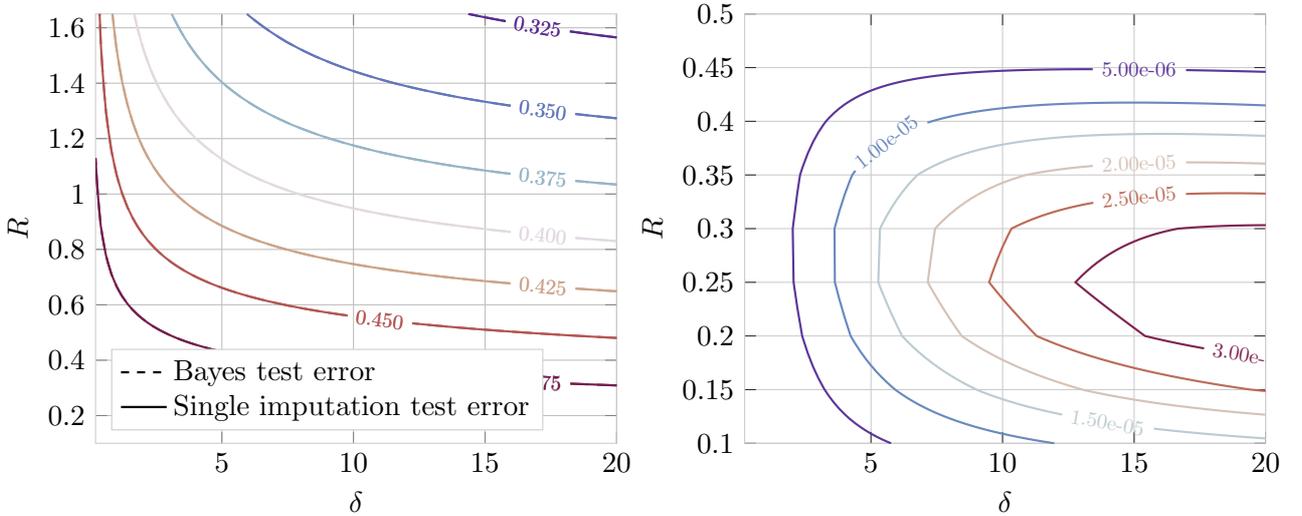

	\begin{subfigure}[b]{0.47\textwidth}
		\centerline{\input{figs/overlay-bayes-si-fix-alpha.tex}}
		\caption{Overlay of the two errors}    
		\label{subfig:overlay-fix-alpha}
	\end{subfigure}
	\hfill
	\begin{subfigure}[b]{0.47\textwidth}  
		\centerline{\input{figs/difference-bayes-si-fix-alpha.tex}}
		\caption{Difference between the errors}
		\label{subfig:difference-fix-alpha}
	\end{subfigure}
	\caption{A comparison of the test error of optimally ridge-regularized logistic regression with the (conjectured) Bayes' optimal test error.  The probability of an observing an entry is set as $\alpha=0.704$ and the contour plots are generated by numerically evaluating the asymptotic expressions for several values of the parameters $R$ (the radius of the problem) and $\delta$ (the ratio of samples to dimension).} 
	\label{fig:fixalpha}
\end{figure*}

We next empirically validate Theorem~\ref{thm:main} and plot the empirical evaluation on the same axis as the high dimensional asymptotics in~\eqref{eq:system}  as well as the Bayes error.  In addition to plotting the test error, we plot the angular error.  For a given estimator $\bt \in \mathbb{R}^p$, we have
\[
\angle(\bt, \bt_0) = \cos^{-1}\biggl(\frac{\langle \bt, \bt_0 \rangle}{\| \bt \|_2 \| \bt_0 \|_2}\biggr) = \phi_{\mathsf{angle}}(\sigma(\bt), \xi(\bt)),
\]
where we define the function $\phi_{\mathsf{angle}}:\mathbb{R}^2 \rightarrow \mathbb{R}$ as
\begin{align}
	\label{eq:phi-angle}
	\phi_{\mathsf{angle}}(\sigma, \xi) := \cos^{-1}\biggl(\frac{\xi R}{\sqrt{\sigma^2 + \xi^2 R^2}}\biggr).
\end{align}
We consider a setting with dimension $p = 600$ and vary the sample size $n$ (thereby varying the parameter $\delta$).  We specify the probability of missing an entry as $\alpha = 0.7$ and simulate the missingness mechanism MCAR($\alpha$).  We then recall the function $T$~\eqref{eq:T-func} and run ridge-regularized logistic with regularization strength $\lambda$ set as
\[
\lambda_{\mathsf{opt}} = \argmin_{\lambda \in \mathbb{R}_{+}}\; T(\lambda).
\]
That is, we use the asymptotic characterization provided by Theorem~\ref{thm:main} to perform model selection offline and then perform classification with the pre-selected model---bypassing the use of cross-validation.  We repeat this for $150$ independent trials.  Figure~\ref{fig:empirical-vary-delta} plots the results of this experiment.  Once more we observe that the Bayes error and optimally-regularized single imputation error are nearly indistinguishable.  Moreover, both of these exact expressions are nearly indistinguishable from the average empirical error.  

\begin{figure*}[!h]
	\begin{subfigure}[b]{0.47\textwidth}
		\centerline{\begin{tikzpicture}
	\begin{axis}[
		xlabel={$\delta = n/p$},
		ylabel={Test error},
		legend pos=north east,
		legend style={font=\small,fill=none},
		ymajorgrids=true,
		grid style=grid,
		]
		
		\addplot[
		color=Black,
		dotted,
		line width = 1.25pt,
		mark size=1.25pt,
		]
		table[x=delta,y=bayes_gen] {data/generalization_data_40.dat};
		\legend{Bayes}
		
		\addplot[
		color=Maroon,
		dashed,
		line width = 1.25pt,
		mark size=1.25pt,
		]
		table[x=delta,y=SI_gen] {data/generalization_data_40.dat};
		\addlegendentry{Single imputation (theory)}

		\addplot[
		only marks,
		mark repeat = 3,
		color=CadetBlue,
		mark=triangle,
		mark size=3pt,
		line width = 1.25pt   
		]
		table[x=delta,y=emp_gen] {data/emp_gen_mse_delta_10_40_v2.dat};
		\addlegendentry{Single imputation (empirical)}
		
		\addplot+[name path=A-gen,CadetBlue!20, no markers] table[x=delta,y=lower_gen] {data/emp_gen_mse_delta_10_40_v2.dat};
		\addplot+[name path=B-gen,CadetBlue!20, no markers] table[x=delta,y=upper_gen] {data/emp_gen_mse_delta_10_40_v2.dat};
		
		\addplot[CadetBlue!20] fill between[of=A-gen and B-gen];

		%
		
	\end{axis}
\end{tikzpicture}}
		\caption{Test error as a function of $\delta$}    
	\end{subfigure}
	\hfill
	\begin{subfigure}[b]{0.47\textwidth}  
		\centerline{\begin{tikzpicture}
	\begin{axis}[
		xlabel={$\delta = n/p$},
		ylabel={Angle error},
		legend pos=north east,
		ymajorgrids=true,
		grid style=grid,
		legend style={font=\small,fill=none}
		]
		
		\addplot[
		color=black,
		line width = 1.25pt,
		dotted,
		mark size=1.25pt   
		]
		table[x=delta,y=bayes_angle] {data/angle_data_40.dat};
		\legend{Bayes}
		
		\addplot[
		color=Maroon,
		line width = 1.25pt,
		dashed,
		mark size=1.25pt   
		]
		table[x=delta,y=si_angle] {data/angle_data_40.dat};
		\addlegendentry{Single imputation (theory)}

		\addplot[
		only marks,
		mark repeat = 3,
		color=CadetBlue,
		mark=triangle,
		line width=1.25pt,
		mark size=3pt   
		]
		table[x=delta,y=emp_angle] {data/emp_gen_mse_delta_10_40_v2.dat};
		\addlegendentry{Single imputation (empirical)}
		
		\addplot+[name path=A-angle,CadetBlue!20, no markers] table[x=delta,y=lower_angle] {data/emp_gen_mse_delta_10_40_v2.dat};
		\addplot+[name path=B-angle,CadetBlue!20, no markers] table[x=delta,y=upper_angle] {data/emp_gen_mse_delta_10_40_v2.dat};
		
		\addplot[CadetBlue!20] fill between[of=A-angle and B-angle];

		%
		
	\end{axis}
\end{tikzpicture}}
		\caption{Angle error as a function of $\delta$}
	\end{subfigure}
	\caption{A comparison of the Bayes error with the optimally regularized single imputation error.  The probability of observing an entry is fixed as $\alpha=0.7$ and the radius of the problem is fixed as $R = 4$, whereas the ratio of samples to dimensions $\delta$ is varied.  Triangular marks denote the empirical average of the empirical error, dashed maroon lines (barely visible) denote the exact single imputation error, and dashed black lines denote the Bayes error.  The shaded region denotes the inter-quartile range.} 
	\label{fig:empirical-vary-delta}
\end{figure*}
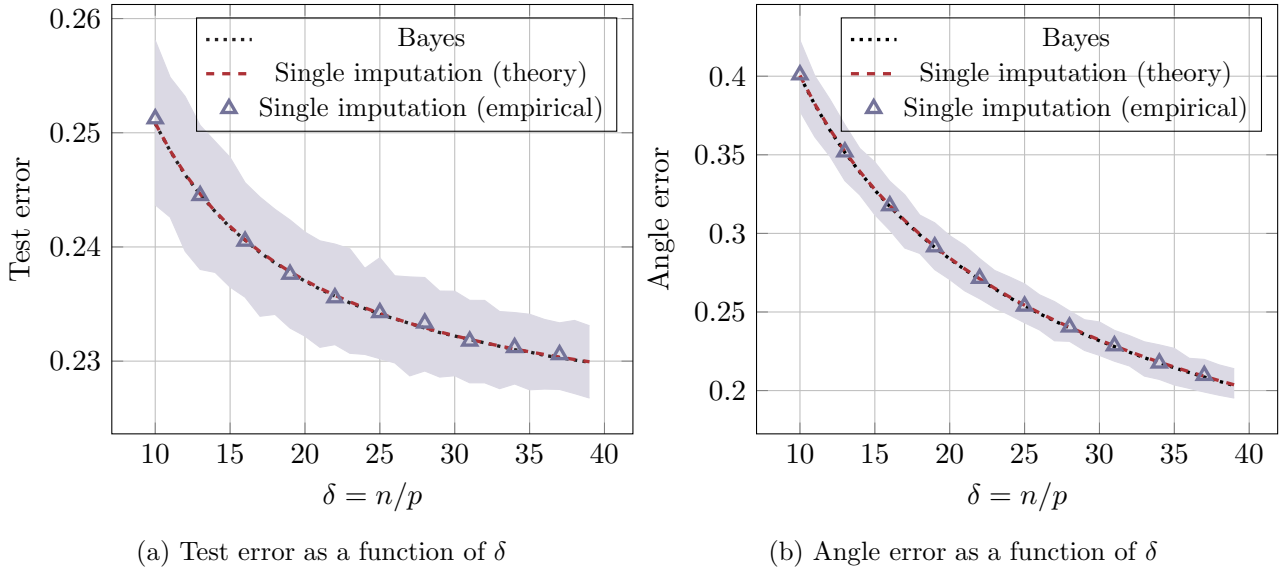

\subsection{The effect of regularization: prior imputation vs. single imputation}\label{subsec:effect-of-reg}
It is natural to wonder whether other simple strategies of handling missing data can match the Bayes optimal performance as well.  We demonstrate that this is not the case, even for procedures which use more knowledge of the covariates than single imputation.  In particular, we demonstrate that prior imputation---which uses full knowledge of the covariates' distribution---can perform significantly worse than single imputation, which uses only knowledge of the mean of the covariates.

The experimental set-up is as follows.  We fix the dimension $p=500$, the number of samples $n=1500$ (so that $\delta = 3$), and the radius of the problem $R = 2$.  The probability of observing an entry is set to be $\alpha = 0.85$ and we simulate data which is~\ref{eq:MCAR}.  Then, we vary the regularization strength $\lambda$ and run ridge regularized logistic regression with one of two data matrices: either formed using~\ref{eq:single_imp} or formed using~\ref{eq:prior_imp}.  We repeat this procedure $1000$ times.  The results are plotted in Figure~\ref{fig:empirical-single-vs-prior}.  We remark on two specific aspects of this simulation.  First, we note that for both imputation strategies, the unregularized estimator is significantly sub-optimal, and regularization alleviates the over-confidence problem in both situations.  Second, we note that regardless of the regularization strength, there is a non-negligible gap between single imputation and prior imputation.  These two aspects are present in both test error, illustrated in Figure~\ref{fig:empirical-single-vs-prior}(a) as well as angle error, illustrated in Figure~\ref{fig:empirical-single-vs-prior}(b).

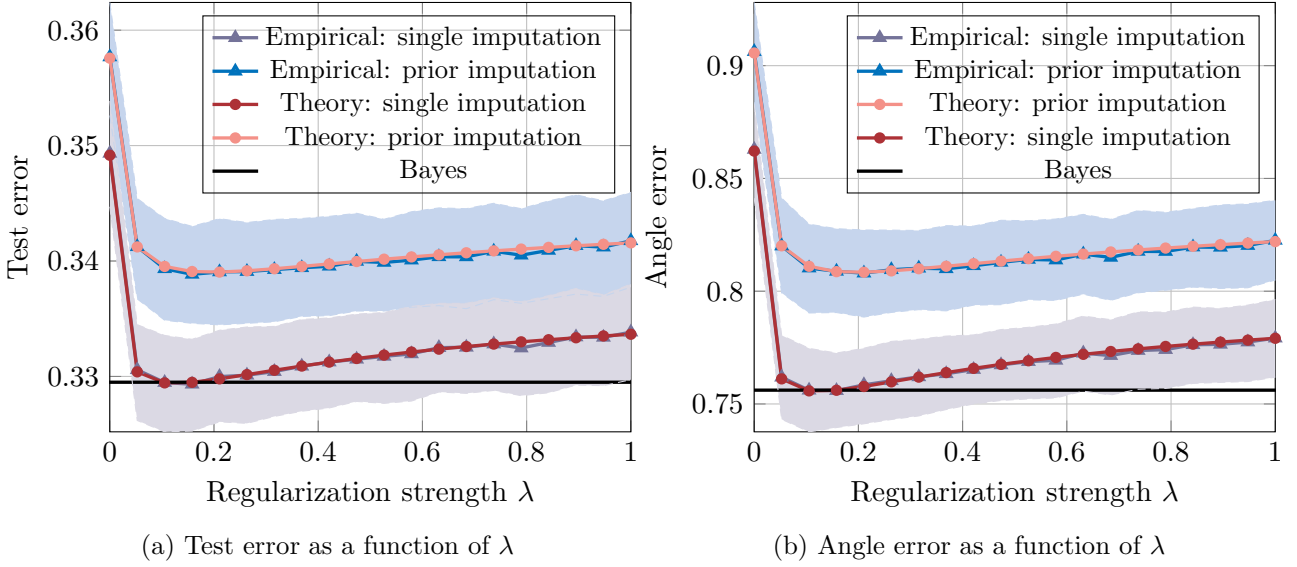
\begin{figure*}[!h]
	\begin{subfigure}[b]{0.47\textwidth}
		\centerline{\begin{tikzpicture}
	\begin{axis}[
		xlabel={Regularization strength $\lambda$},
		ylabel={Test error},
		 xmin=0, xmax=1,
		xtick={0, 0.2, 0.4, 0.6, 0.8,1},
		legend pos=north east,
		legend style={font=\small,fill=none},
		ymajorgrids=true,
		grid style=grid,
		enlargelimits=false,
		]
		
		\addplot[
		color=CadetBlue,
		mark=triangle,
		mark size=2pt,
		line width = 1.25pt   
		]
		table[x=lam,y=avg_gen] {data/regularization-data-mi-vs-si-larger-lam.dat};
		\legend{Empirical: single imputation}
		
		\addplot[
		color=RoyalBlue,
		mark=triangle,
		mark size=2pt,
		line width = 1.25pt
		]
		table[x=lam,y=avg_gen_mi] {data/regularization-data-mi-vs-si-larger-lam.dat};
		\addlegendentry{Empirical: prior imputation}
		
		\addplot[
		color=Maroon,
		mark=*,
		mark size=1.5pt,
		line width = 1.25pt
		]
		table[x=lam,y=exact_gen] {data/regularization-data-mi-vs-si-larger-lam.dat};
		\addlegendentry{Theory: single imputation}
		
		\addplot[
		color=Salmon,
		mark=*,
		mark size=1.5pt,
		line width = 1.25pt   
		]
		table[x=lam,y=exact_gen_mi] {data/regularization-data-mi-vs-si-larger-lam.dat};
		\addlegendentry{Theory: prior imputation}
		
		 \addplot[mark=none, black, line width=1.25pt, samples=2] {0.32949744648618856};
		\addlegendentry{Bayes}

		\addplot+[name path=A-mi,RoyalBlue!20, no markers] table[x=lam,y=lower_gen_mi] {data/regularization-data-mi-vs-si-larger-lam.dat};
		\addplot+[name path=B-mi,RoyalBlue!20, no markers] table[x=lam,y=upper_gen_mi] {data/regularization-data-mi-vs-si-larger-lam.dat};
		
		\addplot[RoyalBlue!20] fill between[of=A-mi and B-mi];

		\addplot+[name path=A-si,CadetBlue!20, no markers] table[x=lam,y=lower_gen] {data/regularization-data-mi-vs-si-larger-lam.dat};
		\addplot+[name path=B-si,CadetBlue!20, no markers] table[x=lam,y=upper_gen] {data/regularization-data-mi-vs-si-larger-lam.dat};
		
		\addplot[CadetBlue!20] fill between[of=A-si and B-si];

		
		%
		
	\end{axis}
\end{tikzpicture}}
		\caption{Test error as a function of $\lambda$}    
	\end{subfigure}
	\hfill
	\begin{subfigure}[b]{0.47\textwidth}  
		\centerline{\begin{tikzpicture}
	\begin{axis}[
		xlabel={Regularization strength $\lambda$},
		ylabel={Angle error},
		xmin=0, xmax=1,
		xtick={0, 0.2, 0.4, 0.6, 0.8,1},
		legend pos=north east,
		legend style={font=\small,fill=none},
		ymajorgrids=true,
		grid style=grid,
		enlargelimits=false,
		]
		
		\addplot[
		color=CadetBlue,
		mark=triangle,
		mark size=2pt,
		line width=1.25pt   
		]
		table[x=lam,y=avg_angle] {data/regularization-data-mi-vs-si-larger-lam.dat};
		\legend{Empirical: single imputation}
		
		\addplot[
		color=RoyalBlue,
		mark=triangle,
		mark size=2pt,
		line width=1.25pt   
		]
		table[x=lam,y=avg_angle_mi] {data/regularization-data-mi-vs-si-larger-lam.dat};
		\addlegendentry{Empirical: prior imputation}
		
		\addplot[
		color=Salmon,
		mark=*,
		line width=1.25pt,
		mark size=1.5pt   
		]
		table[x=lam,y=exact_angle_mi] {data/regularization-data-mi-vs-si-larger-lam.dat};
		\addlegendentry{Theory: prior imputation}
		
		\addplot[
		color=Maroon,
		mark=*,
		mark size=1.5pt,
		line width=1.25pt
		]
		table[x=lam,y=exact_angle] {data/regularization-data-mi-vs-si-larger-lam.dat};
		\addlegendentry{Theory: single imputation}
		
		\addplot[mark=none, black, line width=1.25pt, samples=2] {0.7561208504075206};
		\addlegendentry{Bayes}

		\addplot+[name path=A-mi,RoyalBlue!20, no markers] table[x=lam,y=lower_angle_mi] {data/regularization-data-mi-vs-si-larger-lam.dat};
		\addplot+[name path=B-mi,RoyalBlue!20, no markers] table[x=lam,y=upper_angle_mi] {data/regularization-data-mi-vs-si-larger-lam.dat};
		
		\addplot[RoyalBlue!20] fill between[of=A-mi and B-mi];

		\addplot+[name path=A-si,CadetBlue!20, no markers] table[x=lam,y=lower_angle] {data/regularization-data-mi-vs-si-larger-lam.dat};
		\addplot+[name path=B-si,CadetBlue!20, no markers] table[x=lam,y=upper_angle] {data/regularization-data-mi-vs-si-larger-lam.dat};
		
		\addplot[CadetBlue!20] fill between[of=A-si and B-si];

		
		%
		
	\end{axis}
\end{tikzpicture}}
		\caption{Angle error as a function of $\lambda$}
	\end{subfigure}
	\caption{A comparison of the regularized single imputation and prior imputation errors.  The probability of observing an entry is fixed as $\alpha=0.7$, the radius of the problem is fixed as $R=1$, and the ratio of samples to dimensions is fixed as $\delta = 10$.  Triangular marks denote averages of the empirical error and solid circles denote exact expressions evaluated via Theorem~\ref{thm:main}.  The shaded regions correspond to inter-quartile ranges.} 
	\label{fig:empirical-single-vs-prior}
\end{figure*}

\section{Proofs} \label{sec:proofs}
This section is primarily dedicated to the proofs of Theorem~\ref{thm:main} and Proposition~\ref{prop:sharp}, which are provided in Sections~\ref{sec:proof-thm-main} and~\ref{sec:proof-prop-sharp}, respectively.  Both Theorem~\ref{thm:main} and Proposition~\ref{prop:sharp} rely on parts of the following lemma which characterizes geometric properties of the estimator~\eqref{eq:estimate}. We provide its proof in Appendix~\ref{sec:geometric-prop}.
\begin{lemma}
	\label{lem:l2_norm_theta_hat}
	Under the setting of Theorem~\ref{thm:main}, there exists a tuple of positive constants $(M_0, M_1, M_2, M_3)$, depending only on $(K_1, K_2)$, such that the estimator $\widehat{\bt}(\bZ, \lambda)$~\eqref{eq:estimate} satisfies the following properties.
	\begin{itemize}
		\item[(a)] With probability at least $1 - 2e^{-n}$, both $\| \bZ\|_{\mathsf{op}} \leq M_0$ and $\bigl \| \widehat{\bt}(\bZ, \lambda) \bigr \|_2 \leq M_1 \sqrt{p}$.
		\item[(b)] With probability at least $1 - 6/n$, the maximum entry of the estimator $\widehat{\bt}$ is bounded as
		\[
		 \bigl \| \widehat{\bt}(\bZ, \lambda) \bigr\|_{\infty} \leq M_2 (\log{n})^{3/2} \cdot \bigl( \bigl \| \bt_0 \bigr \|_{\infty} \vee \log{n} \bigr)
		\]
		\item[(c)] With probability at least $1 - 2/(\delta \cdot n)$, the maximum entry of the product $\bZ\widehat{\bt}$ is bounded as
		\[
		\max_{j \in \{1, 2, \dots, n\}}\; \bigl\lvert \bigl \langle \bz_{j}, \widehat{\bt}(\bZ, \lambda) \bigr \rangle \bigr \rvert \leq M_3 \sqrt{\log{n}}.
		\]
	\end{itemize}
\end{lemma}
Equipped with this lemma, we turn now to the proof of Theorem~\ref{thm:main}

\subsection{Proof of Theorem~\ref{thm:main}} \label{sec:proof-thm-main}
We begin with some useful notation.  First, given any function $\phi: \mathbb{R}^2 \rightarrow \mathbb{R}$, we define the map $\psi: \mathbb{R}^p \rightarrow \mathbb{R}$ as
\begin{align}
	\label{eq:psi}
	\psi(\bt) = \phi\bigl(\sigma(\bt), \xi(\bt)\bigr).
\end{align}
We note that if $\phi$ is $1$-Lipschitz, straightforward computation implies that $\psi$ is $C/\sqrt{p}$ Lipschitz.  
Additionally, for all $\lambda \in \mathbb{R}_{\geq 0}$, define the function $\phi_{\star}: \mathbb{R}_{\geq 0} \rightarrow \mathbb{R}$ as $
\phi_{\star}(\lambda) = \phi\bigl(\sigma_{\star}(\lambda), \xi_{\star}(\lambda)\bigr)$.  
Note that it suffices to prove that the following holds for all $1$-Lipschitz functions $\phi: \mathbb{R}^2 \rightarrow \mathbb{R}$
\begin{align}\label{ineq:uniform-psi}
\sup_{\lambda \in [K_1, K_2]}\; \Bigl \lvert \psi\Bigl(\widehat{\bt}(\bZ, \lambda)\Bigr) - \phi_{\star}(\lambda)\Bigr \rvert \overset{P}{\rightarrow} 0,
\end{align}
since Theorem~\ref{thm:main} then follows upon taking the coordinate projections $\phi_{1}: (\sigma, \xi) \mapsto \sigma$ and $\phi_2: (\sigma, \xi) \mapsto \xi$.  The proof of the convergence relation~\eqref{ineq:uniform-psi} follows by first establishing non-asymptotic control of the deviations (pointwise in $\lambda$) and subsequently extending this to the interval $\lambda \in [K_1, K_2]$ via a straightforward approximation argument. 

\paragraph{Step 1: Pointwise (in $\lambda$) control.}  The crux of this step is the following lemma---whose proof  utilizes a perturbation strategy due to~\citet{montanari2017universality} in conjunction with the Lindeberg method~\citep{lindeberg1922neue,chatterjee2006generalization} and is provided in Section~\ref{sec:proof-universality-main-text}---which bounds the deviations in Eq.~\eqref{ineq:uniform-psi} by corresponding deviations for members of the $\mathrm{\eivdef}$ ensemble.
\begin{lemma}\label{lem:universality-main-text}
	Assume the setting of Theorem~\ref{thm:main} and let the random matrices $\widetilde{\bX}, \bG \in \mathbb{R}^{n \times p}$ belong to the $\eivdef$ ensemble with $(\widetilde{X}_{ij})_{i \leq n, j \leq p} \overset{\mathsf{i.i.d.}}{\sim} \mathsf{N}(0, 1/p)$.  Use $\widetilde{\bX}$ and the ground truth $\bt_0$ to draw labels $\widetilde{\by}$ and use these to define the loss $\mathcal{L}_n(\cdot; \bG, \lambda)$~\eqref{eq:loss} and its minimizer $\widehat{\bt}(\bG, \lambda)$~\eqref{eq:estimate}.  For any $1$-Lipschitz function $\phi$, there exists a triplet of positive constants $(c, C_1, C_2)$ depending only on $(K_1, K_2, K_3, K_4, \tau)$ such that
	\begin{align*}
		\pr\Bigl\{\bigl\lvert \psi\bigl(\widehat{\bt}(\bZ, \lambda)\bigr) - \phi_{\star}(\lambda)\bigr\rvert \geq C_1n^{-\tau/12}\Bigr\} &\leq \; 2\pr\biggl\{\Big\lvert \min_{\bt \in \mathbb{R}^p} \mathcal{L}_n(\bt; \bG, \lambda) - \mathfrak{M}_{\star}(\lambda)\Bigr\rvert \geq  c n^{-\tau/12}\biggr\}\\
		&+ \pr\biggl\{ \big\lvert \psi\bigl(\widehat{\bt}(\bG, \lambda)\bigr) - \phi_{\star}(\lambda)\bigr\rvert \geq cn^{-\tau/12}\biggr\} + C_2n^{-\tau/6} (\log{n})^2,
	\end{align*}
where we have set $\mathfrak{M}_{\star}(\lambda) = L\bigl(\sigma_{\star}(\lambda), \xi_{\star}(\lambda), \gamma_{\star}(\lambda)\bigr)$.
\end{lemma}
\noindent Continuing, we apply Lemma~\ref{lem:universality-main-text} in conjunction with Proposition~\ref{prop:sharp}, with $\epsilon = c n^{-\tau/12}$, and take $n$ large enough to obtain the inequality
\begin{align} \label{ineq:pointwise-lambda}
	\pr\biggl\{\bigl\lvert \psi\bigl(\widehat{\bt}(\bZ, \lambda)\bigr) - \phi_{\star}(\lambda)\bigr\rvert \geq C_1 \cdot n^{-\tau/12}\biggr\} \leq C_2n^{-\tau/6} \cdot (\log{n})^2.
\end{align}

\paragraph{Step 2: Uniform control over $\lambda \in [K_1, K_2]$.}
We require the following lemma, which provides estimates of the Lipschitz constants of the maps $\lambda \mapsto \phi_{\star}(\lambda)$ and $\lambda \mapsto \psi(\widehat{\bt}(\bZ, \lambda))$.  We defer its proof to Section~\ref{subsec:proof-lipschitz-minimizers}.
\begin{lemma}
	\label{lem:lipschitz-minimizers}
	Assume the setting of Theorem~\ref{thm:main}.  Let $\phi: \mathbb{R}^2 \rightarrow \mathbb{R}$ be a $1$-Lipschitz function and use it to define the function $\psi$~\eqref{eq:psi}.  There exists a pair of positive constants $(L_1, L_2)$, depending only on $(K_1, K_2, K_3, K_4, \tau)$ such that the following hold.
	\begin{itemize}
		\item[(a)] With probability at least $1 - 2e^{-n}$, the  map $\lambda \mapsto \psi\bigl(\widehat{\bt}(\bZ, \lambda)\bigr)$ is $L_1$-Lipschitz on $[K_1, K_2]$.  
		\item[(b)] The map $\lambda \mapsto \phi\bigl(\sigma_{\star}(\lambda), \xi_{\star}(\lambda)\bigr)$ is $L_2$-Lipschitz on $[K_1, K_2]$.
	\end{itemize}
\end{lemma}

\noindent Equipped with this lemma, we proceed with the proof.  Let $\mathcal{N}_{\lambda}$ denote an $\epsilon$-net of the interval $[\lambda_{\min}, \lambda_{\max}]$ so that for each \sloppy\mbox{$\lambda \in [\lambda_{\min}, \lambda_{\max}]$}, there exists $\lambda_i \in \mathcal{N}_{\lambda}$ such that $\lvert \lambda - \lambda_i \rvert \leq \epsilon$.  Now note the decomposition
\begin{align}
	\label{decomp:unif_lambda}
	\psi\bigl(\widehat{\bt}(\bZ, \lambda)\bigr) - \phi_{\star}(\lambda) = T_1 + T_2 + T_3,
\end{align}
where 
\begin{align*}
	T_1 = \underbrace{\psi\bigl(\widehat{\bt}(\bZ, \lambda)\bigr) - \psi\bigl(\widehat{\bt}(\bZ, \lambda_i)\bigr)}_{\text{approximation error}}&,\quad T_2 = \underbrace{\psi\bigl(\widehat{\bt}(\bZ, \lambda_i)\bigr) - \phi_{\star}(\lambda_i)}_{\text{fluctuation}}, \quad \text{ and } \quad T_3 = \underbrace{\phi_{\star}(\lambda_i) - \phi_{\star}(\lambda)}_{\text{approximation error}}.
\end{align*}
Next, we invoke Lemma~\ref{lem:lipschitz-minimizers} to obtain the inequalities $\lvert T_1 \rvert \leq L_1 \epsilon$ and $\lvert T_3 \rvert \leq L_2 \epsilon$.  Towards bounding the fluctuation term $T_2$ uniformly on the net $\mathcal{N}_{\lambda}$, we define the event
\[
\mathcal{A} = \bigcap_{\lambda \in \mathcal{N}_{\lambda}}\Bigl\{\bigl\lvert \psi\bigl(\widehat{\bt}(\bZ, \lambda)\bigr) - \phi_{\star}(\lambda)\bigr\rvert \leq \epsilon\Bigr\}.
\]
Setting $\epsilon = C_1 n^{-\tau/12}$ and subsequently applying the union bound in conjunction with the inequality~\eqref{ineq:pointwise-lambda} yields the inequality
\[\pr\{\mathcal{A}\} \geq 1 - C\epsilon^{-1}n^{-\tau/6} \cdot (\log{n})^2 \geq 1 - C_2n^{-\tau/12} \cdot (\log{n})^2
\]
Consequently, on event $\mathcal{A}$, we obtain the upper bound $\lvert T_2 \rvert \leq \epsilon$.  Putting the pieces together, we obtain the inequality
\begin{align} \label{ineq:uniform-lambda}
	\pr\biggl\{\sup_{\lambda \in [\lambda_{\min}, \lambda_{\max}]}\left \lvert \psi\bigl(\widehat{\bt}(\bZ, \lambda)\bigr) - \phi_{\star}(\lambda)  \right \rvert \geq C' \cdot n^{-\tau/12}\biggr\} \leq C n^{-\tau/12} \cdot (\log{n})^2.
\end{align}
Taking $n \rightarrow \infty$ yields the inequality~\eqref{ineq:uniform-psi}, which concludes the proof. \qed

\subsubsection{Proof of Lemma~\ref{lem:universality-main-text}}
\label{sec:proof-universality-main-text}
Throughout the proof of this lemma, we will suppress the dependence on $\lambda$ and write $\widehat{\bt}(\cdot, \lambda) = \widehat{\bt}(\cdot)$ as well as $\phi_{\star}(\lambda) = \phi_{\star}$ when the context is clear.
Next, let the tuple of constants $(M_1, M_2, M_3)$ be as in Lemma~\ref{lem:l2_norm_theta_hat} and define the set
\begin{align}\label{eq:T-def}
	\hspace*{-0.5cm}\mathbb{T} = \Bigl\{\bt \in \mathbb{R}^{p}:\; &\norm{\bt}_2 \leq M_1\sqrt{n}, \; \norm{\bt}_{\infty} \leq M_2  (\log{n})^2 (\| \bt_0 \|_{\infty} \vee 1), \; \| \bZ\bt\|_{\infty} \leq M_3\sqrt{\log{n}}\Bigr\}.
\end{align}
Note that by Lemma~\ref{lem:l2_norm_theta_hat}, with probability at least $1 - c/n$, the minimum of $\mathcal{L}_n$ over $\mathbb{R}^n$ is achieved over $\mathbb{T}$; that is, $
\min_{\bt \in \mathbb{R}^p}\; \mathcal{L}_n(\bt; \bZ, \lambda) = \min_{\bt \in \mathbb{T}}\; \mathcal{L}_n(\bt; \bZ, \lambda)$. 
Now, we pursue a strategy of~\citet{montanari2017universality}, which studies universality of the loss function itself and then transfers this control to functions of the minimizers.  In particular, we introduce the perturbed loss
\begin{align}
	\label{eq:perturbed-loss}
	\mathfrak{L}_n(\bt; s, \bZ, \lambda) = \mathcal{L}_n(\bt; \bZ, \lambda) + s \psi(\bt),
\end{align}
and define the minimum value (over the set $\mathbb{T}$) of this perturbed loss $\mathfrak{M}_\mathbb{T}$ as well as its associated difference quotient $\mathfrak{D}_\mathbb{T}$ as 
	\begin{align}\label{eq:perturbed-difference}
		\mathfrak{M}_\mathbb{T}(s; \bZ) := \min_{\bt \in \mathbb{T}}\; \mathfrak{L}_n(\bt; s, \bZ, \lambda), \quad \text{ and } \quad \mathfrak{D}_\mathbb{T}(s, \bZ) := \frac{1}{s} \bigl( \mathfrak{M}_\mathbb{T}(s; \bZ) - \mathfrak{M}_\mathbb{T}(0; \bZ) \bigr).
	\end{align}
Note that $\lim_{s \downarrow 0} \mathfrak{D}_\mathbb{T}(s; \bZ) = \psi(\widehat{\bt}(\bZ))$ is the derivative---evaluated at zero---of the perturbed minimum $\mathfrak{M}_\mathbb{T}$, and we will use the following bound of a finite differences approximation 
\begin{align}\label{ineq:finite-diff}
0 \leq \psi\bigl(\widehat{\bt}(\bZ)\bigr) - \mathfrak{D}_\mathbb{T}(s; \bZ) \leq \frac{Cs}{\lambda}, \qquad \text{ for } s > 0.
\end{align}
We take this inequality for granted for the time being, deferring its proof to the end of the section.  For notational convenience, during this proof we will write
\begin{align}\label{eq:m-star}
\mathfrak{M}_{\star} := L\bigl(\sigma_{\star}(\lambda), \xi_{\star}(\lambda), \gamma_{\star}(\lambda)\bigr).
\end{align}
We additionally require the following lemma, which provides two sided bounds on the distribution function of the random variable $\mathfrak{M}_\mathbb{T}(s; \bZ)$ in terms of the distribution function of the random variable $\mathfrak{M}_\mathbb{T}(s; \bG)$.  We provide the proof, which relies on the Lindeberg principle~\citep{chatterjee2006generalization,lindeberg1922neue}, in Appendix~\ref{sec:proof-universality}.
\begin{lemma}
	\label{lem:universality}
	Under the setting of Lemma~\ref{lem:universality-main-text}, there exists a positive constant $C$, depending only on $(K_1, K_2, K_3, K_4, \tau)$, such that for any scalars $\lvert s \rvert \leq 1$ and $t \in \mathbb{R}$, the following bounds hold.
	\begin{align*}
		\pr\Bigl\{\mathfrak{M}_\mathbb{T}(s; \bZ) \leq t \Bigr\} &\leq \pr\left\{\mathfrak{M}_\mathbb{T}(s; \bG)  \leq t + n^{-\tau/6} \right\} + Cn^{-\tau/6} \cdot (\log{n})^2, \quad \text{ and }\\
		\pr\Bigl\{\mathfrak{M}_\mathbb{T}(s; \bZ) \geq t \Bigr\} &\leq \pr\left\{\mathfrak{M}_\mathbb{T}(s; \bG) \geq t - n^{-\tau/6} \right\} + Cn^{-\tau/6} \cdot (\log{n})^2.
	\end{align*}
\end{lemma}
\noindent Equipped with these tools, we proceed to the proof of Lemma~\ref{lem:universality-main-text}.  We decompose the difference $\psi\bigl(\widehat{\bt}(\bZ)\bigr) - \phi_{\star}$ as 
\begin{align*}	
	\psi\bigl(\widehat{\bt}(\bZ)\bigr) - \phi_{\star} = \bigl( \psi\bigl(\widehat{\bt}(\bZ)\bigr) - \mathfrak{D}_\mathbb{T}(s; \bZ)\bigr) +  \bigl( \mathfrak{D}_\mathbb{T}(s; \bZ) - \phi_{\star}\bigr).
\end{align*}
Subsequently, we apply the triangle inequality in conjunction with the finite differences approximation~\eqref{ineq:finite-diff} to obtain the upper bound
\begin{align}	\label{ineq:main-univ-decomp}
	\bigl \lvert \psi\bigl(\widehat{\bt}(\bZ)\bigr) - \phi_{\star} \bigr \rvert \leq \frac{Cs}{\lambda} + \bigl\lvert \mathfrak{D}_\mathbb{T}(s; \bZ) -\phi_{\star}\bigr\rvert.
\end{align}
Focusing on the second term, we add and subtract the quantity $\mathfrak{M}_{\star}/s$~\eqref{eq:m-star}, expand $\mathfrak{D}_\mathbb{T}$ according to its definition~\eqref{eq:perturbed-difference}, and apply the union bound to obtain the inequality
\begin{align*}
	\hspace*{-0.415cm}\pr\Bigl\{ \bigl \lvert \mathfrak{D}_\mathbb{T}(s; \bZ) - \phi_{\star} \bigr \rvert \geq t \Bigr\} \leq \pr\biggl\{\frac{1}{s} \big\lvert \mathfrak{M}_{\mathbb{T}}(s; \bZ) - \mathfrak{M}_{\star} - s\phi_{\star}\bigr\rvert \geq \frac{t}{2}\biggr\} + \pr\biggl\{\frac{1}{s}\bigl\lvert \mathfrak{M}_{\mathbb{T}}(0; \bZ) - \mathfrak{M}_{\star} \bigr \rvert \geq \frac{t}{2}\biggr\}.
\end{align*}
Applying Lemma~\ref{lem:universality} yields the pair of upper bounds
	\begin{align*}
		\pr\biggl\{\frac{1}{s} \big\lvert \mathfrak{M}_\mathbb{T}(s; \bZ) - \mathfrak{M}_{\star} - s\phi_{\star}\bigr\rvert \geq \frac{t}{2}\biggr\} &\leq \pr\biggl\{\frac{1}{s} \big\lvert \mathfrak{M}_\mathbb{T}(s; \bG) - \mathfrak{M}_{\star} - s\phi_{\star}\bigr\rvert \geq \frac{t}{2}  - \frac{n^{-\tau/6}}{s}\biggr\} + \epsilon_n, \\
		\pr\biggl\{\frac{1}{s}\bigl\lvert \mathfrak{M}_\mathbb{T}(0; \bZ) - \mathfrak{M}_{\star} \bigr \rvert \geq \frac{t}{2}\biggr\} &\leq \pr\biggl\{\frac{1}{s} \big\lvert \mathfrak{M}_\mathbb{T}(0; \bG) - \mathfrak{M}_{\star}\bigr\rvert \geq \frac{t}{2} - \frac{n^{-\tau/6}}{s}\biggr\} + \epsilon_n,
	\end{align*}
where we have let $\epsilon_n = Cn^{-\tau/6} \cdot (\log{n})^2$.
We focus our attention on the first of the pair of inequalities and apply the triangle inequality in conjunction with the inequality~\eqref{ineq:finite-diff} to obtain
\begin{align*}
	\frac{1}{s} \big\lvert \mathfrak{M}_\mathbb{T}(s; \bG) - \mathfrak{M}_{\star} - s\phi_{\star}\bigr\rvert &\leq \bigl \lvert \mathfrak{D}_n(s; \bG) - \psi\bigl(\widehat{\bt}(\bG)\bigr) \bigr \rvert + \frac{1}{s} \bigl \lvert \mathfrak{M}_\mathbb{T}(0; \bG) - \mathfrak{M}_{\star} \bigr \rvert + \bigl \lvert \psi\bigl(\widehat{\bt}(\bG)\bigr) - \phi_{\star} \bigr \rvert\\
	&\leq \frac{Cs}{\lambda} + \frac{1}{s} \bigl \lvert \mathfrak{M}_\mathbb{T}(0; \bG) - \mathfrak{M}_{\star} \bigr \rvert + \bigl \lvert \psi\bigl(\widehat{\bt}(\bG)\bigr) - \phi_{\star} \bigr \rvert,
\end{align*}
 Combining the previous three displays, we deduce the inequality
\begin{align*}
	\pr\biggl\{ \bigl \lvert \mathfrak{D}_{\mathbb{T}}(s; \bZ) - \phi_{\star} \bigr \rvert \geq \frac{t}{2} \biggr\} \leq \;&2\pr\biggl\{\frac{1}{s} \big\lvert \mathfrak{M}_\mathbb{T}(0; \bG) - \mathfrak{M}_{\star}\bigr\rvert \geq \frac{t}{12} - \frac{n^{-\tau/6}}{3s}\biggr\}\\
	&+ \pr\biggl\{ \big\lvert \psi\bigl(\widehat{\bt}(\bG)\bigr) - \phi_{\star}\bigr\rvert \geq \frac{t}{12} - \frac{n^{-\tau/6}}{3s}\biggr\} + Cn^{-\tau/6} \cdot (\log{n})^2,
\end{align*}
as long as the pair of scalars $(s, t)$ satisfy the inequality $t \geq Cs/\lambda$.  Combining the bound of the previous display with the inequality~\eqref{ineq:main-univ-decomp} yield
\begin{align*}
	\pr\Bigl\{\bigl \lvert \psi\bigl(\widehat{\bt}(\bZ)\bigr) - \phi_{\star} \bigr \rvert \geq t\Bigr\} \leq \;&2\pr\biggl\{\frac{1}{s} \big\lvert \mathfrak{M}_\mathbb{T}(0; \bG) - \mathfrak{M}_{\star}\bigr\rvert \geq \frac{t}{12} - \frac{n^{-\tau/6}}{3s}\biggr\}\\
	&+ \pr\biggl\{ \big\lvert \psi\bigl(\widehat{\bt}(\bG)\bigr) - \phi_{\star}\bigr\rvert \geq \frac{t}{12} - \frac{n^{-\tau/6}}{3s}\biggr\} + Cn^{-\tau/6} \cdot (\log{n})^2,
\end{align*}
which again holds for $t \geq Cs/\lambda$.   The conclusion follows upon recognizing that \sloppy\mbox{$
	\mathfrak{M}_\mathbb{T}(0; \bG) = \min_{\bt \in \mathbb{T}}\; \mathcal{L}_n(\bt; \bG, \lambda) = \min_{\bt \in \mathbb{R}^p} \; \mathcal{L}_n(\bt; \bG, \lambda)$},
substituting the definition of $\mathfrak{M}_{\star}$~\eqref{eq:m-star}, and setting $t = C\lambda^{-1/2}n^{-\tau/12}$ and $s = \lambda^{1/2}n^{-\tau/12}$ for a large enough constant $C$.  It remains to prove the inequality~\eqref{ineq:finite-diff}.
\vspace{-0.2cm}
\paragraph{Proof of the finite differences approximation~\eqref{ineq:finite-diff}.} We prove the lower bound before turning our attention to the upper bound. 

\smallskip
\noindent\underline{Non-negativity.}  Let $\bt_s$ denote any minimizer of the perturbed loss $\mathfrak{L}_n(\bt; s, \bZ, \lambda)$ and expand
\[
\psi\bigl(\widehat{\bt}(\bZ) \bigr) - \mathfrak{D}_\mathbb{T}(s, \bZ) = \frac{1}{s} \cdot \bigl[ \mathfrak{L}_n(\widehat{\bt}(\bZ); s, \bZ, \lambda) - \mathfrak{L}_n(\bt_s; s, \bZ, \lambda)\bigr] \geq 0,
\]
where the final inequality follows since $\bt_s$ minimizes the loss $\mathfrak{L}_n(\cdot; s, \bZ, \lambda)$.

\medskip
\noindent\underline{Upper bound.} Again, let $\bt_s$ denote any minimizer of the perturbed loss $\mathfrak{L}_n(\bt; s, \bZ, \lambda)$ and expand
\begin{align} \label{ineq:perturbation-lemma-ineq1}
	s \cdot \bigl[ \psi\bigl(\widehat{\bt}(\bZ) \bigr) - \mathfrak{D}_\mathbb{T}(s, Z) \bigr] &= \mathcal{L}_n(\widehat{\bt}(\bZ); \bZ, \lambda) - \mathcal{L}_n(\bt_s; \bZ, \lambda) + s\psi\bigl(\widehat{\bt}(\bZ) \bigr) - s \psi(\bt_s)\nonumber\\
	&\leq s \cdot \bigl[ \psi\bigl(\widehat{\bt}(\bZ) \bigr) -  \psi(\bt_s) \bigr] \leq \frac{Cs}{\sqrt{p}} \| \widehat{\bt}(\bZ) - \bt_s \|_2,
\end{align}
where the penultimate inequality follows since $\widehat{\bt}(\bZ)$ minimizes $\mathcal{L}_n$, and the final inequality follows since the Lipschitz constant of $\psi$ is bounded as $C/\sqrt{p}$.  It remains to bound the quanity $\| \widehat{\bt}(\bZ) - \bt_s \|_2$.  To this end, note the lower bound $
0 \leq \mathfrak{L}_n(\widehat{\bt}(\bZ)) - \mathfrak{L}_n(\bt_s)$. 
Expanding and applying the Lipschitz nature of the the function $\psi$ yields
\[
0 \leq \mathcal{L}_n(\widehat{\bt}(\bZ)) - \mathcal{L}_n(\bt_s) + Cs\cdot 	\frac{\| \widehat{\bt}(\bZ) - \bt_s \|_2}{\sqrt{p}}.
\]
Re-arranging and applying strong convexity of the function $\mathcal{L}_n$ yields the upper bound
\begin{align}\label{ineq:perturbation-lemma-ineq2}
	\frac{1}{\sqrt{p}}\| \widehat{\bt}(\bZ) - \bt_s \|_2 \leq \frac{Cs}{\lambda}.
\end{align}
To conclude, substitute the inequality~\eqref{ineq:perturbation-lemma-ineq2} into the inequality~\eqref{ineq:perturbation-lemma-ineq1} and re-arrange. \qed

\subsection{Proof of Proposition~\ref{prop:sharp}}\label{sec:proof-prop-sharp}
We begin by introducing some notation and a few technical tools in Section~\ref{sec:preliminaries-prop-sharp}.  Then, in Subsection~\ref{subsec:proof-prop-sharp-a}, we prove Proposition~\ref{prop:sharp}(a).  Proposition~\ref{prop:sharp}(b) follows from a nearly identical argument to that of part (a), so we omit it for brevity.

\subsubsection{Preliminaries} \label{sec:preliminaries-prop-sharp}
Throughout this section, we will consider the regularization strength $\lambda$ as fixed.  Let us first define the auxiliary loss $\ell_n: \mathbb{R}^p \times \mathbb{R}_{\geq 0} \rightarrow \mathbb{R}$ as
\begin{samepage}
\begin{align}
	\label{def:aux_loss}
	\ell_n(\bt, \gamma) = \;&\frac{\lambda \norm{\bt}_2^2}{2p}  - \frac{\alpha_2\gamma \bigl\| \pproj \bg\bigr\|_2^2 \norm{\proj_{\bt_0}^{\perp}\bt}_2^2}{2np}\nonumber\\
	&+ \frac{1}{n} \sum_{i=1}^{n}\min_{u_i \in \mathbb{R}}\Bigl\{\rho(-y_i u_i) + \frac{\gamma}{2}\Bigl(u_i - \left(\bG \proj_{\bt_0} \bt\right)_{i} - \frac{\sqrt{\alpha_2}\norm{\proj_{\bt_0}^{\perp}\bt}_2h_i}{\sqrt{p}}\Bigr)^2\Bigr\},
\end{align}
\end{samepage}
where $(h_1, h_2, \dots, h_n) = \bh \sim \mathsf{N}(0, \bI_n)$ and $\bg \sim \mathsf{N}(0, \bI_p)$ are independent of each other as well as of all other randomness in the problem.  
We introduce the notation $Z_{1, i} = \frac{(\bG \proj_{\bt_0} \bt)_{i}}{\sqrt{\alpha_2} \xi R}$ and  $Z_{2, i} = h_i$
where \sloppy\mbox{$(Z_{1, i})_{1 \leq i \leq n} \overset{\mathsf{i.i.d.}}{\sim} \mathsf{N}(0, 1)$}, $(Z_{2, i})_{1 \leq i \leq n} \overset{\mathsf{i.i.d.}}{\sim} \mathsf{N}(0, 1)$ and $(Z_{1, i})_{1 \leq i \leq n}$ are independent of $(Z_{2, i})_{1 \leq i \leq n}$.  Using these, we define the scalarized auxiliary loss $L_n: \mathbb{R}_{\geq 0} \times \mathbb{R}\times \mathbb{R}_{\geq 0} \rightarrow \mathbb{R}$ as 
\begin{align}
	\label{def:aux_scalar}
	L_n(\sigma, \xi, \gamma) \ =\ &\frac{\lambda (\sigma^2 + \xi^2 R^2)}{2} - \frac{\alpha_2 \gamma \sigma^2\bigl \| \pproj \bg\bigr\|_2^2}{2n} \nonumber\\
	&+ \frac{1}{n} \sum_{i=1}^{n} \min_{u_i \in \mathbb{R}}\left\{\rho(-y_i u_i) + \frac{\gamma}{2}\left(u_i - \sqrt{\alpha_2} \xi R Z_{1, i} - \sqrt{\alpha_2} \sigma Z_{2, i}\right)^2\right\}.
\end{align}
Note that (i.) under the change of variables~\eqref{def:shorthand-xi-sigma}, \sloppy\mbox{$
	\ell_n(\bt, \gamma) =  L_n(\sigma(\bt), \xi(\bt), \gamma)$} and (ii.) the large $n$ limit of $L_n$ coincides with the asymptotic loss $L$~\eqref{eq:asymploss}.  The latter point will be made precise in Lemma~\ref{lem:costconcentration} to follow.  

We next define error sets $\mathbb{D}_{\epsilon}$ and $\mathcal{D}_{\epsilon}$, which we will use to deduce properties of the minimizers of the loss functions $\mathcal{L}_n$ and $L_n$. 
\begin{definition}[Error set]\label{def:error_sets}
	Consider problem parameters satisfying Assumption~\ref{asm:regularity}.  For any scalar $\epsilon \in [0, 1)$, the \emph{error set} $\mathbb{D}_{\epsilon} \subseteq \mathbb{R}^p$ and its scalarized counterpart $\mathcal{D}_{\epsilon} \subseteq{R}^2$ are defined as
	\[
	\mathbb{D}_{\epsilon} = \Bigl\{\bt \in \mathbb{R}^p: \lvert \xi(\bt) - \xi_{\star} \rvert \vee   \lvert \sigma(\bt)- \sigma_{\star} \rvert \geq \epsilon  \Bigr\} \quad \text{ and } \quad 	\mathcal{D}_{\epsilon} = \Bigl\{(\sigma, \xi) \in \mathbb{R}^2: \lvert \xi - \xi_{\star} \rvert \vee \lvert \sigma - \sigma_{\star} \rvert  \geq \epsilon\Bigr\}.
	\]
\end{definition}
\noindent With these preliminary definitions in hand, we state two technical lemmas which we will use throughout.  The first connects the loss function $\mathcal{L}_n$~\eqref{eq:loss} to the auxiliary loss $\ell_n$~\eqref{def:aux_loss}. 
\begin{lemma}
	\label{lem:ao}
	Consider the setting of Proposition~\ref{prop:sharp}.  There exists a pair of positive constants $(c, C)$ depending only on $(K_1, K_2)$ such that if $D$ denotes either (i) the set $D = \mathbb{B}_2(C)$ or (ii) the set $D = \mathbb{D}_{\epsilon} \cap \mathbb{B}_2(C)$ as in Definition~\ref{def:error_sets}, then for all $t \in \mathbb{R}$, 
		\begin{align}
		\label{ineq:ao_ineq1}
		\pr\Bigl\{\min_{\bt \in D}\; \mathcal{L}_n(\bt) \leq t\Bigr\} &\leq 2\pr\Bigl\{\min_{\bt \in D}\max_{\gamma \geq 0 }\; \ell_n(\bt, \gamma) \leq t\Bigr\} .
	\end{align}
	Moreover, if $D = \mathbb{B}_2(C)$, we have in addition
	\begin{align}
		\label{ineq:ao_ineq2}
		\pr\Bigl\{\min_{\bt \in D}\; \mathcal{L}_n(\bt) \geq t \Bigr\} &\leq 2\pr\Bigl\{\min_{\bt \in D}\max_{\gamma \geq 0}\; \ell_n(\bt, \gamma) \geq t\Bigr\}.
	\end{align}
\end{lemma}
\noindent The proof of this lemma---which we provide in Section~\ref{subsubsec:proof_lemma_ao}---relies on the CGMT~\citep{thrampoulidis2018}, a sharpening of Gordon's classical Gaussian comparison inequality~\citep{gordon1985some,Gordon1988}.  

The next lemma---whose proof can be found in Section~\ref{subsubsec:proof_cost_concentration}---provides pointwise and uniform concentration inequalities connecting the empirical loss $L_n$ to the asymptotic loss $L$.

\begin{lemma}
	\label{lem:costconcentration}
	Assume the setting of Proposition~\ref{prop:sharp} and let $M$ denote a positive constant.  There exists a pair of positive constants $(c, C)$, depending only on $(K_1, K_2, M)$ such that the following hold for all $\epsilon \in [0, 1)$.
	\begin{itemize}
		\item[(a)] For all $\gamma > 0$, $\sigma \in [0, M]$, and $\xi \in [-M, M]$, 
			\begin{align}
			\label{term:pointwise_concentration}
			\pr\Bigl\{\bigl \lvert L_n(\sigma, \xi, \gamma) - L(\sigma, \xi, \gamma) \bigr \rvert \geq \epsilon \Bigr\} \leq C\exp\Bigl\{-c\min\bigl(n\epsilon^2 \gamma^{-2}, n\epsilon \gamma^{-1}\bigr)\Bigr\}.
		\end{align}
	\item[(b)] For $\bar{\gamma} \geq 0$, define the set $
	\mathcal{C} = \Bigl\{(\sigma, \xi, \gamma) \in \mathbb{R}^3: 0\leq \sigma \leq M, -M \leq \xi \leq M, \gamma \in [0, \bar{\gamma}] \Bigr\}$.
	We have the uniform bound
		\begin{align}
		\label{term:uniform_concentration}
		\pr\Bigl\{\sup_{\mathcal{C}}\; \bigl\lvert L_n(\sigma, \xi, \gamma) - L(\sigma, \xi, \gamma)\bigr\rvert \geq \epsilon\Bigr\} \leq \frac{C(1 \vee \widebar{\gamma})^3}{\epsilon^3}\exp\Bigl\{-c\min\bigl(n\epsilon^2 \widebar{\gamma}^{-2}, n\epsilon \widebar{\gamma}^{-1}\bigr)\Bigr\}.
	\end{align}
	\end{itemize}
\end{lemma}

\subsubsection{Proof of Proposition~\ref{prop:sharp}(a)}\label{subsec:proof-prop-sharp-a}
We pursue the local stability strategy of~\citet[$\S$ 5.2]{miolane2021distribution}.  First, we construct an error set $\mathbb{D}_{\epsilon}$ as in Definition~\ref{def:error_sets} and reduce the problem to studying the minimum value of $\mathcal{L}_n$ as opposed to its minimizers; second, we use Lemma~\ref{lem:ao} to pass from  $\mathcal{L}_n$ to the auxiliary problem $\ell_n$ (and in particular its scalarized form $L_n$); third, we apply the uniform concentration bounds from Lemma~\ref{lem:costconcentration} to reduce to the study of the asymptotic loss $L$~\eqref{eq:asymploss}; and finally, we assemble the pieces to conclude.
\paragraph{Step 1: Restricting the domain to pass from minimizers to minima.}  
Let $\mathbb{D}_{\epsilon}$ be as in Definition~\ref{def:error_sets}.  Applying Lemma~\ref{lem:l2_norm_theta_hat}, we obtain the inequality
\begin{align}
	\label{ineq:prop_sharp_step2_bound1}
	\pr\bigl\{\widehat{\bt} \in \mathbb{D}_{\epsilon}\bigr\} \leq \pr\bigl\{\widehat{\bt} \in \mathbb{D}_{\epsilon} \cap \mathbb{B}_2(M_1\sqrt{p})\bigr\} + 2e^{-n}.
\end{align}
From the implication \sloppy\mbox{$
	\bigl\{\widehat{\bt} \in \mathbb{D}_{\epsilon} \cap \mathbb{B}_2(M_1\sqrt{p}) \bigr\} \implies \bigl\{\min_{\bt \in \mathbb{D}_{\epsilon}\cap \mathbb{B}_2(M_1\sqrt{p})} \mathcal{L}_n(\bt) \leq \min_{\bt \in \mathbb{B}_2(M_1\sqrt{p})}\mathcal{L}_n(\bt)\bigr\}$}, 
we deduce 
\[
\pr\bigl\{\widehat{\bt} \in \mathbb{D}_{\epsilon} \cap \mathbb{B}_2(M_1\sqrt{p})\bigr\} \leq \pr\Bigl\{\min_{\bt \in \mathbb{D}_{\epsilon} \cap \mathbb{B}_2(M_1\sqrt{p})} \mathcal{L}_n(\bt) \leq \min_{\bt \in \mathbb{B}_2(M_1\sqrt{p})} \mathcal{L}_n(\bt)\Bigr\}.
\]
Now, recall the definiton of the quantity $\mathfrak{M}_{\star}$~\eqref{eq:m-star} as the minimum of the asymptotic loss and introduce the scalar $\epsilon_1 > 0$, whose value we will set at the end of the proof.  Decomposing the RHS of the above display in terms of its intersection with the event
\sloppy\mbox{$\{\min_{\bt \in \mathbb{B}_2(M_1\sqrt{n})} \mathcal{L}_n(\bt) \leq \mathfrak{M}_{\star} + \epsilon_1\}$} and its complement, we deduce
\begin{align*}
	\pr\Bigl\{\min_{\bt \in \mathbb{D}_{\epsilon} \cap \mathbb{B}_2(M_1\sqrt{p})} \mathcal{L}_n(\bt) \leq \min_{\bt \in \mathbb{B}_2(M_1\sqrt{p})} \mathcal{L}_n(\bt)\Bigr\} &\leq \pr\left\{\min_{\bt \in \mathbb{D}_{\epsilon} \cap \mathbb{B}_2(M_1\sqrt{p})} \mathcal{L}_n(\bt) \leq \mathfrak{M}_{\star} + \epsilon_1 \right\} \nonumber\\
	& + \pr\left\{\min_{\bt \in \mathbb{B}_2(M_1\sqrt{p})} \mathcal{L}_n(\bt) \geq \mathfrak{M}_{\star} + \epsilon_1\right\}.
\end{align*}
Summarizing, 
\begin{align}
	\label{ineq:prop_sharp_step1_boundfinal}
	\hspace*{-0.2cm}\pr\bigl\{\hbt \in \mathbb{D}_{\epsilon}\bigr\} \leq \pr\left\{\min_{\bt \in \mathbb{D}_{\epsilon} \cap \mathbb{B}_2(M_1\sqrt{p})} \mathcal{L}_n(\bt) \leq \mathfrak{M}_{\star} + \epsilon_1 \right\} + \pr\left\{\min_{\bt \in \mathbb{B}_2(M_1\sqrt{p})} \mathcal{L}_n(\bt) \geq \mathfrak{M}_{\star} + \epsilon_1\right\}
\end{align}

\paragraph{Step 2: Passing to the scalarized auxiliary loss.}
Applying Lemma~\ref{lem:ao} to both terms on the RHS of the inequality~\eqref{ineq:prop_sharp_step1_boundfinal} yields
\begin{align}
		\label{ineq:prop_sharp_step2_bound1}
	\pr\bigl\{\hbt \in \mathbb{D}_{\epsilon}\bigr\} &\leq 2\pr\Bigl\{\min_{\bt \in \mathbb{D}_{\epsilon} \cap \mathbb{B}_2(M_1\sqrt{p})}\max_{\gamma \geq 0} \ell_n(\bt;\gamma) \leq \mathfrak{M}_{\star} + \epsilon_1\Bigr\} \nonumber\\
	& \qquad \qquad \qquad+ 2\pr\Bigl\{\min_{\bt \in \mathbb{B}_2(M_1\sqrt{p})} \max_{\gamma \geq 0} \ell_n(\bt;\gamma) \geq \mathfrak{M}_{\star} + \epsilon_1\Bigr\}.
\end{align}
From the change of variables~\eqref{def:shorthand-xi-sigma}, we deduce the inequality
	\begin{align}\label{ineq:before-sharpening}
	\pr\bigl\{\hbt \in \mathbb{D}_{\epsilon}\bigr\} &\leq \pr\Bigl\{\min_{(\sigma, \xi) \in \mathcal{D}_{\epsilon}(M_1)}\max_{\gamma \geq 0} L_n(\sigma, \xi, \gamma) \leq \mathfrak{M}_{\star} + \epsilon_1 \Bigr\},\nonumber \\
	&\qquad \qquad+ \pr\Biggl\{\min_{\substack{0 \leq \sigma \leq M_1/\sqrt{2} \\ \lvert \xi \rvert \leq M_1/\sqrt{2}}}\max_{\gamma \geq 0} L_n(\sigma, \xi, \gamma) \geq \mathfrak{M}_{\star} + \epsilon_1 \Biggr\}, 
	\end{align}
where we have used the shorthand $\mathcal{D}_{\epsilon}(M) := \mathcal{D}_\epsilon \cap ([0, M] \times [-M, M])$.
In order to apply the uniform bound from Lemma~\ref{lem:costconcentration}(b) in the next step, we use the following claim---whose proof we provide in Section~\ref{sec:proof-lem-constrain-gamma}---to restrict the maximization over $\gamma$ to a bounded set.  
\begin{lemma}
	\label{claim:constrain-gamma}
	There exists a tuple of positive constants $(\ubar{\gamma}, \bar{\gamma}, c, C)$---depending only on $(K_1, K_2)$---and an event $\mathcal{A}$ which holds with probability at least $1 - Ce^{-cn}$ such that on the event $\mathcal{A}$ the following hold 
		\begin{align*}
			\min_{\substack{0 \leq \sigma \leq M_1/\sqrt{2} \\ \lvert \xi \rvert \leq M_1/\sqrt{2}}}\max_{\gamma \geq 0} L_n(\sigma, \xi, \gamma) &= \min_{\substack{0 \leq \sigma \leq M_1/\sqrt{2} \\ \lvert \xi \rvert \leq M_1/\sqrt{2}}}\max_{\ubar{\gamma} \leq \gamma \leq \bar{\gamma}} L_n(\sigma, \xi, \gamma) \qquad \text{ and }\\
			\min_{(\sigma, \xi) \in \mathcal{D}_{\epsilon}(M_1)}\max_{\gamma \geq 0} L_n(\sigma, \xi, \gamma) &= \min_{(\sigma, \xi) \in \mathcal{D}_{\epsilon}(M_1)}\max_{\ubar{\gamma} \leq \gamma \leq \bar{\gamma}} L_n(\sigma, \xi, \gamma).
		\end{align*}
\end{lemma}
\noindent Applying Lemma~\ref{claim:constrain-gamma} to both terms on the RHS of the inequality~\eqref{ineq:before-sharpening} and summarizing the step, we obtain
	\begin{align}\label{ineq:after-claim}
	\pr\bigl\{\hbt \in \mathbb{D}_{\epsilon}\bigr\} &\leq  A+ B + Ce^{-cn},
	\end{align}
where we have set 
	\begin{align*}
		A &= \pr\Bigl\{\min_{(\sigma, \xi) \in \mathcal{D}_{\epsilon}(M_1)}\max_{\ubar{\gamma} \leq \gamma \leq \bar{\gamma}} L_n(\sigma, \xi, \gamma) \leq \mathfrak{M}_{\star} + \epsilon_1 \Bigr\} \quad \text{ and } \\ 
		B &= \pr\Biggl\{\min_{\substack{0 \leq \sigma \leq M_1/\sqrt{2} \\ \lvert \xi \rvert \leq M_1/\sqrt{2}}}\max_{\ubar{\gamma} \leq \gamma \leq \bar{\gamma}} L_n(\sigma, \xi, \gamma) \geq \mathfrak{M}_{\star} + \epsilon_1 \Biggr\}. 
	\end{align*}

\paragraph{Step 3: Passing from the auxiliary loss to the asymptotic loss.}
The aim of this step is to bound terms $A$ and $B$.  To this end, we introduce the following event:
\begin{align*}
	\mathcal{A}_1 = \biggl\{\sup_{0 \leq \sigma \leq M_1, \lvert \xi \rvert \leq M_1, \ubar{\gamma} \leq \gamma \leq \bar{\gamma}} \lvert L_n(\sigma, \xi, \gamma) - L(\sigma, \xi, \gamma)\rvert \leq \epsilon_2\biggr\},
\end{align*}
on which we will work for the remainder of the proof.  We handle terms $A$ and $B$ in turn.\\

\noindent \underline{Controlling the event in term $A$.}
Applying Lemma~\ref{lem:useful-asymptotic-loss}---which delineates additional properties of the asymptotic loss $L$ and can be found in Appendix~\ref{sec:properties-asymptotic}---we decompose 
\begin{align*}
	\min_{(\sigma, \xi) \in \mathcal{D}_{\epsilon}(M_1)}\max_{\ubar{\gamma} \leq \gamma \leq \bar{\gamma}} L_n(\sigma, \xi, \gamma) - \mathfrak{M}_{\star}  = T_1 + T_2,
\end{align*}
where we set
\begin{align*}
	T_1 &= \min_{(\sigma, \xi) \in \mathcal{D}_{\epsilon}(M_1)}\max_{\ubar{\gamma} \leq \gamma \leq \bar{\gamma}} L_n(\sigma, \xi, \gamma) - \min_{(\sigma, \xi) \in \mathcal{D}_{\epsilon}(M_1)}\max_{\ubar{\gamma} \leq \gamma \leq \bar{\gamma}} L(\sigma, \xi, \gamma), \qquad \text { and } \\
	T_2 &= \min_{(\sigma, \xi) \in \mathcal{D}_{\epsilon}(M_1)}\max_{\gamma \geq 0} L(\sigma, \xi, \gamma) - \mathfrak{M}_{\star}.
\end{align*}
By definition, on the event $\mathcal{A}_1$, the lower bound $T_1 \geq -\epsilon_2$ holds.  
Additionally, we apply Lemma~\ref{lem:structural_L}(a)---which establishes strong convexity of the map \sloppy\mbox{$(\sigma, \xi) \mapsto \max_{\gamma \geq 0}\; L(\sigma, \xi, \gamma)$}---to obtain the lower bound $T_2 \geq \lambda(1 \wedge R)/2 \cdot \epsilon^2$.  Putting the pieces together, we obtain the inequality
\begin{align}
	\label{ineq:bound-T1-T3}
	\min_{(\sigma, \xi) \in \mathcal{D}_{\epsilon}(M_1)}\max_{\ubar{\gamma} \leq \gamma \leq \bar{\gamma}} L_n(\sigma, \xi, \gamma) - \mathfrak{M}_{\star}  \geq \frac{\lambda(1 \wedge R)}{2} \cdot \epsilon^2 - \epsilon_2.
\end{align}

\noindent \underline{Controlling the event in term $B$.}
Similarly, by Lemma~\ref{lem:useful-asymptotic-loss}(b), (c), and (d), 
\begin{align*}
	\min_{\substack{0 \leq \sigma \leq M_1/\sqrt{2}\\ \lvert \xi \rvert \leq M_1/\sqrt{2}}} \max_{\ubar{\gamma} \leq \gamma \leq \bar{\gamma}} L_n(\sigma, \xi, \gamma) - \mathfrak{M}_{\star} = \min_{\substack{0 \leq \sigma \leq M_1/\sqrt{2}\\ \lvert \xi \rvert \leq M_1/\sqrt{2}}} \max_{\ubar{\gamma} \leq \gamma \leq \bar{\gamma}} L_n(\sigma, \xi, \gamma) -\min_{\substack{0 \leq \sigma \leq M_1/\sqrt{2}\\ \lvert \xi \rvert \leq M_1/\sqrt{2}}} \max_{\ubar{\gamma} \leq \gamma \leq \bar{\gamma}} L(\sigma, \xi, \gamma),
\end{align*}
so that on the event $\mathcal{A}_1$,
\begin{align}
	\label{ineq:bound-T4-T6}
	\min_{\substack{0 \leq \sigma \leq M_1/\sqrt{2}\\ \lvert \xi \rvert \leq M_1/\sqrt{2}}} \max_{\ubar{\gamma} \leq \gamma \leq \bar{\gamma}} L_n(\sigma, \xi, \gamma) - \mathfrak{M}_{\star} \leq \epsilon_2.
\end{align}

\noindent \underline{Concluding step 3.}  Set $\epsilon_2 = \epsilon_1 = \lambda (1 \wedge R)/4 \cdot \epsilon^2$.  We combine this with the inequalities~\eqref{ineq:bound-T1-T3} and~\eqref{ineq:bound-T4-T6} to deduce that on the event $\mathcal{A}_1$, 
\begin{align*}
	\min_{(\sigma, \xi) \in \mathcal{D}_{\epsilon}(M_1)}\max_{\ubar{\gamma} \leq \gamma \leq \bar{\gamma}} L_n(\sigma, \xi, \gamma) - \mathfrak{M}_{\star} \geq \epsilon_1 \qquad \text{ and } \qquad \min_{\substack{0 \leq \sigma \leq M_1/\sqrt{2}\\ \lvert \xi \rvert \leq M_1/\sqrt{2}}} \max_{\ubar{\gamma} \leq \gamma \leq \bar{\gamma}} L_n(\sigma, \xi, \gamma) - \mathfrak{M}_{\star} \leq \epsilon_1.
\end{align*}
Since both of the above inequalities hold on event $\mathcal{A}_1$, we deduce the inequality $A \vee B \leq \pr\{\mathcal{A}_1^{c}\}$.  To complete the step, we invoke Lemma~\ref{lem:costconcentration}(b) to obtain the inequality \sloppy\mbox{$\pr\{\mathcal{A}_1^{c}\} \leq \frac{C}{\epsilon^6} \exp\bigl\{-c\min(n\epsilon^4, n\epsilon^2)\bigr\}$}.  Thus, 
\begin{align}
	\label{ineq:bound-A-B-final}
	A \vee B \leq \frac{C}{\epsilon^6} \exp\bigl\{-c\min(n\epsilon^4, n\epsilon^2)\bigr\}.
\end{align}

\paragraph{Step 4: Putting the pieces together.}
Combining the inequalities~\eqref{ineq:after-claim} and~\eqref{ineq:bound-A-B-final} yields
\begin{align*} \label{ineq:prop-sharp-penultimate}
	\pr\bigl\{\hbt \in \mathbb{D}_{\epsilon}\bigr\} \leq \frac{C}{\epsilon^6} \exp\bigl\{-c\min(n\epsilon^4, n\epsilon^2)\bigr\} + Ce^{-cn}.
\end{align*}
Thus, since
\[
\mathbb{D}_{\epsilon}^c = \Bigl\{\bt \in \mathbb{R}^p: \max\bigl(\lvert \xi(\bt) - \xi_{\star} \rvert,   \lvert \sigma(\bt)- \sigma_{\star} \rvert \bigr) \leq \epsilon  \Bigr\},
\]
the desired result follows from Lipschitz continuity of the function $\phi$. \qed

\section{Discussion}\label{sec:discussion}
We studied high-dimensional logistic regression in a simple setting in which the data matrix consists of i.i.d. random variables and the mechanism by which data is missing is MCAR (missing completely at random).  Contrasting with the high dimensional linear model, we demonstrated that in the logistic model, single imputation may result in an inconsistent estimator in mean squared error.  On the other hand, we showed---relying on Conjecture~\ref{conj:bayes-conj}---that in this simple setting, single imputation yields an estimator with optimal prediction error.  We believe our results comprise compelling evidence that our understanding of imputation-based methodology remains incomplete, especially in high-dimensions.  

Several intriguing open questions remain and we detail a few here.  First, proving Conjecture~\ref{conj:bayes-conj} would solidify the ground on which our observations stand.  Second, even in the simple MCAR setting considered here, it would be of great interest to move beyond i.i.d. covariates to understand the effect of missing data under more realistic assumptions on the data.  Recent techniques~\citep{celentano2020lasso,montanari2022empirical} have been developed to study the Gaussian ensemble with covariance structure as well as universality beyond i.i.d. covariates and may prove useful in this endeavor.  Turning to the missingness mechanism, the MCAR assumption considered here is far too strong to reflect practical situations in which the mechanism can, in general, be far more complicated.  One concrete and interesting direction would be to study the setting where the event that each feature is missing remains independent, but the probability with which each feature is missing may differ for each feature.  Finally, the imputation method of choice in practice is multiple imputation, which is beyond the reach of the theory we develop.  It would be extremely interesting to understand, even in the stylized models considered here, the precise tradeoffs between single and multiple imputation in high-dimensions.

%

\subsection*{Acknowledgements}
 K.A.V.\ was supported in part by a National Science Foundation Graduate Research Fellowship, the Sony Stanford Graduate Fellowship, and by European Research Council Advanced Grant 101019498.  K.A.V.\ would like to thank Saminul Haque, Ashwin Pananjady, and Richard Samworth for feedback on an earlier version of this manuscript.

\bibliographystyle{plainnat}
\addcontentsline{toc}{section}{References}

\providecommand{\bysame}{\leavevmode\hbox to3em{\hrulefill}\thinspace}
\providecommand{\MR}{\relax\ifhmode\unskip\space\fi MR }
\providecommand{\MRhref}[2]{%
	\href{http://www.ams.org/mathscinet-getitem?mr=#1}{#2}
}
\providecommand{\href}[2]{#2}
\bibliography{refs/ref,refs/lasso,refs/imputation}

\appendix 

\section{Properties of the asymptotic loss} \label{sec:properties-asymptotic}
We begin by stating the following lemma which delineates several useful properties of the asymptotic loss $L$~\eqref{eq:asymploss}.   We provide its proof in Section~\ref{sec:proof-lem-useful-asymptotic-loss}.
\begin{lemma}\label{lem:useful-asymptotic-loss}
	 Under the setting of Lemma~\ref{lem:structural_L}, let $M$ denote a positive scalar.  There exists a tuple of strictly positive constants $(\ubar{\gamma}_M, c_M, c_M', C_M)$, depending only on $M$ and the problem parameters of Assumption~\ref{asm:regularity} such that the following hold.
	\begin{itemize}
		\item[(a)] For all $(\sigma, \xi) \in [0, \infty) \times \mathbb{R}$, the map $\gamma \mapsto L(\sigma, \xi, \gamma)$ is strictly concave on the domain $\gamma \in [0, \infty)$.
		\item[(b)] For all pairs $(\sigma, \xi) \in [0, M] \times [-M, M]$, $\argmax_{\gamma \geq 0} L(\sigma, \xi, \gamma) \geq \ubar{\gamma}_M$.
		\item[(c)] For all pairs $(\sigma, \xi) \in [0, M] \times [-M, M]$,
		\[
		\frac{c_M}{\sigma} \leq \argmax_{\gamma \in [0, \infty)} \; L(\sigma, \xi, \gamma) \leq \frac{C_M}{\sigma}.
		\]
		\item[(d)] For any $\xi \in [-M, M]$, $\argmin_{\sigma \geq 0} \max_{\gamma \geq 0} L(\sigma, \xi, \sigma) \geq c_M'$. 
		\item[(e)] Let $L_1$ and $U_1$ denote fixed positive constants and consider the set $\mathcal{C} = \{(\sigma, \xi, \gamma) \in \mathbb{R}^3: \sigma \in [0, M], \xi \in [-M, M], \gamma \in [L_1, U_1]\}$.  There exists a constant $C_{M, L_1, U_1}$ depending only on the problem parameters, $M, L_1$, and $U_1$, such that the function $L$ is $C_{M, L_1, U_1}$--Lipschitz over the domain $\mathcal{C}$.
	\end{itemize}
\end{lemma}
Equipped with this lemma, we prove each part of Lemma~\ref{lem:structural_L} in turn.  
\paragraph{Proof of Lemma~\ref{lem:structural_L}(a)}
We decompose $\Psi$ into a quadratic component and a non-quadratic component $\Phi$:
\[
\Psi(\sigma, \xi) =  \frac{\lambda(\xi^2 R^2 + \sigma^2)}{2}  + \underbrace{\max_{\gamma \geq 0} \Bigl\{ - \frac{\alpha_2 \gamma \sigma^2}{2\delta} + \E\Bigl\{\min_{u \in \mathbb{R}}\Bigl[ \rho(-Yu) + \frac{\gamma}{2} \cdot \bigl(u - V(Z_1, Z_2)\bigr)^2\Bigr]\Bigr\} \Bigr\} }_{\Phi(\sigma, \xi)}.
\]
It suffices to show that for any $M > 0$, $\Phi$ is convex on the domain $[0, M] \times [-M, M]$; accordingly, the remainder of the proof is dedicated to establishing this.
Before proceeding, we note that by Lemma~\ref{lem:useful-asymptotic-loss}(b), there exists a constant $\ubar{\gamma}_M$ such that $\Phi$ admits the equivalent characterization
\[
\Phi(\sigma, \xi) = \max_{\gamma \geq \ubar{\gamma}_M} \Bigl\{ - \frac{\alpha_2 \gamma \sigma^2}{2\delta} + \E\Bigl\{\min_{u \in \mathbb{R}}\Bigl[ \rho(-Yu) + \frac{\gamma}{2} \cdot \bigl(u - V(Z_1, Z_2)\bigr)^2\Bigr]\Bigr\} \Bigr\}.
\]
Continuing, we adopt the strategy of~\citet[Propostion 5.1]{Montanari2019} and pass to an infinite dimensional optimization problem.  To this end, let $\mathbb{X}$ denote the set of all measurable functions from $\mathbb{R}^3 \rightarrow \mathbb{R}$, $\mathbb{Q}$ denote the joint distribution of the triplet $(Z_1, Z_2, Y)$ and $\mathbb{L} = L^2(\mathbb{Q})$ denote the space of square integrable (with respect to $\mathbb{Q}$) functions.  Next, for $\gamma \geq 0$, define the function $g_{\gamma}: \mathbb{X} \rightarrow \mathbb{R}$ as $g_{\gamma}(U; \sigma, \xi) = \mathbb{E}[\rho(-YU)] + \frac{\gamma}{2} \mathbb{E}\bigl(U - V(Z_1, Z_2)\bigr)^2$.  Note that, for $\gamma \geq \ubar{\gamma}$, $g_{\gamma}$ is infinite on $\mathbb{X} \setminus \mathbb{L}$ and, on $\mathbb{L}$, $g_{\gamma}$ is lower semicontinuous.  Moreover, since, for $\gamma \geq \ubar{\gamma}$, $g_{\gamma}$ is strongly convex, it admits a unique minimizer in $\mathbb{L}$~\citep[][Corollary 11.17]{bauschke2017convex}.  Hence, 
\begin{align*}
	\Phi(\sigma, \xi)&= \max_{\gamma \geq \ubar{\gamma}_M}\; \biggl\{-\frac{\alpha_2 \gamma \sigma}{2 \delta} + \inf_{U \in \mathbb{X}}\Bigl[ \E\{\rho(-Y U)\} + \frac{\gamma}{2}\mathbb{E}\bigl(U - V(Z_1, Z_2)\bigr)^2\Bigr]\biggr\}\\
	& = \max_{\gamma \geq \ubar{\gamma}_M}\; \biggl\{-\frac{\alpha_2 \gamma \sigma}{2 \delta} + \min_{U \in \mathbb{L}}\Bigl[ \E\{\rho(-Y U)\} + \frac{\gamma}{2}\mathbb{E}\bigl(U - V(Z_1, Z_2)\bigr)^2\Bigr]\biggr\}.
\end{align*}
Now, let $U_{\gamma}(\sigma, \xi) = \argmin_{U\in \mathbb{L}} \{g_{\gamma}(U; \sigma, \xi)\}$.  Next, consider the points $(\sigma_1, \xi_1) \in [0, M] \times [-M, M]$ and $(\sigma_2, \xi_2) \in [0, M] \times [-M, M]$ and, for each $\gamma \geq \ubar{\gamma}_M$, associate to them the minimizers 
\[
U_{\gamma}^{(1)} = \argmin_{U\in \mathbb{L}} \{g_{\gamma}(U; \sigma_1, \xi_1)\}, \qquad \text{ and } \qquad U^{(2)}_{\gamma} = \argmin_{U \in \mathbb{L}} \; \{g_{\gamma}(U; \sigma_2, \xi_2)\}.
\]
Additionally, consider the convex combinations
\[
(\sigma_t, \xi_t) = (t \sigma_1 + (1 - t) \sigma_2, t\xi_1 + (1 - t) \xi_2), \qquad \text{ and } \qquad U^{(t)}_{\gamma} = tU^{(1)}_{\gamma} + (1 - t) U^{(2)}_{\gamma}.
\]
We conclude by noting the inequality
\begin{align*}
	\Phi(\sigma_t, \xi_t) &\overset{\1}{\leq} \max_{\gamma \geq \ubar{\gamma}} \biggl\{-\frac{\alpha_2 \gamma \sigma_t}{2\delta} + g_{\gamma}\bigl(U_{\gamma}^{(t)}; \sigma_t, \xi_t\bigr) \biggr\}\\
	&\overset{\2}{\leq} \max_{\gamma \geq \ubar{\gamma}} \biggl\{-\frac{t\alpha_2 \gamma \sigma_1}{2\delta} + tg_{\gamma}\bigl(U_{\gamma}^{(1)}; \sigma_1, \xi_1\bigr) +  \frac{(t-1)\alpha_2 \gamma \sigma_2}{2\delta} + (1 - t)g_{\gamma}\bigl(U_{\gamma}^{(2)}; \sigma_2, \xi_2\bigr)\biggr\}\\
	&\overset{\3}{\leq} t \cdot \Phi(\sigma_1, \xi_1) + (1 - t) \cdot \Phi (\sigma_2, \xi_2),
\end{align*}
where step $\1$ follows since $U^{(t)}_{\gamma} \in \mathbb{L}$, step $\2$ follows from convexity of the function $\rho$ and the squared norm $\| \cdot \|_{\mathbb{L}}^2$, and step $\3$ follows by definition of $U^{(1)}_{\gamma}$ and $U^{(2)}_{\gamma}$.  Convexity of $\Phi$ follows by definition. \qed

\paragraph{Proof of Lemma~\ref{lem:structural_L}(b)}
We show that there exists a constant $M$ such that it suffices to restrict the minimization to components with magnitude at most $M$.  Once this is established, we apply Lemma~\ref{lem:useful-asymptotic-loss}(b) and set $\gamma_0 = \ubar{\gamma}_M$, and subsequently combine Lemmas~\ref{lem:useful-asymptotic-loss}(d) and (c) in sequence to conclude.  
By definition, we obtain the upper bound
\[
\Psi(0, 0) = \Phi(0, 0) = \max_{\gamma \geq 0} \EE\Bigl\{\min_{u \in \mathbb{R}} \rho(-Y u) + \frac{\gamma}{2} u^2\Bigr\} \leq \log{2},
\]
where the final inequality follows by taking $u = 0$ in the inner minimization.  On the other hand, taking $\gamma = 0$ implies that for any $\sigma, \xi$, $\Phi(\sigma, \xi) \geq 0$, whence we obtain the lower bound $\Psi(\sigma, \xi) \geq \frac{\lambda (\xi^2 R^2 + \sigma^2)}{2}$.  Subsequently, 
set $M = 2\sqrt{\frac{\log{2}}{\lambda (R^2 \wedge 1)}}$.
Thus, for any $\sigma \geq M$ or $\lvert \xi \rvert \geq M$, the conclusion follows from the chain of inequalities $
\Psi(\sigma, \xi) > \log{2} \geq \Psi(0, 0)$. 
\qed

\paragraph{Proof of Lemma~\ref{lem:structural_L}(c)}
First, note that by definition
\[
\min_{\sigma \geq 0, \xi \in \mathbb{R}}\max_{\gamma \geq 0}\; L(\sigma, \xi, \gamma) = \min_{\sigma \geq 0, \xi \in \mathbb{R}}\; \Psi(\sigma, \xi).
\]
By part (a), $\Psi$ is strongly convex and consequently admits a unique minimizer $(\sigma_{\star}, \gamma_{\star})$.  Moreover, by Lemma~\ref{lem:useful-asymptotic-loss}(a), the function $\gamma \mapsto L(\sigma_{\star}, \xi_{\star}, \gamma)$ is strictly concave and thus admits a unique minimizer $\gamma_{\star}$.  To conclude, note that the function $L$ is continuously differentiable, strongly convex in its first two arguments and strictly concave in its last argument, whence the triple $(\sigma_{\star}, \xi_{\star}, \gamma_{\star})$ is identified by the first-order stationary conditions.  Straightforward calculation verifies that the system of equations~\eqref{eq:system} correspond to the first-order stationary conditions. \qed

\subsection{Proof of Lemma~\ref{lem:useful-asymptotic-loss}}\label{sec:proof-lem-useful-asymptotic-loss}
We note that the proofs of parts (a), (b), and (c) follow in an analogous manner to parts (a) and (b) of Lemma~\ref{lem:useful-L-n} (see Section~\ref{sec:proof-lem-useful-L-n}), so for brevity we omit their proofs.  

\paragraph{Proof of Lemma~\ref{lem:useful-asymptotic-loss}(d)}
Let $\gamma(\sigma, \xi)$ denote the maximizer of the map $\gamma \mapsto L(\sigma, \xi, \gamma)$ over the domain $\gamma \in [0, \infty)$ and let $\widehat{u} = \prox_{\rho(-Y\cdot)}(\xi R \sqrt{\alpha_2} Z_1 + \sigma\sqrt{\alpha_2}Z_2; \gamma)$.  We apply Lemma~\ref{lem:derivative_prox} to compute the partial derivative directly, obtaining
\begin{align*}
	\frac{\partial}{\partial \sigma} \Bigl\{ \max_{\gamma \geq 0}\; L(\sigma, \xi, \gamma)\Bigr\} &= \lambda \sigma - \frac{\alpha_2 \gamma(\sigma, \xi) \sigma}{\delta} - \sqrt{\alpha_2} \E\{Y \rho'(-Y \widehat{u}) Z_2\}\\
	&= \lambda \sigma - \frac{\alpha_2 \gamma(\sigma, \xi) \sigma}{\delta} + \E\biggl\{\frac{\alpha_2 \sigma \gamma(\sigma, \xi)}{\gamma(\sigma, \xi) + \rho''(-Y \widehat{u})}\biggr\},
\end{align*}
where the final step follows from Gaussian integration by parts~\citep[Lemma 7.2.3]{vershynin2018high}.
In turn, applying Lemma~\ref{lem:useful-asymptotic-loss}(c) yields
\[
\frac{\partial}{\partial \sigma} \Bigl\{ \max_{\gamma \geq 0}\; L(\sigma, \xi, \gamma)\Bigr\} \leq \lambda \sigma - \frac{\alpha_2 c_M}{\delta} + \E\biggl\{\frac{\alpha_2 \sigma \gamma(\sigma, \xi)}{\gamma(\sigma, \xi) + \rho''(-Y \widehat{u})}\biggr\} \leq  \lambda \sigma - \frac{\alpha_2 c_M}{\delta} + \alpha_2 \sigma,
\]
where in the final inequality we used the bound $\gamma(\sigma, \xi)/(\gamma(\sigma, \xi) + \rho''(-Y \widehat{u})) \leq 1$.  Consequently,
\begin{align}\label{ineq:outcome-case2}
	\frac{\partial}{\partial \sigma} \Bigl\{ \max_{\gamma \geq 0}\; L(\sigma, \xi, \gamma)\Bigr\} \leq - \frac{\alpha_2 c_M}{2\delta}, \qquad \text{ for all } \sigma \leq  \frac{\alpha_2 c_M}{2\delta}.
\end{align}
We conclude by setting $c_M' = \alpha_2 c_M/(2\delta)$. \qed

\paragraph{Proof of Lemma~\ref{lem:useful-asymptotic-loss}(e)}
We bound the partial derivatives of $L$~\eqref{eq:asymploss} with respect to each of the variables $\sigma, \xi$ and $\gamma$ in turn.

\smallskip
\noindent \underline{Bounding the partial derivative with respect to $\sigma$.}
Let $
	X^{(1)} = -\sqrt{\alpha_2} Y \rho'(-Y \widehat{u})Z_{2}$.
Subsequently, apply Lemma~\ref{lem:derivative_prox}(b) in conjunction with the triangle inequality to obtain
\begin{align*}
	\Bigl \lvert \frac{\partial}{\partial \sigma} L(\sigma, \xi, \gamma)\Bigr\rvert = \Bigl \lvert \lambda \sigma - \frac{\alpha_2 \gamma \sigma}{\delta} - \E \{ X^{(1)}\} \Bigr \rvert \leq \lambda \sigma + \frac{\alpha_2 U_1 \sigma}{\delta} + \lvert \E \{X^{(1)}\} \rvert.
\end{align*}
Applying Jensen's inequality, we deduce the inequality
\begin{align}
	\label{ineq:bound-expec-X1}
	\lvert \E\{ X ^{(1)}\} \rvert \leq \E\{ \lvert X^{(1)} \rvert\} \leq \sqrt{\alpha_2}\E \{\lvert Z_2 \rvert \} \leq C.
\end{align}
Combining the two previous displays with the assumption $\sigma \leq M$ yields the upper bound
$\sup_{(\sigma, \xi, \gamma) \in \mathcal{C}}\bigl \lvert \frac{\partial}{\partial \sigma} L(\sigma, \xi, \gamma)\bigr\rvert \leq C (1 \vee U_1).$

\smallskip
\noindent \underline{Bounding the partial derivative with respect to $\xi$.}
Let $X^{(2)} = -\sqrt{\alpha_2} R Y \rho'(-Y \widehat{u})Z_{1}$,
and apply Lemma~\ref{lem:derivative_prox} followed by the triangle inequality to obtain the upper bound \sloppy\mbox{$
	\bigl \lvert \frac{\partial}{\partial \xi} L(\sigma, \xi, \gamma)\bigr\rvert \leq \lambda R^2 \lvert \xi \rvert + \lvert \E \{X^{(2)}\} \rvert$.}
Following identical steps as in the chain of inequalities~\eqref{ineq:bound-expec-X1} yields the bound $\lvert \E\{ X^{(2)}\} \rvert \leq C$.  Consequently, we deduce the desired upper bound$
	\sup_{(\sigma, \xi, \gamma) \in \mathcal{C}} \bigl \lvert \frac{\partial}{\partial \xi} L(\sigma, \xi, \gamma)\bigr\rvert \leq C$.

\smallskip
\noindent \underline{Bounding the partial derivative with respect to $\gamma$.}
We define $
	X^{(3)} = \frac{1}{2\gamma^2} \rho'(-Y\widehat{u})$,
and apply Lemma~\ref{lem:derivative_prox}(c) followed by the triangular inequality, to obtain the upper bound $
	\bigl \lvert \frac{\partial}{\partial \gamma} L(\sigma, \xi, \gamma)\bigr\rvert \leq \frac{\alpha_2 \sigma^2}{2\delta} + \lvert \E \{X^{(3)}\} \rvert$.
Following the same logic as in the chain of inequalities~\eqref{ineq:bound-expec-X1}, we deduce $
\lvert \E\{X^{(3)}\} \rvert \leq \frac{1}{2L_1^2} \leq C
$.  Consequently, we obtain the uniform bound \sloppy\mbox{$
	\sup_{(\sigma, \xi, \gamma) \in \mathcal{C}} \bigl \lvert \frac{\partial}{\partial \gamma} L(\sigma, \xi, \gamma)\bigr\rvert \leq C$.}

\smallskip
\noindent \underline{Putting the pieces together.} Combining the uniform upper bounds, we immediately obtain $
\sup_{(\sigma, \xi, \gamma) \in \mathcal{C}} \norm{\nabla L(\sigma, \xi, \gamma)}_2 \leq C(1 \vee \widebar{\gamma})$, 
which implies the result.

\section{Auxiliary proofs for the error-in-variables model}
This appendix is organized as follows: in Section~\ref{subsubsec:proof_lemma_ao}, we prove Lemma~\ref{lem:ao}; in Section~\ref{subsubsec:proof_cost_concentration}, we prove Lemma~\ref{lem:costconcentration}; and in Section~\ref{sec:proof-lem-useful-L-n}, we prove Lemma~\ref{lem:useful-L-n}.

\subsection{Proof of Lemma~\ref{lem:ao}}
\label{subsubsec:proof_lemma_ao}
The proof follows two steps: First, we apply the CGMT~\citep{thrampoulidis2018} to relate $\mathcal{L}_n$~\eqref{eq:loss} to a simpler problem; and second, we show that with high probability, this simpler problem coincides with the auxiliary loss $\ell_n$~\eqref{def:aux_loss}.

\paragraph{Step 1: Reduction to the auxiliary problem via the CGMT.}
In order to apply the CGMT, we must express the minimization of $\mathcal{L}_n$~\eqref{eq:loss} as a variational problem over convex, compact sets.  To this end, define the function $F_n: \mathbb{R}^n \times \mathbb{R}^p \rightarrow \mathbb{R}$ as 
\begin{align}\label{eq:define-F_n}
	F_n(\bu, \bt) = \frac{1}{n} \sum_{i=1}^{n} \rho(u_i) + \frac{\lambda}{2p} \| \bt \|_2^2. 
\end{align}
Following the proof of~\citet[Corollary 5.1]{miolane2021distribution}, we note that by tightness, for any $\varepsilon > 0$, there exists a constant $M$, depending on $\varepsilon$, such that the event
\[
\mathcal{A} = \Bigl\{ \|\bg \|_2 \leq M \quad \text{ and } \quad  \| \bh \|_2 \leq M\Bigr\} \bigcap  \Bigl\{ \| \bG \|_{\mathsf{op}} \leq M \quad \text{ and } \quad \| \widehat{\bt} \|_2 \leq M\Bigr\},
\]
satisfies $\pr\{\mathcal{A}\} \geq 1 - \varepsilon$.  On the event $\mathcal{A}$, we note the relation
\[
\min_{\bt \in D}\; \mathcal{L}_n(\bt) \;= \min_{\substack{\bt \in D \\ \bu \in \mathbb{B}_2(M)}} \Bigl\{F_n(-\by \odot \bu; \bt) \quad \text{ s.t. } \quad \bG \bt = \bu\Bigr\}.
\]
A straightforward calculation---which we provide at the end of the subsection---establishes that $\bu \mapsto F_n(\bu; \bt)$ is $n^{-1/2}$--Lipschitz.  Consequently, for any $s > 0$, we obtain the inequality
\[
\max_{\| \bv \|_2 = s} \Bigl\{F_n(-\by \odot \bu; \bt) + \langle \bv, \bG \bt - \bu \rangle \Bigr\} \geq F_n(-\by \odot \bG \bt; \bt) + \left(s - \frac{1}{\sqrt{n}}\right)  \norm{\bu - \bG \bt}_2.
\]
Evidently, in order to enforce the constraint $\bG \bt = \bu$, it is unnecessary to maximize over $\mathbb{R}^n$; the constraint is enforced by maximizing $\bv \in \mathbb{B}_2(r^{-1}n^{-1/2})$, for any $r$ satisfies $0 < r < 1$.  Thus, on the event $\mathcal{A}$, 
\begin{align}
	\min_{\bt \in D}\; \mathcal{L}_n(\bt)\ &=  \min_{\substack{\bt \in D \\ \bu \in \mathbb{B}_2(M)}}\max_{\bv \in \mathbb{B}_2(r^{-1}n^{-1/2})} \Bigl\{F_n(-\by \odot \bu; \bt)+ \langle \bv, \bG \bt - \bu \rangle \Bigr\}.\nonumber
\end{align}
Having written the minimization of $\mathcal{L}_n$ as a variational problem over compact sets, it remains only to handle the correlation between the labels $\by$ and the data $\bG$.  To this end, we consider the orthogonal decomposition $\bt = \proj_{\bt_0} \bt + \pproj \bt$ and define the intermediate loss
\begin{align}
	\widetilde{\mathcal{L}}_n(\bt) := &\min_{\bu \in \mathbb{B}_2(M)}\max_{\bv \in \mathbb{B}_2(r^{-1}n^{-1/2})} \Bigl\{F_n(-\by \odot \bu; \bt) + \langle \bv, 
	\widetilde{\bG} \proj_{\bt_0}\bt\rangle + \langle \bv, \widetilde{\bG} \proj^{\perp}_{\bt_0}\bt\rangle - \langle \bv, \bu \rangle\Bigr\},\nonumber,
\end{align}
where $\widetilde{\bG}$ is an independent copy of $\bG$.  Note that on the event $\mathcal{A}$, we have the distributional equality $\min_{\bt \in D}\; \mathcal{L}_n(\bt) \overset{d}{=} \min_{\bt \in D} \;\widetilde{\mathcal{L}}_n(\bt)$.  Thus, defining $\widetilde{\ell}_n: \mathbb{R}^p  \rightarrow \mathbb{R}$ as
\begin{align*}
	\widetilde{\ell}_n(\bt)\ :=  \min_{\bu \in \mathbb{B}_2(M)}\max_{\bv \in \mathbb{B}_2(r^{-1}n^{-1/2})} \Bigl\{&F_n(-\by \odot \bu; \bt)+ \langle \bv, \bG \proj_{\bt_0}\bt\rangle \nonumber\\
	&+ \sqrt{\frac{\alpha_2}{p}}\norm{\bv}_2 \langle \bg, \proj_{\bt_0}^{\perp}\bt \rangle + \sqrt{\frac{\alpha_2}{p}}\bigl \|\proj_{\bt_0}^{\perp}\bt \bigr \|_2\langle \bh, \bv \rangle - \langle \bv, \bu \rangle\Bigr\},
\end{align*}
and applying the CGMT~\citep[see, e.g.,][Theorem 3]{thrampoulidis2018} yields 
\begin{align} \label{ineq:step1-cgmt-proof}
\pr\Bigl\{ \min_{\bt \in D}\; \mathcal{L}_n(\bt) \leq t\Bigr\} \leq 2\pr\Bigl\{ \min_{\bt \in D}\; \widetilde{\ell}_n(\bt) \leq t\Bigr\} + \varepsilon.
\end{align}

\paragraph{Step 2: Passing from the intermediate problem $\widetilde{\ell}_n$ to $\ell_n$~\eqref{def:aux_loss}.}
Collecting terms and applying the Cauchy--Schwarz inequality yields
\begin{align*}
	\widetilde{\ell}_n(\bt)\ =  \min_{\bu \in \mathbb{B}_2(M)}\max_{0 \leq \gamma \leq r^{-1}} \Bigl\{F_n(-\by \odot \bu; \bt)+
	\frac{\gamma}{\sqrt{n}} \Bigl( \sqrt{\frac{\alpha_2}{p}}\langle \bg, \pproj \bt \rangle + \Bigl \| \sqrt{\frac{\alpha_2}{p}}\|\pproj \bt\|_2 \bh + \bG \proj_{\bt_0}\bt - \bu \Bigr\|_2\Bigr)\Bigr\}.
\end{align*}
Now, note that the RHS of the above display depends on the direction of the vector $\pproj \bt$ only through the term $\gamma \sqrt{\alpha_2/np} \langle \bg, \pproj \bt\rangle$.  Thus, for sets $D$ of the form $\mathbb{B}_2(C)$ or $\mathbb{D}_{\epsilon} \cap \mathbb{B}_2(C)$, we may first minimize over the direction of $\pproj \bt$~\citep[by, e.g.,][Lemma 8]{kammoun2021precise} to obtain 
\begin{align}
		\label{eq:opt-perp-theta}
	\min_{\bt \in D}\; \widetilde{\ell}_n(\bt) &=  \min_{\bt \in D}\min_{\bu \in \mathbb{B}_2(M)}\max_{0 \leq \gamma \leq \frac{1}{r}} \Bigl\{F_n(-\by \odot \bu; \bt) \nonumber\\
	&\qquad + \frac{\gamma}{\sqrt{n}} \Bigl( -\sqrt{\frac{\alpha_2}{p}} \|\pproj \bg\|_2 \| \pproj \bt \|_2 + \Bigl \| \sqrt{\frac{\alpha_2}{p}}\|\pproj \bt\|_2 \bh + \bG \proj_{\bt_0}\bt - \bu \Bigr\|_2\Bigr)\Bigr\}.
\end{align}
Next, consider the ball 
\begin{align}\label{eq:ball-proof-gordon}
B := \Bigl\{\bu \in \mathbb{R}^d:  \Bigl \|\bu - \sqrt{\frac{\alpha_2}{p}}\|\pproj \bt\|_2 \bh - \bG \proj_{\bt_0}\bt  \Bigr\|_2 \leq  \sqrt{\frac{\alpha_2}{p}}  \|\pproj \bg\|_2 \| \pproj \bt \|_2 \Bigr\},
\end{align}
and note that if the inner maximization over $\gamma$ were over the constraint $\gamma \geq 0$, the problem~\eqref{eq:opt-perp-theta} would be equivalent to imposing the further constraint $\bu \in B$. 

  For any $\bu \in \mathbb{B}_2(M)$, let $\widebar{\bu}$ denote its projection onto $B$. Since $F_n$ is Lipschitz continuous with constant $1/\sqrt{n}$ and $\widebar{\bu}$ lies on the boundary of the ball $B$, 
\begin{align*}
F_n(-\by \odot \bu; \bt) + \frac{\gamma}{\sqrt{n}} \Bigl( -\sqrt{\frac{\alpha_2}{p}} \|\pproj \bg\|_2 \| \pproj \bt \|_2 &+ \Bigl \| \sqrt{\frac{\alpha_2}{p}}\|\pproj \bt\|_2 \bh + \bG \proj_{\bt_0}\bt - \bu \Bigr\|_2\Bigr) \\
& \qquad \qquad \qquad \geq F_n(-\by \odot \widebar{\bu}; \bt) + \frac{\gamma - 1}{\sqrt{n}} \| \bu - \widebar{\bu} \|_2.
\end{align*}
Consequently, the constraint is enforced as long as $r < 1$ and we deduce 
\begin{align*}
	\min_{\bt \in D}\; \widetilde{\ell}_n(\bt) &=  \min_{\bt \in D}\min_{\bu \in \mathbb{B}_2(M)}\max_{\gamma \geq 0} \Bigl\{F_n(-\by \odot \bu; \bt) \\
	&\qquad + \frac{\gamma}{n} \Bigl( -\frac{\alpha_2}{p} \|\pproj \bg\|_2^2 \| \pproj \bt \|_2^2 + \Bigl \| \sqrt{\frac{\alpha_2}{p}}\|\pproj \bt\|_2 \bh + \bG \proj_{\bt_0}\bt - \bu \Bigr\|_2^2\Bigr)\Bigr\},
\end{align*}
where we have additionally used the fact that squaring both sides of the inequality in the definition of the ball $B$~\eqref{eq:ball-proof-gordon} forms an equivalent definition.  Applying Sion's minimax inequality~\citep[Corollary 3.3]{sion1958general} to interchange minimization over $\bu$ and maximization over $\gamma$ and expanding $F_n$ yields
\begin{align*}
	\min_{\bt \in D}\; \widetilde{\ell}_n(\bt) &=  \min_{\bt \in D} \max_{\gamma \geq 0}  \Bigl\{ \frac{\lambda}{2p} \| \bt \|_2^2 -\frac{\gamma \alpha_2}{2np} \|\pproj \bg\|_2^2 \| \pproj \bt \|_2^2 \\
	&+ \min_{\bu \in \mathbb{B}_2(M)} \Bigl[\frac{1}{n} \sum_{i=1}^{n} \rho(-y_i u_i) + \frac{\gamma}{2} \Bigl(u_i - (\bG\proj_{\bt_0} \bt)_i - \frac{\sqrt{\alpha_2} \| \pproj \bt \|_2}{\sqrt{p}} h_i\Bigr)^2 \Bigr]\Bigr\}.
\end{align*}
Now, on the event $\mathcal{A}$, the quantities $(\bG\proj_{\bt_0} \bt)_i$ and $h_i$ are bounded for all $i \in [n]$, so that---inflating $M$ if necessary---the minimum over $\bu$ is obtained by minimizing each coordinate separately:
\begin{align*}
	\min_{\bt \in D}\; \widetilde{\ell}_n(\bt) &=  \min_{\bt \in D} \max_{\gamma \geq 0}  \Bigl\{ \frac{\lambda}{2p} \| \bt \|_2^2 -\frac{\gamma \alpha_2}{2np} \|\pproj \bg\|_2^2 \| \pproj \bt \|_2^2 \\
	&+ \frac{1}{n} \sum_{i=1}^{n} \min_{u_i \in \mathbb{R}} \Bigl[\rho(-y_i u_i) + \frac{\gamma}{2} \Bigl(u_i - (\bG\proj_{\bt_0} \bt)_i - \frac{\sqrt{\alpha_2} \| \pproj \bt \|_2}{\sqrt{p}} h_i\Bigr)^2 \Bigr]\Bigr\}.
\end{align*}

Noting that RHS of the above display is $\min_{\bt \in D} \ell_n(\bt)$ and combining with the inequality~\eqref{ineq:step1-cgmt-proof} yields
\[
\pr\Bigl\{ \min_{\bt \in D}\; \mathcal{L}_n(\bt) \leq t\Bigr\} \leq 2\pr\Bigl\{ \min_{\bt \in D}\; \ell_n(\bt) \leq t\Bigr\} + 2e^{-n} +  \varepsilon.
\]
Since $\varepsilon$ was arbitrary, the proof is complete upon taking $\varepsilon \downarrow 0$.   It remains to show the deferred proof of the Lipschitz continuity of $F_n$.

\medskip
\noindent \underline{Establishing Lipschitz continuity of $F_n$.}  Let $\bx, \by \in \mathbb{R}^{n}$ and $x(t) = t\bx + (1 - t) \by$ for $t \in [0, 1]$.  Then,
\[
\lvert F_n(\bx; \bt) - F_n(\by; \bt)\rvert = \left\lvert\int_{0}^{1}\frac{\mathrm{d}}{\mathrm{d}t}F_n(x(t))\mathrm{d}t\right\rvert = \frac{1}{n}\left\lvert \int_{0}^{1} \langle \bx - \by, \rho'(t\bx + (1 - t) \by)\rangle \mathrm{d}t \right\rvert \leq \frac{\norm{\bx - \by}_2}{\sqrt{n}},
\]
where the last inequality follows by applying the triangle inequality in conjunction with the Cauchy--Schwarz inequality and the uniform bound $\sup_{t \in \mathbb{R}}\lvert \rho'(t) \rvert \leq 1$. \qed

\subsection{Proof of Lemma~\ref{lem:costconcentration}}\label{subsubsec:proof_cost_concentration}
This subsection is organized in the following way.  We prove the point-wise inequality, Lemma~\ref{lem:costconcentration}(a), in Section~\ref{sec:proof-costconcentration-a}.  We then prove the uniform inequality, Lemma~\ref{lem:costconcentration}(b), in Section~\ref{sec:proof-costconcentration-b}.  We will additionally require the following lemma, whose proof is a straightforward extension of Lemma~\ref{lem:useful-asymptotic-loss}(e), and is consequently omitted for brevity.
\begin{lemma}
	\label{lem:lipschitz-L}
	Consider the setting of Lemma~\ref{lem:costconcentration} and recall the set $\mathcal{C}$ as well as the constant $\bar{\gamma}$ therein.  There exists a pair of positive constants $(c, C)$ depending only on $(K_1, K_2)$ such that $L_n$~\eqref{def:aux_scalar} is $C(1 \vee \bar{\gamma})$--Lipschitz on the domain $\mathcal{C}$.
\end{lemma}

\subsubsection{Proof of Lemma~\ref{lem:costconcentration}(a)} \label{sec:proof-costconcentration-a}
We begin by decomposing the scalarized auxiliary loss as
\begin{align}
	\label{eq:pointconcentration_decompL}
	L_n(\sigma, \xi, \gamma) &= \frac{\lambda (\sigma^2 + \xi^2 R^2)}{2}- A_n(\sigma, \xi, \gamma) + B_n(\sigma, \xi, \gamma),
\end{align}
where
	\begin{align*}
		A_n(\sigma, \xi, \gamma) &:= \frac{\alpha_2  \gamma \sigma^2\bigl\| \pproj \bg\bigr\|_2^2}{2n}, \\
		B_n(\sigma, \xi, \gamma) &:= \frac{1}{n} \sum_{i=1}^{n}\min_{u_i \in \mathbb{R}}\Bigl\{\rho(-y_i u_i) + \frac{\gamma}{2}\Bigl(u_i - \sqrt{\alpha_2} \xi R Z_{1, i} - \sqrt{\alpha_2} \sigma Z_{2,i}\Bigr)^2\Bigr\}.
	\end{align*}
With this shorthand, we claim the following two tail bounds. 
\begin{subequations}
	\begin{align}
		\pr\Bigl\{\Bigl \lvert A_n(\sigma, \xi, \gamma) - \frac{\alpha_2\gamma \sigma^2}{2\delta}  \Bigr \rvert \geq \epsilon \Bigr\} &\leq 4\exp\left\{-c\min(n\epsilon^2 \gamma^{-2}, n\epsilon \gamma^{-1})\right\}, \label{ineq:pointconcentration_An} \\
		\pr\left\{\left \lvert B_n(\sigma, \xi, \gamma) - \E B_n(\sigma, \xi, \gamma) \right \rvert \geq \epsilon \right\} &\leq 2\exp\{-cn\epsilon^2\}.\label{ineq:pointconcentration_Bn}
	\end{align}
\end{subequations}
Note that the bound~\eqref{ineq:pointconcentration_An} is a straightforward application of Bernstein's inequality~\citep[Theorem 2.8.1]{vershynin2018high}.  We prove the second bound at the end of the section.  The desired result follows by combining the decomposition~\eqref{eq:pointconcentration_decompL} with the inequalities~\eqref{ineq:pointconcentration_An} and~\eqref{ineq:pointconcentration_Bn} to obtain
\[
\pr\left\{\left \lvert L_n(\sigma, \xi, \gamma) - \E L_n(\sigma, \xi, \gamma) \right \rvert \geq \epsilon \right\} \leq 2e^{-c n\epsilon^2} + 4\exp\left\{-c\min(n\epsilon^2 \gamma^{-2}, n\epsilon \gamma^{-1})\right\}.
\]
It remains to prove the inequality~\eqref{ineq:pointconcentration_Bn}.

\paragraph{Proof of the inequality~\eqref{ineq:pointconcentration_Bn}.}
First, define the function 
\[
F (Y, Z_1, Z_2) := \min_{u \in \mathbb{R}}\left\{\rho(-Y u) + \frac{\gamma}{2} \left(u -  \sqrt{\alpha_2} \xi R Z_1 - \sqrt{\alpha_2} \sigma Z_2\right)^2\right\},
\]
so that $B_n(\sigma, \xi, \gamma) = \frac{1}{n}\sum_{i=1}^{n} F(y_i, Z_{1, i}, Z_{2, i})$.
We now introduce the decomposition
\begin{align}
	\label{eq:control_Bn_T1T2_def}
	F(Y, Z_1, Z_2) - \E F(Y, Z_1, Z_2) = T_1 + T_2,
\end{align}
where
\begin{align*}
	T_1 &:= F(Y, Z_1, Z_2) - \E \left\{F(Y, Z_1, Z_2) \mid Z_1, Z_2 \right\},\\
T_2 &:= \E \left\{F(Y, Z_1, Z_2) \mid Z_1, Z_2 \right\} - \E F(Y, Z_1, Z_2).
\end{align*}
The remainder of the proof consists of bounding the Orlicz norms $\| T_1 \|_{\psi_2}$ and $\| T_2 \|_{\psi_2}$ and concluding via Hoeffding's inequality~\citep[Theorem 2.6.3]{vershynin2018high}.

\smallskip
\noindent \underline{Bounding $\| T_1 \|_{\psi_2}$.}
Note that
\begin{align*}
	\norm{\nabla_{(Z_1, Z_2)} F(y, Z_1, Z_2)}_2 \overset{\mathsf{(i)}}{\leq} \alpha_2 \sqrt{\xi^2 R^2 + \sigma^2} \leq M\alpha_2\sqrt{R^2 + 1},
\end{align*}
where step $\1$ follows by applying Lemma~\ref{lem:derivative_prox}(b) to compute partial derivatives of $F$ and subsequently using the fact that $\sup_{t \in \reals} \lvert \rho'(t) \rvert \leq 1$.  Consequently, for fixed $y \in \{\pm 1\}$, the function $F(y, \cdot, \cdot)$ is at most $M\alpha_2\sqrt{R^2 + 1}$--Lipschitz in $(Z_1, Z_2)$.
Note that
\begin{align}
	\left \lvert F(Y, Z_1, Z_2) - \E \left\{F(Y, Z_1, Z_2) \mid Z_1, Z_2 \right\}\right \rvert &\overset{\1}{\leq} \left \lvert F(Y, Z_1, Z_2) - F(-Y, Z_1, Z_2)\right \rvert\nonumber\\
	&\overset{\mathsf{(ii)}}{\leq} M\alpha_2\sqrt{R^2 + 1} \cdot \sqrt{Z_1^2 + Z_2^2}\nonumber,
\end{align}
where step $\1$ follows since $Y \in \{ \pm 1\}$ and step $\2$ follows by noting that $F(1, Z_1, Z_2) = F(-1, -Z_1, -Z_2)$ and subsequently applying Lipschitz continuity of the function $F(y, \cdot, \cdot)$.  Evidently, since $Z_1, Z_2 \sim \mathsf{N}(0, 1)$, combining the elements yields
\begin{align}
	\label{ineq:control_BnT1_psi2}
	\norm{T_1}_{\psi_2} = \norm{F(Y, Z_1, Z_2) - \E \left\{F(Y, Z_1, Z_2) \mid Z_1, Z_2 \right\}}_{\psi_2} \leq C'.
\end{align}

\smallskip
\noindent \underline{Bounding $\| T_2 \|_{\psi_2}$.}
Fix $y \in \{\pm 1\}$ and note the inequality
\begin{align}
	\label{ineq:control_BnT2_ineq1}
	\norm{p(Y=y \mid Z_1, Z_2) F(y, Z_1, Z_2)}_{\psi_2} \overset{\1}{\leq} \norm{F(y, Z_1, Z_2)}_{\psi_2} \overset{\2}{\leq} C'',
\end{align}
where step $\1$ follows since $\left \lvert p(Y=y \mid Z_1, Z_2) \right \rvert \leq 1$ and step $\2$ follows by noting the Lipschitz continuity of the function $F(y, \cdot, \cdot)$ and subsequently applying~\citet[Theorem 5.2.2]{vershynin2018high}.
Expanding $T_2$ and applying the triangle inequality in conjunction with~\citet[Lemma 2.6.8]{vershynin2018high} and the inequality~\eqref{ineq:control_BnT2_ineq1} yields the inequality
\begin{align}
	\label{ineq:control_BnT2_psi2}
	\hspace*{-0.2cm}\norm{T_2}_{\psi_2} &= \Bigl\|\sum_{y \in \{\pm 1\}} p(Y=y \mid Z_1, Z_1) F(y, Z_1, Z_2) - \E \left\{p(Y=y \mid Z_1, Z_2) F(y, Z_1, Z_2)\right\}\Bigr\|_{\psi_2} \leq C''.
\end{align}

\smallskip
\noindent \underline{Putting the pieces together.}
We invoke the triangle inequality in conjunction with the decomposition~\eqref{eq:control_Bn_T1T2_def} and the inequalities~\eqref{ineq:control_BnT1_psi2} and~\eqref{ineq:control_BnT2_psi2} to obtain the upper bound
\begin{align*}
	\norm{F(Y, Z_1, Z_2) - \E F(Y, Z_1, Z_2)}_{\psi_2} \leq \norm{T_1}_{\psi_2} + \norm{T_2}_{\psi_2} \leq C' + C''.
\end{align*}
The conclusion follows by Hoeffding's inequality~\cite[Theorem 2.6.3]{vershynin2018high}.
\qed

\subsubsection{Proof of Lemma~\ref{lem:costconcentration}(b)} \label{sec:proof-costconcentration-b}
We employ a straightforward epsilon-net argument.  To this end, let $\mathcal{N}_{\sigma}, \mathcal{N}_{\xi}, \mathcal{N}_{\gamma}$ be $\epsilon(1 \vee \widebar{\gamma})^{-1}/C$--nets of the domain of $\sigma$, $\xi$, and $\gamma$, respectively.  Now, for any $(\sigma, \xi, \gamma) \in \mathcal{C}$, let $\pi(\sigma) \in \mathcal{N}_{\sigma}$ denote the projection of $\sigma$ onto $\mathcal{N}_{\sigma}$ so that $\lvert \sigma - \pi(\sigma) \rvert \leq \epsilon(1 \vee \widebar{\gamma})^{-1}/C$, with $\pi(\xi), \pi(\gamma)$ defined analogously.  Then, define the events
\begin{align*}
	\mathcal{A}_1 &:= \Bigl\{\sup_{(\sigma, \xi, \gamma) \in \mathcal{C}} \bigl \lvert L_n(\sigma, \xi, \gamma) - L_n(\pi(\sigma), \pi(\xi), \pi(\gamma)) \bigr\rvert \geq \frac{\epsilon}{3}\Bigr\},\\
\mathcal{A}_2 &:= \Bigl\{\sup_{(\sigma, \xi, \gamma) \in \mathcal{C}} \bigl \lvert L(\sigma, \xi, \gamma) - L(\pi(\sigma), \pi(\xi), \pi(\gamma)) \bigr\rvert \geq \frac{\epsilon}{3}\Bigr\},\\
	\mathcal{A}_3 &:= \Bigl\{\max_{(\sigma, \xi, \gamma) \in \mathcal{N}_{\sigma} \times \mathcal{N}_{\xi} \times \mathcal{N}_{\gamma}} \bigl \lvert L_n(\sigma, \xi, \gamma) - L(\sigma, \xi, \gamma) \bigr\rvert \geq \frac{\epsilon}{3}\Bigr\},\\
	\mathcal{A} &:= \Bigl\{\sup_{(\sigma, \xi, \gamma) \in \mathcal{C}} \bigl \lvert L_n(\sigma, \xi, \gamma) - L(\sigma, \xi, \gamma) \bigr\rvert \geq \epsilon\Bigr\},
\end{align*}
and note that $\mathcal{A} \implies \mathcal{A}_1 \cup \mathcal{A}_2 \cup \mathcal{A}_3$.
We proceed to bound the probabilities $\pr\{\mathcal{A}_1\}, \pr\{\mathcal{A}_2\},$ and $\pr\{\mathcal{A}_3\}$ in turn.  First, note that for all $(\sigma, \xi, \gamma) \in \mathcal{C}$, $
\| (\sigma, \xi, \gamma) - (\pi(\sigma), \pi(\xi), \pi(\gamma))\|_2 \leq C\epsilon(1 \vee \widebar{\gamma})^{-1}$.
Consequently, Lemma~\ref{lem:lipschitz-L} implies $\pr\{\mathcal{A}_1\} \leq 8e^{-cn}$ and $\pr\{\mathcal{A}_2\} = 0$.  Turning to event $\mathcal{A}_3$, we have
\begin{align}
	\pr\left\{\mathcal{A}_3\right\} &= \pr\biggl\{\bigcup_{\sigma \in \mathcal{N}_{\sigma}, \xi \in \mathcal{N}_{\xi}, \gamma \in \mathcal{N}_{\gamma}}\lvert L_n(\sigma, \xi, \gamma) - L(\sigma, \xi, \gamma)\rvert \geq \frac{\epsilon}{3}\biggr\}\nonumber \leq \frac{C(1 \vee \widebar{\gamma})^3}{\epsilon^3}\exp\left\{-c\min(n\epsilon^2 \widebar{\gamma}^2, n\epsilon \widebar{\gamma})\right\},
\end{align}
where the final inequality follows by applying the union bound in conjunction with Lemma~\ref{lem:costconcentration}(a).  Finally, since  $\mathcal{A} \implies \mathcal{A}_1 \cup \mathcal{A}_2 \cup \mathcal{A}_3$, we apply the union bound to obtain the inequality
\[
\pr\{\mathcal{A}\} \leq \frac{C(1 \vee \widebar{\gamma})^3}{\epsilon^3}\exp\left\{-c\min(n\epsilon^2 \widebar{\gamma}^2, n\epsilon \widebar{\gamma})\right\} + 8e^{-cn}.
\]
We conclude by increasing the constant $C$ in the display above. \qed

\subsection{Proof of Lemma~\ref{claim:constrain-gamma}}\label{sec:proof-lem-constrain-gamma}
We require the following technical lemma, which provides several useful properties of the loss function $L_n$, and whose proof we provide in Section~\ref{sec:proof-lem-useful-L-n}. 
\begin{lemma}
	\label{lem:useful-L-n}
	Consider the loss function $L_n$~\eqref{def:aux_scalar} and let $M$ denote a positive scalar.  Let \sloppy\mbox{$\widehat{u}_{i}(\gamma) = \prox_{\rho(-y_i\cdot)}(\sqrt{\alpha_2}\xi R Z_{1,i} + \sqrt{\alpha_2}\sigma Z_{2,i}; \gamma)$}.  The following hold.
	\begin{itemize}
		\item[(a)] For any fixed $\sigma \in [0, M]$ and $\xi \in [-M,M]$, the function $\gamma \mapsto \max_{\gamma \geq 0} L_n(\sigma, \xi, \gamma)$ admits a unique maximizer $\gamma_n(\sigma, \xi)$ which satisfies the fixed point equation
		\[
		\frac{1}{2n}\sum_{i=1}^{n} \bigl[ \rho'\bigl(-y_i \widehat{u}_i(\gamma_n(\sigma, \xi))\bigr)\bigr]^2 = \frac{\gamma_n(\sigma, \xi)^2 \alpha_2 \sigma^2 \| \pproj \bg \|_2^2}{2n}. 
		\]
		\item[(b)] There exists a tuple of positive constants $(c, C, \ubar{\gamma})$ such that with probability at least $1 - Ce^{-cn}$, it holds that for all $\sigma \in [0, M]$ and \sloppy\mbox{$\xi \in [-M, M]$}, $\gamma_n(\sigma, \xi) \geq  \ubar{\gamma} \cdot \Bigl( 1 \vee \frac{1}{\sigma} \sqrt{\frac{\delta}{\alpha_2}} \Bigr)$.
		\item[(c)] There exists a tuple of positive constants $(c, c_1, C, C_1)$ such that with probability at least $1 - ce^{-Cn}$, it holds that for all $\sigma \in [0, c_1]$ and $\xi \in [-M,M]$, we have $\lvert \frac{\partial}{\partial \sigma} \gamma_n(\sigma, \xi) \rvert \leq C_1/\sigma^2$.
	\end{itemize}
\end{lemma}
Let $\mathcal{A}_1$ denote the event on which the guarantees of Lemma~\ref{lem:useful-L-n}(b),(c) hold and let $\mathcal{A}_2 = \{ \| \pproj \bg \|_2^2 \geq n/(2 \delta)\}$, which holds with probability at least $1 -2e^{-cn}$, by Bernstein's inequality.  We see that on $\mathcal{A}_1 \cap \mathcal{A}_2$, for any $\sigma \in [0, M]$ and $\xi \in [-M, M]$, 
\[
\ubar{\gamma} \cdot \Bigl( 1 \vee \frac{1}{\sigma} \sqrt{\frac{\delta}{\alpha_2}} \Bigr) \leq \gamma_n(\sigma, \xi) \leq \frac{C}{\sigma},
\]  
where the upper bound follows from Lemma~\ref{lem:useful-L-n}(a) since $\sup_{t \in \mathbb{R}} \{\rho'(t)\} \leq 1$.  It thus suffices to show that there exists a constant $\bar{\sigma}$ such that for all $0 \leq \sigma \leq \bar{\sigma}$, $\frac{\partial}{\partial \sigma} L_n(\sigma, \xi, \gamma_n(\sigma, \xi)) < 0$.  To this end, we apply Lemma~\ref{lem:derivative_prox} to find, on $\mathcal{A}_1 \cap \mathcal{A}_2$, 
\begin{align} \label{ineq:bound-derivative-sigma}
	\frac{\partial}{\partial \sigma}\; L_n\bigl(\sigma, \xi, \gamma_n(\sigma, \xi)\bigr) &= \lambda \sigma - \frac{\alpha_2 \gamma_n(\sigma, \xi) \sigma \| \pproj \bg \|_2^2}{n} - \frac{\sqrt{\alpha_2}}{n} \sum_{i=1}^{n} y_i Z_{2, i} \rho'(-y_i \widehat{u}_i) \nonumber \\
	&\leq \lambda \sigma -\frac{ \ubar{\gamma}}{2\sqrt{\alpha_2 \delta}} - \frac{\sqrt{\alpha_2}}{n} \sum_{i=1}^{n} y_i Z_{2, i} \rho'(-y_i \widehat{u}_i),
\end{align}  
Focusing our attention on the final term on the RHS of the above display, we apply Taylor's theorem with remainder in conjunction with Lemma~\ref{lem:derivative_prox} to obtain that for some $s \in [0, \sigma]$, (writing $\gamma_n$ as shorthand for $\gamma_n(\sigma, \xi)$)
\[
\rho'(-y_i \widehat{u}_i(\sigma)) = \rho'(-y_i \widehat{u}_i(0))- \sigma \cdot y_i \rho''(-y_i \widehat{u}_i(s)) \cdot \biggl\{ \frac{\gamma_n \sqrt{\alpha_2} Z_{2,i}}{\gamma_n + \rho''(\widehat{u}_i(s))} - \frac{y_i \rho'(-y_i \widehat{u}_i(s)) \frac{\partial}{\partial \sigma} \gamma_n}{\gamma_n \cdot (\gamma_n + \rho''(\widehat{u}_i(s)))}\biggr\}.
\]
Consequently, 
\[
\frac{1}{n} \sum_{i=1}^{n} y_i Z_{2, i} \rho'(-y_i \widehat{u}_i) =T_1 + T_2,
\]
where
	\begin{align*}
		T_1 &= \frac{1}{n} \sum_{i=1}^{n} y_i Z_{2, i} \rho'(-y_i \widehat{u}_i(0)),\\
		T_2 &= - \frac{\sigma}{n} \sum_{i=1}^{n} Z_{2, i}  \rho''(-y_i \widehat{u}_i(s)) \cdot  \biggl\{ \frac{\gamma_n \sqrt{\alpha_2} Z_{2,i}}{\gamma_n + \rho''(\widehat{u}_i(s))} - \frac{y_i \rho'(-y_i \widehat{u}_i(s)) \frac{\partial}{\partial \sigma} \gamma_n}{\gamma_n \cdot (\gamma_n + \rho''(\widehat{u}_i(s)))}\biggr\}.
	\end{align*}
Note that $y_i \rho'(-y_i \widehat{u}_i(0))$ is bounded and independent of $Z_{2, i}$ so that $\E[y_i Z_{2, i} \rho'(-y_i \widehat{u}_i(0))] = 0$ and $\| y_i Z_{2, i} \rho'(-y_i \widehat{u}_i(0))\|_{\psi_1} \leq C$, whence applying Bernstein's inequality~\citep[Theorem 2.8.1]{vershynin2018high} yields $\lvert T_1 \rvert \leq \ubar{\gamma}/4(\sqrt{\alpha_2 \delta})$ with probability at least $1 - 2e^{-cn}$.  

Next,  let $\mathcal{A}_3 = \{\frac{1}{n} \sum_{i=1}^{n} Z_{2, i}^2 \vee \frac{1}{n} \sum_{i=1}^{n} \lvert Z_{2, i} \rvert \leq 2\}$, which by Bernstein's inequality satisfies $\mathbb{P}(\mathcal{A}_3) \geq 1 - 2e^{-cn}$.  Then, on $\mathcal{A}_1 \cap \mathcal{A}_2 \cap \mathcal{A}_3$, we apply the triangle inequality to deduce that,
\begin{align*} 
	\lvert T_2 \rvert \leq \sigma \frac{1}{n} \sum_{i=1}^{n} \lvert Z_{2, i} \rvert \cdot \bigl(\sqrt{\alpha_2} \lvert Z_{2, i} \rvert + C \bigr) \leq C \sigma
\end{align*}
Substituting these upper bounds into the inequality~\eqref{ineq:bound-derivative-sigma} yields that on $\mathcal{A}_1 \cap \mathcal{A}_2 \cap \mathcal{A}_3$, 
\[
\frac{\partial}{\partial \sigma} \; L_n\bigl(\sigma, \xi, \gamma_n(\sigma, \xi)\bigr) \leq C\sigma - \frac{ \ubar{\gamma}}{4\sqrt{\alpha_2 \delta}}.
\]
Putting the pieces together implies that for \sloppy\mbox{$\bar{\sigma} = \frac{\ubar{\gamma}}{4C \sqrt{\alpha_2 \delta}} $}, $L_n(\sigma, \xi, \gamma_n(\sigma, \xi))$ is decreasing for all $\sigma \in [0, \ubar{\sigma}]$ and $\xi \in [-M, M]$, which completes the proof. \qed 

\subsubsection{Proof of Lemma~\ref{lem:useful-L-n}}\label{sec:proof-lem-useful-L-n}
We prove each part in turn. 

\paragraph{Proof of Lemma~\ref{lem:useful-L-n}(a)}
Consider fixed $\xi$ and $\sigma$ and define $\psi: \mathbb{R}_{+} \rightarrow \mathbb{R}$ as $\psi(\gamma) = L_n(\sigma, \xi, \gamma)$.  Applying Lemma~\ref{lem:derivative_prox} yields
\[
\frac{\partial^2}{\partial \gamma^2} \psi(\gamma) = -\frac{1}{\gamma^2}\cdot \frac{1}{n} \sum_{i=1}^{n} \frac{\bigl(\rho'(-y_i \widehat{u}_i)\bigr)^2}{\gamma + \rho''(-y_i \widehat{u}_i)} < 0,
\]
whence we deduce that $\psi$ is strictly convex and admits a unique maximizer.  By the first order conditions, this maximizer satisfies the fixed point equation in the statement, as desired.  

\paragraph{Proof of Lemma~\ref{lem:useful-L-n}(b)}
We will use the shorthand $X_i = \sqrt{\alpha_2}\xi R Z_{1,i} +\sqrt{\alpha_2} \sigma Z_{2,i}$.
Additionally, define the event $\mathcal{A}_1 = \{\| \bg \|_2 \leq 2\sqrt{n}\}$, which occurs with probability at least $1 - 2e^{-cn}$.  Further, let $T > 0$ be a constant large enough to ensure that $ \mathbb{P}\bigl\{\lvert Z_{1,i} \rvert + \lvert Z_{2, i} \rvert \leq T \cdot \bigl[\sqrt{\alpha_2} M (1 \vee R)\bigr]^{-1} \bigr\} \geq \frac{1}{2}$ and define the event
\[
\mathcal{A}_2 = \biggl\{ \frac{1}{n} \sum_{i=1}^{n} \mathbbm{1}\Bigl\{\lvert Z_{1,i} \rvert + \lvert Z_{2, i} \rvert \leq T \cdot \bigl[\sqrt{\alpha_2} M (1 \vee R)\bigr]^{-1} \Bigr\} \geq \frac{1}{4}\biggr\},
\]
which by Hoeffding's inequality satisfies $\mathbb{P}(\mathcal{A}_2) \geq 1 - 2e^{-c n}$.  For the remainder of the proof, we will work on the event $\mathcal{A} = \mathcal{A}_1 \cap \mathcal{A}_2$.  

We first show that, on $\mathcal{A}$,  there exists a constant $\ubar{\gamma}$ such that for all $\sigma \in [0,M], \xi \in [-M, M]$, $\gamma_n(\sigma, \xi) \geq \ubar{\gamma}$.  To this end, define the function $\zeta: \mathbb{R}_{\geq 0} \rightarrow \mathbb{R}$ as 
\[
\zeta(\gamma) = - \frac{\gamma \alpha_2}{2n}\bigl\|\pproj \bg \bigr\|_2^2 \sigma^2 + \frac{1}{n}\sum_{i=1}^{n}\min_{u_i \in \mathbb{R}}\Bigl\{\rho(-y_i u_i) + \frac{\gamma}{2}(u_i - X_i)^2\Bigr\}.
\]
We will show that there exists a constant $\ubar{\gamma} > 0$ such that the function $\zeta$ is increasing for all $\gamma \leq \ubar{\gamma}$, from which the conclusion follows immediately.  To this end, we compute the derivative
\[
\zeta'(\gamma) = -\frac{\alpha_2}{2n}\bigl\| \pproj\bg\|_2^2 \sigma^2 + \frac{1}{2n} \sum_{i=1}^{n} \bigl(\prox_{\rho(-y_i \cdot)}(X_i; \gamma) - X_i\bigr)^2.
\]
Since $\bt \in \mathbb{B}_2(C)$ and on the event $\mathcal{A}$, $\norm{\bg}_2^2 \leq 2n$, we deduce that
\begin{align} \label{ineq:penultimate-zeta-prime}
\zeta'(\gamma) \geq - C + \frac{1}{2n} \sum_{i=1}^{n} &\bigl(\prox_{\rho(-y_i \cdot)}(X_i; \gamma) - X_i\bigr)^2 \nonumber\\
&\geq - C + \frac{1}{2n} \sum_{i=1}^{n} \bigl(\prox_{\rho(-y_i \cdot)}(X_i; \gamma) - X_i\bigr)^2 \mathbbm{1}_{\lvert X_i \rvert \leq T}.
\end{align}
Applying Lemma~\ref{lem:lbgamma} in turn yields a constant $\ubar{\gamma}_T$ (which may depend on the truncation level $T$) such that the following lower bound holds
\[
\zeta'(\gamma) \geq - C + \frac{1}{2n} \sum_{i=1}^{n} \frac{1}{64} \log^2(\gamma)\mathbbm{1}_{\lvert X_i \rvert \leq T} \qquad \text{ for all } \qquad \gamma \leq \ubar{\gamma}_T.
\]
Consequently, since $0 \leq \sigma \leq M, \lvert \xi \rvert \leq M$, we have, for all $\gamma \in [0, \ubar{\gamma}_T]$ that on
\[
\zeta'(\gamma) \geq - C + \frac{\log^2(\gamma)}{128n} \sum_{i=1}^{n} \mathbbm{1}_{\lvert X_i \rvert \leq T} \geq -C + \frac{\log^2(\gamma)}{128 n} \sum_{i=1}^{n}  \mathbbm{1}\Bigl\{\lvert Z_{1,i} \rvert + \lvert Z_{2, i} \rvert \leq T \cdot \bigl[\sqrt{\alpha_2} M (1 \vee R)\bigr]^{-1} \Bigr\},
\]
so that on $\mathcal{A}$, $\zeta'(\gamma) \geq -C + \frac{\log^2(\gamma)}{512} $.  Setting $\ubar{\gamma}$ as a small enough positive constant yields the desired result.  

We next show that, on $\mathcal{A}$, for all $\sigma \in [0, M]$ and $\xi \in [-M, M]$, $\gamma_n(\sigma, \xi) \geq c /\sigma$.  To this end, note that the first order condition in part (a) implies
\[
\gamma_n(\sigma, \xi) = \frac{1}{\sigma} \cdot \sqrt{\frac{n}{\alpha_2 \| \pproj \bg \|_2^2}} \cdot \frac{1}{n} \sum_{i=1}^{n} \bigl[\rho'(-y_i \widehat{u}_i(\gamma_n(\sigma, \xi))\bigr)\bigr]^2 \geq \frac{c}{\sigma} \cdot \frac{1}{n} \sum_{i=1}^{n} \bigl[\rho'(-y_i \widehat{u}_i(\gamma_n(\sigma, \xi))\bigr)\bigr]^2.
\]
Now, note that for $\gamma > 0$, the Moreau envelope $x \mapsto M_{-y_i \cdot}(x; \gamma)$ is coercive so that on the truncation event $\lvert X_i \rvert \leq T$, the proximal operator is absolutely bounded by a constant: $\lvert \widehat{u}_i(\gamma)\rvert \leq C$.  Hence, since $\rho'$ is a non-negative and increasing function, we deduce that on $\mathcal{A}$, 
\[
\frac{1}{n} \sum_{i=1}^{n} \bigl[\rho'(-y_i \widehat{u}_i)\bigr]^2 \geq \frac{1}{n} \sum_{i=1}^{n} \bigl[\rho'(-y_i \widehat{u}_i)\bigr]^2 \mathbbm{1}\{\lvert X_i \rvert \leq T\} \geq \frac{\bigl[\rho'(-C)\bigr]^2}{n} \sum_{i=1}^{n} \mathbbm{1}\{\lvert X_i \rvert \leq T\} \geq c.
\]
Combining the previous two displays then yields the desired result. \hfill \qed

\paragraph{Proof of Lemma~\ref{lem:useful-L-n}(c)}
Recall the notation $X_i$ from the proof of part (b) as well as \sloppy\mbox{$\widehat{u}_i \equiv \widehat{u}_i(\gamma, \sigma, \xi) = \prox_{\rho(-y_i \cdot)}(X_i; \gamma)$}.  When clear, we will write $\gamma_n$ in place of $\gamma_n(\sigma, \xi)$.  Using Lemma~\ref{lem:derivative_prox}(a), we re-write the first order condition in part (a) as
\[
\frac{1}{2n} \sum_{i=1}^{n} \gamma_n(\sigma, \xi)^2 \cdot \bigl[\widehat{u}_i(\gamma_n(\sigma, \xi), \sigma, \xi) - X_i\bigr]^2 - \frac{1}{2n} \gamma_n(\sigma, \xi)^2 \sigma^2 \alpha_2 \| \pproj \bg \|_2^2 = 0.
\]
Differentiating the above expression in $\sigma$ thus yields $I + II + III + IV = 0$, where
\begin{align*}
	I &= \gamma_n \cdot \Bigl(\frac{\partial}{\partial \sigma} \gamma_n\Bigr) \cdot \frac{1}{n} \sum_{i=1}^{n} \bigl(\widehat{u}_i - X_i\bigr)^2 =  \gamma_n^{-1} \cdot \Bigl(\frac{\partial}{\partial \sigma} \gamma_n\Bigr) \cdot \frac{1}{n} \sum_{i=1}^{n} \bigl[\rho'\bigl(-y_i \widehat{u}_i \bigr)\bigr]^2,\\
	II &= \frac{\gamma_n^2}{n} \sum_{i=1}^{n} \bigl(\widehat{u}_i - X_i\bigr) \cdot \Bigl\{ \frac{\partial}{\partial \sigma} \widehat{u}\bigl(\gamma_n, \sigma) + \frac{\partial}{\partial \sigma} \gamma_n(\sigma, \xi) \cdot \frac{\partial}{\partial \gamma} \widehat{u}_i(\gamma_n, \sigma) - \sqrt{\alpha_2} Z_{2,i}\Bigr\}\\
	&= - \Bigl(\frac{\partial}{\partial \sigma} \gamma_n \Bigr) \frac{1}{n} \sum_{i=1}^{n} \frac{\bigl[\rho'(-y_i \widehat{u}_i)\bigr]^2}{\gamma_n + \rho''(-y_i \widehat{u}_i)} +  \frac{\gamma_n^2}{n} \sum_{i=1}^{n} \frac{y_i \rho'(-y_i \widehat{u}_i) \sqrt{\alpha_2} Z_{2,i}}{\gamma_n + \rho''(-y_i \widehat{u}_i)} - \frac{\gamma_n}{n} \sqrt{\alpha_2} \sum_{i=1}^{n} y_i \rho'(-y_i \widehat{u}_i) Z_{2, i},\\
	III &= - \gamma_n \Bigl( \frac{\partial}{\partial \sigma} \gamma_n\Bigr) \frac{\alpha_2 \sigma^2 \| \pproj \bg \|_2^2}{n} = -\gamma_n^{-1}  \Bigl( \frac{\partial}{\partial \sigma} \gamma_n\Bigr) \frac{1}{n} \sum_{i=1}^{n} \bigl[\rho'(-y_i \widehat{u}_i)\bigr]^2, \\
	IV &= -\frac{\gamma_n^2 \alpha_2 \sigma \| \pproj \bg \|_2^2}{n}.
\end{align*}
Re-arranging gives
\begin{align*}
\frac{\partial}{\partial \sigma} \gamma_n(\sigma, \xi) = \biggl( \frac{1}{n} \sum_{i=1}^{n} \frac{\bigl[\rho'(-y_i \widehat{u}_i)\bigr]^2}{\gamma_n + \rho''(-y_i \widehat{u}_i)} \biggr)^{-1} \cdot  &\biggl\{ \frac{\gamma_n^2}{n} \sum_{i=1}^{n} \frac{y_i \rho'(-y_i \widehat{u}_i) \sqrt{\alpha_2} Z_{2,i}}{\gamma_n + \rho''(-y_i \widehat{u}_i)} \\
&- \frac{\gamma_n}{n} \sqrt{\alpha_2} \sum_{i=1}^{n} y_i \rho'(-y_i \widehat{u}_i) Z_{2, i} - \frac{1}{\sigma n} \sum_{i=1}^{n} \bigl[\rho'(- y_i \widehat{u}_i)\bigr]^2
\biggr\}.
\end{align*}
Let $\mathcal{A}_1, \mathcal{A}_2$ be as in the proof of part (b), so that on $\mathcal{A}_1 \cap \mathcal{A}_2$, there exists a constant $c' > 0$ such that for all $\sigma \in [0, c']$ and all $\xi \in [-M, M]$, $\gamma_n(\sigma, \xi) \geq 2$.  Then, applying the triangle inequality yields
\[
\Bigl \lvert \frac{\partial}{\partial \sigma} \gamma_n(\sigma, \xi) \Bigr \rvert \leq 2 \gamma_n \biggl(\frac{1}{n} \sum_{i=1}^{n}  \bigl[\rho'(- y_i \widehat{u}_i)\bigr]^2 \biggr)^{-1} \cdot \biggl\{  \frac{\gamma_n}{n} \sum_{i=1}^{n} \sqrt{\alpha_2} \lvert Z_{2, i} \rvert \biggr\} + \frac{1}{\sigma} \leq \frac{c \gamma_n^2}{n} \sum_{i=1}^{n} \lvert Z_{2, i} \rvert + \frac{\gamma_n}{\sigma}. 
\]
Next, define the events $\mathcal{A}_3 = \{ \| \pproj \bg \|_2^2 \geq n/2 \}$ and $\mathcal{A}_4 = \{\frac{1}{n} \sum_{i=1}^{n} \lvert Z_i \rvert \leq 2\}$, noting that by Bernstein's inequality, $\mathbb{P}(\mathcal{A}_3) \wedge \mathbb{P}(\mathcal{A}_4) \geq 1 - 2e^{-cn}$.  On $\mathcal{A}_3$, since $\sup_{t \in \mathbb{R}} \rho'(t) \leq 1$, we deduce from the first order condition in Lemma~\ref{lem:useful-L-n}(a) that
$\gamma_n(\sigma, \xi) \leq \frac{1}{\sigma\sqrt{\alpha_2}} \sqrt{\frac{n}{\| \pproj \bg\|_2^2}} \leq \frac{C}{\sigma}$.
Putting the pieces together, we deduce that on $\bigcup_{\ell = 1}^{4} \mathcal{A}_{\ell}$, for all $\sigma \in [0, c']$ and $ \xi \in [-M, M]$, $\bigl \lvert \frac{\partial}{\partial \sigma} \gamma_n(\sigma, \xi) \bigr \rvert \leq C/\sigma^2$, as desired.  \qed

\section{Auxiliary proofs for universality: Proof of Lemma~\ref{lem:universality}} \label{sec:aux-universality}
This appendix is dedicated to the proof of Lemma~\ref{lem:universality}.  The organization is as follows: In Section~\ref{sec:aux-universality-preliminaries} we provide preliminary definitions and lemmas, we then prove Lemma~\ref{lem:universality} in Section~\ref{sec:proof-universality} and prove the lemmas introduced in Section~\ref{sec:aux-universality-preliminaries} in later subsections.

\subsection{Preliminaries}\label{sec:aux-universality-preliminaries}
The main workhorse of the proof is the following theorem, which implements the Lindeberg principle with correlation.  The proof is provided in Section~\ref{sec:proof-lindeberg-correlation}.
\begin{theorem}[\cite{lindeberg1922neue,chatterjee2006generalization}]
	\label{thm:lindeberg_correlation}
	Suppose $\bx, \widebar{\bx}, \bz, \widebar{\bz}$ are zero-mean random vectors in $\mathbb{R}^n$ with independent components.  Assume that for $1 \leq i \leq n$, 
	\begin{align*}
	(a)\;\; \E\{x_i^2\} = \E\{\widebar{x}_i^2\}, \qquad  (b)\;\; \E\{z_i^2\} = \E\{\widebar{z}_i^2\}, \qquad \text{ and } \qquad (c)\;\; \E\{x_i z_i\} = \E\{\widebar{x}_i \widebar{z}_i\}.
	\end{align*}
	Let $f: \mathbb{R}^{2n} \rightarrow \mathbb{R}$, $(\bx, \bz) \mapsto f(x_1, z_1, \dots, x_i, z_i, \dots, x_n, z_n)$ be thrice continuously differentiable.  Further, assume that the third-order partial derivatives are uniformly bounded as 
	\[
	\sup_{\bx \in \mathbb{R}^n, \bz \in \mathbb{R}^n} \bigl\lvert \partial_{2i}^{q} \partial_{2i + 1}^{3-q} f(\bx, \bz) \bigr\rvert \leq L \qquad \text{ for all } \qquad 1 \leq i \leq n \;\; \text{ and } \;\; q \in \{0, 1, 2, 3\},
	\]
	where $\partial_k$ denotes partial differentiation with respect to the $k$th coordinate.
	Then, the following bound holds.
	\begin{align}
		\label{term:lindeberg_thm_term}
		\left\lvert \E f(\bx, \bz) - \E f(\widebar{\bx}, \widebar{\bz}) \right\rvert \leq \frac{8nL}{3} \cdot \max_{\substack{1 \leq i \leq n,\\ q = \{0, 1, 2, 3\}}} \bigl\{\E \lvert x_i^{q}z_i^{3-q} \rvert, \E \vert \bar{x}_i^q \bar{z}_i^{3-q}\rvert\bigr\}.
	\end{align}
\end{theorem}
We next present several approximation lemmas.  Recall the set $\mathbb{T}$~\eqref{eq:T-def} and the minimum over that set $\mathfrak{M}_T$~\eqref{eq:perturbed-difference}.
Subsequently, consider a positive scalar $\epsilon$ and apply~\citet[Corollary 4.2.13]{vershynin2018high} to obtain $\mathbb{T}_{\epsilon}$---an $\epsilon$-net of the set $\mathbb{T}$.  Then, define the quantity $\mathfrak{M}_{\epsilon}(s, \bZ)$, the minimum over the $\epsilon$-net $\mathbb{T}_{\epsilon}$, as
\begin{align}
	\label{eq:perturbed-difference-eps}
	 \mathfrak{M}_{\epsilon}(s, \bZ) = \min_{\bt \in \mathbb{T}_{\epsilon}}\; \mathfrak{L}_n(\bt; s, \bZ, \lambda).
\end{align}
We next define a smoothing of the quantity $\mathfrak{M}_{\epsilon}$ to enable the computation of derivatives as required by Theorem~\ref{thm:lindeberg_correlation}.  To this end, let $F: \mathbb{R} \rightarrow \mathbb{R}$ be a thrice continuously differentiable function with the properties\footnote{Although we will not require an explicit expression for $F(x)$, one such function that satisfies the required properties is $F(x) = \zeta(x+1) + \zeta(x) - 1$, where $\zeta(x) = 140\int_{x}^{1} u^3(1-u)^3 \mathrm{d}u$.}
\begin{align}
	\label{prop:F_prop}
	F(x) = \begin{cases} 1, \; x \leq -1\\ -1, \; x \geq 1 \end{cases} \qquad \text{ and } \qquad \Bigl\|\frac{\partial^i}{\partial x^i} F(x) \Bigr\|_{\infty} \leq C' \text{ for } i \in \{1, 2, 3\},
\end{align}
where $C'$ denotes a universal positive constant.  Additionally, consider the collection of random variables $(U_i)_{i=1}^{n} \overset{\mathsf{i.i.d.}}{\sim} \mathsf{Unif}[0, 1]$ and note the relation
\[
y_i \overset{\mathsf{d}}{=} \lim_{r \rightarrow 0}\; F\Bigl(r^{-1}\cdot \bigl[U_i - \rho'(\langle \bx_i, \bt_0 \rangle) \bigr]\Bigr).
\]
Thus, the map $r \mapsto F(r^{-1}\cdot [U_i - \rho'(\langle \bx_i, \bt_0 \rangle) ])$ can be understood as a smoothing of the discrete labels $y_i$.  Use this smoothing to define the smoothed loss $\mathcal{L}_r^{\mathsf{smooth}}: \mathbb{R}^{p} \times \mathbb{R}^{n \times p} \times \mathbb{R}^{n \times p} \rightarrow \mathbb{R}$ (cf. the loss $\mathcal{L}_n$~\eqref{eq:loss}) as
\begin{align}
	\label{def:L_tilde}
	\mathcal{L}_r^{\mathsf{smooth}}(\bt, \bZ, \bX) = \frac{1}{n} \sum_{i=1}^{n} \rho\Bigl(- F\Bigl(r^{-1}\cdot \bigl[U_i - \rho'(\langle \bx_i, \bt_0 \rangle) \bigr]\Bigr) \cdot \langle \bz_i, \bt \rangle\Bigr) + \frac{\lambda}{2p} \norm{\bt}_2^2.
\end{align}
Use this smoothed loss to define a partition function $Z_{\beta}: \mathbb{R}^{n \times p} \times \mathbb{R}^{n \times p} \rightarrow \mathbb{R}$ as
\begin{align} \label{eq:partition}
	Z_{\beta}(\bZ, \bX) = \sum_{\bt \in \mathbb{T}_{\epsilon}} \exp\Bigl\{-\beta \cdot \bigl(\mathcal{L}_r^{\mathsf{smooth}}(\bt, \bZ, \bX) + s \psi(\bt) \bigr)\Bigr\},
\end{align}
as well as a smoothed minimum $f$ as
	\begin{align}\label{def:f-smoothed-min}
	f(\beta, \bZ, \bX) = -\frac{1}{\beta} \log Z_{\beta}(\bZ, \bX).
\end{align}
Note that the interpretation of $f$ as a smoothed minimum relation comes from the limiting relation 
\[
\lim_{\beta \rightarrow \infty} f(\beta, \bZ, \bX) = \min_{\bt \in \mathbb{T}_{\epsilon}}\; \mathcal{L}_r^{\mathsf{smooth}}(\bt, \bZ, \bX) + s\psi(\bt).
\]
With these definitions in hand, we state two key technical lemmas.  The first bounds the approximation error of the smoothing over both the labels (through the parameter $r$) as well as the minimum (through the parameter $\beta$).  Its proof is provided in Section~\ref{sec:proof-smoothing-label-min}.
\begin{lemma}
	\label{lem:smoothing-label-min}
	Fix positive smoothing parameters $r$ and $\beta$ and scalar $\epsilon$ and consider the smoothed minimum $f$~\eqref{def:f-smoothed-min} as well as the discretized minimum $\mathfrak{M}_{\epsilon}$~\eqref{eq:perturbed-difference-eps}.  Let $h_K: \mathbb{R} \rightarrow \mathbb{R}$ be a thrice continuously differentiable function which satisfies 
	$\sup_{t \in \mathbb{R}} \lvert h_K^{(i)}(t) \rvert \leq K^i$ for $i \in \{1, 2, 3\}$.  There exists a pair of universal positive constants $(c, C)$ such that for all scalars $t \in \mathbb{R}$ the following holds with probability at least $1 - 2e^{-cnr}$.  
	\[
	\bigl \lvert h_K\bigl(\mathfrak{M}_{\epsilon}(s, \bZ) - t\bigr) - h_K\bigl(f(\beta, \bZ, \bX) -t \bigr) \bigr \rvert \leq C \cdot \Bigl[\sqrt{r} + \frac{Kn}{\beta} \cdot \log\Bigl(1 + \frac{2M_1\sqrt{n}}{\epsilon}\Bigr)\Bigr].
	\]
\end{lemma}
The next lemma bounds the third derivatives of the smoothing over both the labels as well as the minimum.  Its proof is provided in Section~\ref{sec:proof-third-deriv}.
\begin{lemma} \label{lem:third-deriv}
	Consider the smoothed minimum $f$~\eqref{def:f-smoothed-min}, defined for smoothing parameters which satisfy the relation $\beta/n > r$. Let $h_K: \mathbb{R} \rightarrow \mathbb{R}$ be a thrice continuously differentiable function with 
	$\sup_{t \in \mathbb{R}} \lvert h_K^{(i)}(t) \rvert \leq K^i$ for $i \in \{1, 2, 3\}$.  There exists a universal positive constant $C$ such that for all indices $k \in \{1, 2, \dots, n\}$ and $\ell \in \{1, 2, \dots, p\}$ as well as scalars $t \in \mathbb{R}$, the following holds.
		\begin{align}
		\label{ineq:res_lem_third_deriv}
	\Bigl \lvert \partial_{Z_{k\ell}}^i \px^{3 - i} \; \bigl[ h_K\bigl(f(\beta, \bZ, \bX) - t\bigr) \bigr] \Bigr \rvert \leq C \cdot \frac{K \beta^2 \cdot \| \bt_0 \|_{\infty}^3 \cdot (\log{n})^{3/2}}{n^3 r^3} \quad \text{ for all } \quad i \in \{0, 1, 2, 3\}.
	\end{align}
\end{lemma}

Equipped with these tools, we proceed to the proof of Lemma~\ref{lem:universality}.

\subsection{Proof of Lemma~\ref{lem:universality}} \label{sec:proof-universality}
We begin with some preliminaries.  First, consider the function $\zeta: \mathbb{R} \rightarrow \mathbb{R}$ defined as \sloppy\mbox{$\zeta(x) = 140 \int_{x}^{1} u^3 \cdot (1 - u)^3 \mathrm{d}u$} and use this to define the functions $h_k^{-}: \mathbb{R} \rightarrow \mathbb{R}$ and $h_k^{+}: \mathbb{R} \rightarrow \mathbb{R}$ as
\[
h_k^-(x) = \zeta\bigl(\min\{1, kx + 1\}_{+} \bigr) \qquad \text{ and } \qquad h_k^+(x) = \zeta\bigl(\min\{1, kx\}_{+}\bigr).
\]
Note that these functions satisfy the following properties
\begin{itemize}
	\item Let $x \in \mathbb{R}$ and $X$ be a random variable.  The following two sandwich relations hold
	\[
	h_k^{-}(x) \leq \mathbbm{1}_{(-\infty, 0)}(x) \leq h_k^{+}(x) \qquad \text{ and } \qquad \E h_k^-(X) \leq \pr\{X < 0\} \leq \E h_k^+(X).
	\]
	\item There is a universal constant $C'$ such that the first three derivatives of $h_k^{-}$ are bounded as
	\[
	\sup_{t \in \mathbb{R}}\; \bigl \lvert h_k^{-, (i)}(t) \bigr \rvert \leq C' \cdot k^{i} \qquad \text{ for } \qquad i \in \{1, 2, 3\}.
	\]
	\item The two functions $h_k^{-}$ and $h_k^{+}$ satisfy the relation
	\[
	h_k^{+}(x) = h_k^{-}(x - 1/k).
	\]
\end{itemize}
We apply the first property of the functions $h_k^{\pm}$ to obtain the inequality
\begin{align} \label{ineq:M_n-upper-bound}
	\pr\Bigl\{\mathfrak{M}_{\mathbb{T}}(s, \bZ) < t\Bigr\} &\leq \E \Bigl\{h_k^+\bigl(\mathfrak{M}_{\mathbb{T}}(s, \bZ) - t \bigr) \Bigr\} = \E \Bigl\{h_k^-\bigl(\mathfrak{M}_{\mathbb{T}}(s, \bZ) - t -k^{-1}\bigr) \Bigr\}.
\end{align}
We next claim the following inequality, deferring its proof to the end of the section
\begin{align}
	\label{ineq:penultimate-universality-lem}
	\Bigl \lvert \E \Bigl\{h_k^-\bigl(\mathfrak{M}_{\mathbb{T}}(s, \bZ) - t - k^{-1}\bigr) \Bigr\} - \E \Bigl\{h_k^-\bigl(\mathfrak{M}_{\mathbb{T}}(s, \bG) - t - k^{-1}\bigr) \Bigr\} \Bigr \rvert \leq C \cdot kn^{-\tau/3} (\log{n})^2.
\end{align}
Substituting the bound~\eqref{ineq:penultimate-universality-lem} into the inequality~\eqref{ineq:M_n-upper-bound}, we obtain the inequality
\begin{align*}
		\pr\Bigl\{\mathfrak{M}_{\mathbb{T}}(s, \bZ) < t\Bigr\} &\leq \E \Bigl\{h_k^-\bigl(\mathfrak{M}_\mathbb{T}(s, \bG) - t - k^{-1}\bigr) \Bigr\} + C \cdot kn^{-\tau/3} (\log{n})^2\\
		&\leq \pr\Bigl\{\mathfrak{M}_{\mathbb{T}}(s, \bG) < t + k^{-1}\Bigr\} + C \cdot kn^{-\tau/3} (\log{n})^2,
\end{align*}
where the final inequality follows by invoking the properties of the function $h_k^-$.  Setting $k = n^{\tau/6}$ and arguing similarly for the right tail yields the desired conclusion.  It remains to prove the inequality~\eqref{ineq:penultimate-universality-lem}.

\paragraph{Proof of the inequality~\eqref{ineq:penultimate-universality-lem}.}
The strategy is to apply the sequence of approximations developed in Section~\ref{sec:aux-universality-preliminaries}.  Recall the quantities $\mathfrak{M}_{\epsilon}(s, \bZ)$~\eqref{eq:perturbed-difference-eps} and note the lower bound
\[
\mathfrak{M}_{\epsilon}(s, \bZ) - \mathfrak{M}_{\mathbb{T}}(s, \bZ) \geq 0,
\]
which holds by definition.  Towards obtaining an upper bound on the same quantity, consider the minimizers
\[
\widehat{\bt}_{\mathbb{T}} = \argmin_{\bt \in \mathbb{T}}\; \mathfrak{L}_n(\bt; s, \bZ, \lambda) \qquad \text{ and } \qquad \widehat{\bt}_{\mathbb{T}_{\epsilon}} = \argmin_{\bt \in \mathbb{T}_{\epsilon}}\; \mathfrak{L}_n(\bt; s, \bZ, \lambda),
\]
and let $\pi: \mathbb{T} \rightarrow \mathbb{T}_{\epsilon}$ denote the projection of a vector $\bt$ onto the $\epsilon$-net $\mathbb{T}_{\epsilon}$.  Expanding the definitions of $\mathfrak{M}_{\epsilon}(s, \bZ)$ and $\mathfrak{M}_{\mathbb{T}}(s, \bZ)$ yields the upper bound
\begin{align*}
	\mathfrak{M}_{\epsilon}(s, \bZ) - \mathfrak{M}_{\mathbb{T}}(s, \bZ) &= \mathcal{L}_n\bigl(\widehat{\bt}_{\mathbb{T}_{\epsilon}}\bigr) - \mathcal{L}_n\bigl(\widehat{\bt}_{\mathbb{T}}\bigr) + s \cdot \psi\bigl(\widehat{\bt}_{\mathbb{T}_{\epsilon}}\bigr) - s \cdot \psi\bigl(\widehat{\bt}_{\mathbb{T}}\bigr)\\
	&\leq \mathcal{L}_n\bigl(\pi\bigl(\widehat{\bt}_{\mathbb{T}}\bigr)\bigr) - \mathcal{L}_n\bigl(\widehat{\bt}_{\mathbb{T}}\bigr) + s \cdot \psi\bigl(\pi\bigl(\widehat{\bt}_{\mathbb{T}}\bigr)\bigr) - s \cdot \psi\bigl(\widehat{\bt}_{\mathbb{T}}\bigr),
\end{align*}
where the inequality follows by minimality of $\widehat{\bt}_{\mathbb{T}_{\epsilon}}$ over the set $\mathbb{T}_{\epsilon}$.  Next, straightforward calculation implies that both $\mathcal{L}_n$ as well as $\psi$ are Lipschitz continuous with Lipschitz constant bounded as $C/\sqrt{p}$, whence we obtain the upper bound
\[
\mathfrak{M}_{\epsilon}(s, \bZ) - \mathfrak{M}_{\mathbb{T}}(s, \bZ) \leq C \cdot \frac{\epsilon}{\sqrt{p}}.
\]
We note that in the inequality above, we have implicitly used the fact that $s \leq 1$.
Consequently, we apply the upper bound on the first derivative of the function $h_k^-$ to deduce the inequality
\[
\Bigl \lvert h_k^-\bigl(\mathfrak{M}_{\epsilon}(s, \bZ) - t - k^{-1}\bigr) - h_k^{-}\bigl(\mathfrak{M}_{\mathbb{T}}(s, \bZ) - t - k^{-1}\bigr)\Bigr \rvert \leq C \cdot \frac{k \epsilon}{\sqrt{p}}.
\]
Subsequently, we apply Lemma~\ref{lem:smoothing-label-min} in conjunction with the triangle inequality to obtain the inequality 
\[
\Bigl \lvert h_k^{-}\bigl( f(\beta, \bZ, \bX)  - t - k^{-1}\bigr) - h_k^{-}\bigl(\mathfrak{M}_{\mathbb{T}}(s, \bZ) - t - k^{-1}\bigr)\Bigr \rvert \leq C \cdot \Bigl[\sqrt{r} + \frac{kn}{\beta} \cdot \log\Bigl(1 + \frac{2M_1 \sqrt{n}}{\epsilon}\Bigr) + \frac{k \epsilon}{\sqrt{p}}\Bigr].
\]
Next, we invoke Lemma~\ref{lem:third-deriv} in conjunction with Theorem~\ref{thm:lindeberg_correlation} to obtain the inequality 
\[
\E\Bigl \{\bigl \lvert h_k^{-}\bigl( f(\beta, \bZ, \bX)  - t -  k^{-1}\bigr) - h_k^{-}\bigl( f(\beta, \bG, \bX)  - t  - k^{-1}\bigr) \bigr \rvert\Bigr\} \leq C \cdot \frac{k \beta^2 \cdot \| \bt_0 \|_{\infty}^3 \cdot (\log{n})^{3/2}}{n^{5/2} r^3}.
\]
Combining the previous three displays yields the inequality 
\begin{align*}
\Bigl \lvert h_k^-\bigl(\mathfrak{M}_{\mathbb{T}}(s, \bZ) - t  - k^{-1}\bigr) - h_k^{-}\bigl(\mathfrak{M}_{\mathbb{T}}(s, \bG) - t - k^{-1}\bigr)\Bigr \rvert \leq\; &  C \cdot \Bigl[\sqrt{r} + \frac{kn}{\beta} \cdot \log\Bigl(1 + \frac{2M_1 \sqrt{n}}{\epsilon}\Bigr) + \frac{k \epsilon}{\sqrt{p}}\\
&+\frac{k \beta^2 \cdot \| \bt_0 \|_{\infty}^3 \cdot (\log{n})^{3/2}}{n^{5/2} r^3}\Bigr].
\end{align*}
Balancing terms, we take $\beta = n^{1  + \tau/3},\ \epsilon = k^{-1}$,\ and $r = n^{-2\tau/3} (\log{n})^{3/2}$, and the result follows immediately. \qed

\subsection{Proof of Theorem~\ref{thm:lindeberg_correlation}} \label{sec:proof-lindeberg-correlation}
The proof is nearly identical to that of~\citet[Theorem 2]{korada2011applications}, and follows a simple swapping argument~\citep{lindeberg1922neue,chatterjee2006generalization}.  To this end, define the random vector $W_i$ as well as the random function $\widebar{W}_i: \mathbb{R}^2 \rightarrow \mathbb{R}$ as
\begin{align*}
	W_i &= (x_1, z_1, \dots, x_i, z_i, \widebar{x}_{i+1},\wz_{i+1}, \dots, \wx_{n}, \wz_{n}) \qquad \text{ and }\\
	\widebar{W}_i(t_1,t_2) &= (x_1, z_1, \dots, x_{i-1}, z_{i-1}, t_1, t_2, \wx_{i+1},\wz_{i+1}, \dots \wx_n, \wz_n).
\end{align*}
Next, decompose the difference $\E f(\bx, \bz) - \E f(\wbx, \wbz)$ into the telescoping series
\begin{align*}
	\E f(\bx, \bz) - \E f(\wbx, \wbz) = \sum_{i=1}^{n} \E f(W_i) - \E f(W_{i-1}) \quad \text{ for all functions } f: \mathbb{R}^{2n} \rightarrow \mathbb{R}.
\end{align*}
Now, note that for an index $j \in \{1, 2, \dots, 2n\}$, the notation $\partial_{j}$ denotes partial differentiation with respect to the $j$th coordinate.  Applying Taylor's theorem yields the expansion
\begin{align*}
	f(W_i) = f\bigl(\widebar{W}_i(x_i, 0)\bigr) + z_i \partial_{2i} f\bigl(\widebar{W}_i(x_i, 0)\bigr) + \frac{1}{2} z_i^2 \partial_{2i}^2 f\bigl(\widebar{W}_i(x_i, 0)\bigr) + \frac{1}{6}z_i^3 \partial_{2i}^3f\bigl(\widebar{W}_i(x_i, u_1)\bigr),
\end{align*}
for some scalar $u_1 \in [0, z_i]$.  Further applying Taylor's theorem to the functions $t \mapsto f\bigl(\widebar{W}_i(t, 0)\bigr)$, $t \mapsto \partial_{2i} f\bigl(\widebar{W}_i(t, 0)\bigr)$, and $t \mapsto \partial_{2i}^2 f\bigl(\widebar{W}_i(t, 0)\bigr)$ yields
\begin{align*}
	f(W_i) &= f\bigl(\widebar{W}_i(0, 0)\bigr) + x_i \partial_{2i - 1} f\bigl(\widebar{W}_i(0, 0)\bigr) + \frac{x_i^2}{2} \partial_{2i-1}^2f\bigl(\widebar{W}_i(0, 0)\bigr)  + \frac{x_i^3}{6} \partial_{2i -1}^3f\bigl(\widebar{W}_i(u_2, 0)\bigr) \\
	&\quad + z_i \cdot \Bigl[\partial_{2i} f\bigl(\widebar{W}_i(0, 0)\bigr) + x_i  \partial_{2i-1}\partial_{2i} f\bigl(\widebar{W}_i(0, 0)\bigr) + \frac{x_i^2}{2} \partial_{2i-1}^2\partial_{2i}f\bigl(\widebar{W}_i(u_3, 0)\bigr) \Bigr]\\
	&\quad + \frac{1}{2}z_i^2 \cdot \Bigl[\partial_{2i}^2 f\bigl(\widebar{W}_i(0, 0)\bigr) + x_i \partial_{2i -1}\partial_{2i}^2 f\bigl(\widebar{W}_i(u_4, 0)\bigr) \Bigr] + \frac{1}{6}z_i^3 \partial_{2i}^3f\bigl(\widebar{W}_i(x_i, u_1)\bigr),  
\end{align*}
for some scalars $u_2$, $u_3$, and $u_4$ contained in the interval $[0, x_i]$.  Note that the random variable $f\bigl(\widebar{W}_i(0, 0)\bigr)$ is independent of both $x_i$ as well as $z_i$.  Consequently, we deduce the equation
\begin{align*}
	\E f(W_i) &= \E \bigl\{f\bigl(\widebar{W}_i(0, 0)\bigr)\bigr\} + \frac{1}{2}\E\bigl\{x_i^2\bigr\} \E\bigl\{\partial_{2i-1}^2f\bigl(\widebar{W}_i(0, 0)\bigr)\bigr\}  + \frac{1}{6}\E\bigl\{x_i^3\cdot \partial_{2i -1}^3f\bigl(\widebar{W}_i(u_2, 0)\bigr)\bigr\} \\
	&\quad + \E\bigl\{ x_i z_i \bigr\} \E\bigl\{ \partial_{2i-1}\partial_{2i} f\bigl(\widebar{W}_i(0, 0)\bigr)\bigr\} + \frac{1}{2}\E\bigl\{x_i^2 z_i \cdot \partial_{2i-1}^2\partial_{2i}f\bigl(\widebar{W}_i(u_3, 0)\bigr)\bigr\} \\
	&\quad + \frac{1}{2}\E\bigl\{z_i^2\bigr\} \E\bigl\{\partial_{2i}^2 f\bigl(\widebar{W}_i(0, 0)\bigr)\bigr\} + \frac{1}{2}\E\bigl\{z_i^2 x_i \cdot \partial_{2i -1}\partial_{2i}^2 f\bigl(\widebar{W}_i(u_4, 0)\bigr)\bigr\} + \frac{1}{6}\E\bigl\{z_i^3 \cdot \partial_{2i}^3 f\bigl(\widebar{W}_i(x_i, u_1)\bigr)\bigr\}.
\end{align*}
We proceed in a parallel manner to evaluate the quantity $\E f(W_{i-1})$, replacing the values $x_i$ and $z_i$ with $\widebar{x}_i$ and $\widebar{z}_i$, respectively.  Under assumptions (a)--(c), we thus obtain that the difference $\E f(W_i) - \E f(W_{i-1})$ consists solely of the terms consisting of third-order partial derivatives.  Thus, we apply the triangle inequality in conjunction with the absolute bound on the third-order partial derivatives of the function $f$ to obtain the bound
\[
\bigl\lvert \E f(W_i) - \E f(W_{i-1}) \bigr\rvert \lesssim L \cdot \Bigl( \max_{\mathfrak{p} = 0, 1, 2, 3} (\E \lvert x_i^{\mathfrak{p}}z_i^{3-\mathfrak{p}} \rvert + \E \vert \bar{x}_i^{\mathfrak{p}} \bar{z}_i^{3-\mathfrak{p}}\rvert\Bigr).
\]
The result follows upon summing over $i$. \qed

\subsection{Proof of Lemma~\ref{lem:smoothing-label-min}} \label{sec:proof-smoothing-label-min}
We first require a technical lemma, whose proof we provde in Section~\ref{sec:proof-r-approx}.
\begin{lemma}
	\label{lem:r-approx}
	Fix a smoothing parameter $r > 0$ and recall the loss $\mathcal{L}_r^{\mathsf{smooth}}$~\eqref{def:L_tilde}, the original loss $\mathcal{L}_n$~\eqref{eq:loss}, and the set $\mathbb{T}$~\eqref{eq:T-def}.  There exists a pair of positive constants $(c, C)$ such that with probability at least $1 - 2e^{-cnr}$, the following holds uniformly for all $\bt \in \mathbb{T}$
	\[
	\bigl \lvert \mathcal{L}_n(\bt; \bZ, \lambda) - \mathcal{L}_r^{\mathsf{smooth}}(\bt, \bZ, \bX) \bigr \rvert \leq C\sqrt{r}.
	\]
\end{lemma}
Next, define the smoothed minimum with respect to the perturbed loss~\eqref{eq:perturbed-loss} (cf. the smoothed minimum $f$~\eqref{def:f-smoothed-min}) as
\begin{align}
	g(\beta, \bZ, \bX) = - \frac{1}{\beta} \log\sum_{\bt \in \mathbb{T}_{\epsilon}}\exp\Bigl\{-\beta \cdot \bigl(\mathcal{L}_n(\bt; \bZ, \lambda) + s\psi(\bt)\bigr)\Bigr\},
\end{align}
and note that $\lim_{\beta \rightarrow \infty} g(\beta, \bZ, \bX) = \mathfrak{M}_{\epsilon}(s, \bZ)$.  Consequently, we obtain the inequality
\begin{align}
	\bigl \lvert \mathfrak{M}_{\epsilon}(s, \bZ) - g(\beta, \bZ, \bX) \bigr \rvert \leq \int_{\beta}^{\infty} \Bigl \lvert \frac{\partial}{\partial \beta}\; g(\beta, \bZ, \bX) \Bigr \rvert \mathrm{d} t &\leq 2 \log \bigl(\lvert \mathbb{T}_{\epsilon} \rvert \bigr) \cdot \int_{\beta}^{\infty} \frac{1}{t^2} \mathrm{d} t\nonumber\\
	&\leq \frac{2n}{\beta} \cdot \log\Bigl(1 + \frac{2M_1\sqrt{n}}{\epsilon}\Bigr),\label{ineq:first-r-approx}
\end{align}
where the final inequality follows by applying~\citet[Corollary 4.2.13]{vershynin2018high} to upper bound the size of the $\epsilon$-net $\lvert \mathbb{T}_{\epsilon} \rvert$.
Next, by definition, we obtain the relation
\[
\left \lvert g(\beta, \bZ, \bX) - f(\beta, \bZ, \bX) \right \rvert = \frac{1}{\beta}\left \lvert \log\frac{\sum_{\bt \in \mathbb{T}_{\epsilon}}\exp\{-\beta (\mathcal{L}_n(\bt; \bZ, \lambda) + s\psi(\bt))\}}{\sum_{\bt \in \mathbb{T}_{\epsilon}}\exp\{-\beta (\mathcal{L}_r^{\mathsf{smooth}}(\bt, \bZ, \bX) + s\psi(\bt))\}}\right \rvert.
\]
Let $Z = \sum_{\bt \in \mathbb{T}_{\epsilon}}\exp\{-\beta (\mathcal{L}_r^{\mathsf{smooth}}(\bt, \bZ, \bX) + s\psi(\bt))\}$ and decompose the argument of the logarithm as
\[
\frac{1}{Z}\sum_{\bt \in \mathbb{T}_{\epsilon}}\exp\{-\beta (\mathcal{L}_r^{\mathsf{smooth}}(\bt, \bZ, \bX) + s\psi(\bt))\} \cdot \exp\{-\beta (\mathcal{L}_n(\bt; \bZ, \lambda) -\mathcal{L}_r^{\mathsf{smooth}}(\bt, \bZ, \bX))\} \leq e^{C\beta\sqrt{r}},
\]
where the inequality follows upon applying Lemma~\ref{lem:r-approx}.
Combining the preceding two displays with Lemma~\ref{lem:r-approx}, we deduce $
\left \lvert g(\beta, \bZ, \bX) - f(\beta, \bZ, \bX) \right \rvert \leq  C\sqrt{r}$, and the result follows upon combining with the inequality~\eqref{ineq:first-r-approx}. \qed

\subsubsection{Proof of Lemma~\ref{lem:r-approx}}\label{sec:proof-r-approx}
Applying the triangle inequality and exploiting the $1$-Lipschitz nature of $\rho$, we obtain
\begin{align*}
	\bigl \lvert \mathcal{L}_n(\bt; \bZ, \lambda) - \mathcal{L}_r^{\mathsf{smooth}}(\bt, \bZ, \bX) \bigr \rvert &\leq \frac{1}{n} \sum_{i=1}^{n} \Bigl\lvert\rho\Bigl(- F\Bigl(r^{-1}\cdot \bigl[U_i - \rho'(\langle \bx_i, \bt_0 \rangle) \bigr]\Bigr) \cdot \langle \bz_i, \bt \rangle\Bigr) - \rho(-y_i \cdot \langle \bz_i, \bt \rangle) \Bigr \rvert\\
	&\leq  \frac{1}{n} \sum_{i=1}^{n} \Bigl\lvert y_i - F\Bigl(r^{-1}\cdot \bigl[U_i - \rho'(\langle \bx_i, \bt_0 \rangle) \bigr]\Bigr) \Bigr \rvert \cdot \bigl \lvert \langle \bz_i, \bt \rangle \bigr \rvert.
\end{align*}
Subsequently, we apply the Cauchy--Schwarz inequality as well as the first property~\eqref{prop:F_prop} of the function $F$ to obtain the bound 
\begin{align*}
		\bigl \lvert \mathcal{L}_n(\bt; \bZ, \lambda) - \mathcal{L}_r^{\mathsf{smooth}}(\bt, \bZ, \bX) \bigr \rvert &\leq \sqrt{\frac{1}{n}\sum_{i=1}^{n} 4\mathbbm{1}\bigl\{\lvert U_i - \rho'(\langle \bx_i, \bt_{0}\rangle)\rvert \leq r\bigr\}}\sqrt{\frac{1}{n}\sum_{i=1}^{n}\lvert \langle \bz_i, \bt \rangle\rvert^2} \\
		&\leq M_0 \cdot M_1 \cdot \sqrt{\frac{1}{n}\sum_{i=1}^{n} 4\mathbbm{1}\bigl\{\lvert U_i - \rho'(\langle \bx_i, \bt_{0}\rangle)\rvert \leq r\bigr\}},
\end{align*}
where the final inequality follows, with probability at least $1-2e^{-n}$, from Lemma~\ref{lem:l2_norm_theta_hat}(a) in conjunction with the fact that $\| \bt \|_2 \leq M_1 \sqrt{n}$ on the set $\mathbb{T}$.  Next, note that 
\begin{align*}
	\E_{U_i}\mathbbm{1}\bigl\{\lvert U_i - \rho'(\langle \bx_i, \bt_{0}\rangle)\rvert \leq r\bigr\} = \int_{(\rho'(\< \bx_i, \bt_0 \>) - r) \vee 0}^{(\rho'(\< \bx_i, \bt_0 \>) + r) \wedge 1} 1 \mathrm{d}u \leq 2r.
\end{align*}
Additionally, applying~\citet[Theorem 2.8.4]{vershynin2018high} yields the probabilistic inequality
\begin{align*}
	\pr\biggl\{\Bigl \lvert\frac{1}{n}\sum_{i=1}^{n} 4\mathbbm{1}(\lvert U_i - \rho'(\langle \bx_i, \bt_{0}\rangle)\rvert \leq r) - 4\E\mathbbm{1}(\lvert U_i - \rho'(\langle \bx_i, \bt_{0}\rangle)\rvert \leq r)\Bigr \rvert \geq t\biggr\} \leq 2 \exp\left\{-\frac{nt^2}{r + 4t/3}\right\}.
\end{align*}
Putting the pieces together, we take $t = r$ in the inequality above to obtain the desired result. \qed

\subsection{Proof of Lemma~\ref{lem:third-deriv}} \label{sec:proof-third-deriv}
This lemma consists, primarily, of a lengthy third derivative computation.  To facilitate the computation, for a function $g: \mathbb{T}_{\epsilon} \rightarrow \mathbb{R}$, we will use the shorthand
\begin{align} \label{eq:bracket-notation}
\left\langle g \right\rangle = \frac{\sum_{\bt \in \mathbb{T}_{\epsilon}}g(\bt)\exp\Bigl\{-\beta \cdot \bigl(\mathcal{L}_r^{\mathsf{smooth}}(\bt, \bZ, \bX) + s \psi(\bt) \bigr)\Bigr\}}{\sum_{\bt \in \mathbb{T}_{\epsilon}} \exp\Bigl\{-\beta \cdot \bigl(\mathcal{L}_r^{\mathsf{smooth}}(\bt, \bZ, \bX) + s \psi(\bt) \bigr)\Bigr\}},
\end{align}
to denote the expectation of the function $g$ with respect to the Gibbs measure associated to the partition function $Z_{\beta}$~\eqref{eq:partition}.  

Throughout this section, we will be taking derivatives with respect to the random variables $Z_{k\ell}$ and $X_{k\ell}$.  We conserve notation and write 
\[
\partial_{Z} := \partial_{Z_{k\ell}} \qquad \text{ and } \qquad \partial_{X} := \partial_{X_{k\ell}}.
\]
Additionally, let $\partial_{Z}^{i}$ denote the $i$th partial derivative with respect to $Z_{k\ell}$, $\partial_{X}^{i}$ the $i$th partial derivative with respect to $X_{k\ell}$, and recall that $h_K^{(i)}$ denotes the $i$th derivative of $h_K$.  With this notation, we require upper bounds on the quantities 
\begin{align*}
\partial_{Z}^3 \bigl[h_K\bigl(f(\beta, \bZ, \bX) - t\bigr)\bigr], \qquad \partial_{Z}^2\partial_{X}\bigl[h_K\bigl(f(\beta, \bZ, \bX) - t\bigr)\bigr], \qquad \\
\qquad \partial_{X}^3 \bigl[h_K\bigl(f(\beta, \bZ, \bX) - t\bigr)\bigr] \qquad \text{ and } \qquad \partial_{Z}\partial_{X}^2 \bigl[h_K\bigl(f(\beta, \bZ, \bX) - t\bigr)\bigr].
\end{align*}
To conserve notation, we write $h_K^{(i)}\bigl(f(\beta, \bZ, \bX) - t \bigr)$ as $h_K^{(i)}$.  Straightforward computation subsequently yields
\begin{align*}
		\partial_{Z}^3 \bigl[h_K\bigl(f(\beta, \bZ, \bX) - t\bigr)\bigr] &= (\partial_{Z}^3 f) \cdot ( h_K^{(1)})+(\partial_{Z} f)^3 \cdot  (h_K^{(3)}) + 3\cdot (\partial_{Z}^2f) \cdot (\partial_{Z}f) \cdot (h_K^{(2)})\\
		\partial_{X}^3 \bigl[h_K\bigl(f(\beta, \bZ, \bX) - t\bigr)\bigr] &= (\partial_{X}^3 f) \cdot ( h_K^{(1)})+(\partial_{X} f)^3 \cdot  (h_K^{(3)}) + 3\cdot (\partial_{X}^2f) \cdot (\partial_{X}f) \cdot (h_K^{(2)})\\
		\partial_{Z}^2\partial_{X}\bigl[h_K\bigl(f(\beta, \bZ, \bX) - t\bigr)\bigr] &= (\partial_{X}\partial_{Z}^2 f)\cdot(h_K^{(1)}) +(\partial_{X} f)\cdot(\partial_{Z}^2f)\cdot(h_K^{(2)}) + (\partial_{X}f)\cdot(\partial_{Z}f)^2 \cdot (h_K^{(3)}) \nonumber\\
		&\quad + 2(\partial_{Z}f)\cdot (\partial_{X}\partial_{Z}f)\cdot (h_K^{(2)})\\
	\partial_{Z}\partial_{X}^2 \bigl[h_K\bigl(f(\beta, \bZ, \bX) - t\bigr)\bigr] &= (\partial_{X}^2\partial_{Z} f) \cdot (h_K^{(1)}) +(\partial_{X}^2 f) \cdot (\partial_{Z}f) \cdot (h_K^{(2)}) + (\partial_{X}f)^2 \cdot (\partial_{Z}f) \cdot (h_K^{(3)}) \nonumber \\
		&\quad + 2(\partial_{X}f) \cdot (\partial_{X}\partial_{Z}f) \cdot (h_K^{(2)}) 
\end{align*}
Next, we recall the notation~\eqref{eq:bracket-notation} and compute
\[
\partial_{Z}(f) = \bigl \langle \partial_{Z} \mathcal{L}_r^{\mathsf{smooth}} \bigr\rangle \qquad \text{ and } \qquad \partial_{X}(f) = \bigl \langle \partial_{X} \mathcal{L}_r^{\mathsf{smooth}} \bigr\rangle.
\]
To conserve space, for the remainder of the proof, we will write $\mathcal{L}_r$ in place $\mathcal{L}_r^{\mathsf{smooth}}$.  
We compute the higher order derivatives of $f$ through the pair of identities (which hold for any sufficiently smooth function $g: \mathbb{T}_{\epsilon} \rightarrow \mathbb{R}$)
	\begin{align*}
	\partial_{Z} \langle g \rangle &= \bigl\langle \partial_{Z}(g) \bigr\rangle - \beta \bigl \langle g \cdot  \partial_{Z}(\mathcal{L}_r)\bigr \rangle + \beta \langle g \rangle \cdot  \bigl\langle \partial_{Z}(\mathcal{L}_r)\bigr\rangle \qquad \text{ and }\\
	 	\partial_{X} \langle g \rangle &= \bigl\langle \partial_{X}(g) \bigr\rangle - \beta \bigl \langle g \cdot  \partial_{X}(\mathcal{L}_r)\bigr \rangle + \beta \langle g \rangle \cdot  \bigl\langle \partial_{X}(\mathcal{L}_r)\bigr\rangle.
\end{align*}
Repeated application of the above two identities yields the second order partial derivatives
\begin{align*}
	\partial_{Z}^2 (f) &= \bigl\langle \partial_{Z}^2 (\mathcal{L}_r) \bigr\rangle - \beta \cdot \bigl\langle \bigl[\partial_{Z}( \mathcal{L}_r)\bigr]^2 \rangle + \beta \cdot \bigl \langle \partial_{Z}( \mathcal{L}_r) \bigr\rangle^2\nonumber\\
	\partial_{X}^2 (f) &= \bigl\langle \partial_{X}^2 (\mathcal{L}_r) \bigr\rangle - \beta \cdot \bigl\langle \bigl[\partial_{X}(\mathcal{L}_r)\bigr]^2 \bigr\rangle + \beta \cdot \bigl\langle \partial_{X} (\mathcal{L}_r) \bigr\rangle^2\nonumber\\
	\partial_{Z}\partial_{X} (f) &= \bigl\langle \partial_{X}\partial_{Z} (\mathcal{L}_r) \bigr\rangle - \beta  \cdot \bigl\langle \bigl[\pxn (\wtln)\bigr] \cdot \bigl[\pzn (\wtln)\bigr] \bigr\rangle + \beta \cdot \bigl \langle \pzn( \wtln) \bigr\rangle \cdot \bigl \langle \pxn(\wtln) \bigr\rangle\nonumber.
\end{align*}
Applying the same pair of identities to each term on the RHS of each equation in the display above yields the third order partial derivatives
\begin{align*}
	\partial_{Z}\partial_{X}^2 (f) &= \bigl\langle \pzn \pxn^2 (\wtln) \bigr\rangle - \beta \cdot \bigl \langle \bigl[\pxn^2( \wtln)\bigr] \cdot \bigl[\pzn (\wtln)\bigr] \bigr\rangle + \beta \cdot \bigl \langle \pxn^2 (\wtln)\bigr \rangle \cdot \bigl\langle \pzn (\wtln) \bigr\rangle   + 2\beta \cdot \bigl \langle \pxn( \wtln)  \bigr\rangle \cdot \pzn\pxn (f) \nonumber\\
	&\quad+ \beta^2 \cdot \bigl\langle \bigl[\pzn( \wtln)\bigr] \cdot \bigl[\pxn (\wtln)\bigr]^2 \bigr\rangle- \beta^2 \cdot \bigl\langle \bigl[\pxn (\wtln)\bigr]^2 \bigr\rangle \cdot \bigl\langle \pzn( \wtln) \bigr \rangle - 2\beta\cdot  \bigl\langle \bigl[\pxn (\wtln)\bigr] \cdot \bigl[\pzn \pxn (\wtln)\bigr] \bigr\rangle \nonumber \\
	\partial_{Z}^2\partial_{X} (f) &= \bigl\langle \pzn^2 \pxn (\wtln)\bigr\rangle - \beta  \cdot \bigl \langle \bigl[\pxn( \wtln)\bigr] \cdot \bigl[\pzn^2( \wtln)\bigr] \bigr\rangle + \beta \cdot \bigl \langle \pxn( \wtln) \bigr\rangle \cdot \bigl\langle \pzn^2 (\wtln) \bigr\rangle  - 2\beta\cdot \bigl\langle \bigl[\pzn( \wtln)\bigr] \cdot \bigl[\pzn \pxn (\wtln)\bigr] \bigr\rangle \nonumber\\
	&\quad + \beta^2 \cdot \bigl\langle \bigl[\pxn( \wtln)\bigr] \cdot \bigl[ \pzn( \wtln)\bigr]^2 \bigr\rangle- \beta^2 \cdot \bigl \langle \bigl[\pzn( \wtln)\bigr]^2 \bigr\rangle \cdot \bigl \langle \pxn (\wtln) \bigr\rangle + 2\beta \cdot \bigl\langle \pzn (\wtln) \bigr\rangle \cdot \pzn\pxn (f) \nonumber\\
	\pzn^3 (f) &= \bigl\langle \pzn^3 (\wtln) \bigr\rangle - 3\beta \cdot \bigl\langle \bigl[\pzn (\wtln)\bigr] \cdot \bigl[\pzn^2 (\wtln)\bigr] \bigr\rangle + \beta \cdot  \bigl\langle \pzn( \wtln) \bigr\rangle \cdot \bigl \langle \pzn^2 (\wtln) \bigr\rangle + \beta^2 \cdot \langle \bigl[\pzn (\wtln)\bigr]^3 \bigr\rangle \\
	&\quad - \beta^2 \cdot \bigl \langle \pzn (\wtln) \bigr\rangle \cdot \bigl \langle \bigl[\pzn (\wtln)\bigr]^2 \rangle + 2 \beta\cdot \bigl\langle \pzn( \wtln) \bigr\rangle\cdot \pzn^2 (f)\\
	\pxn^3 (f) &= \bigl\langle \pxn^3 (\wtln) \bigr\rangle - 3\beta \cdot \bigl\langle \bigl[\pxn (\wtln)\bigr] \cdot \bigl[\pxn^2 (\wtln)\bigr] \bigr\rangle + \beta \cdot  \bigl\langle \pxn( \wtln) \bigr\rangle \cdot \bigl \langle \pxn^2 (\wtln) \bigr\rangle + \beta^2 \cdot \langle \bigl[\pxn (\wtln)\bigr]^3 \bigr\rangle \\
	&\quad - \beta^2 \cdot \bigl \langle \pxn (\wtln) \bigr\rangle \cdot \bigl \langle \bigl[\pxn (\wtln)\bigr]^2 \rangle + 2 \beta\cdot \bigl\langle \pxn( \wtln) \bigr\rangle\cdot \pxn^2 (f).
\end{align*}
\begin{subequations} 
	Noting that the computation of the partial derivatives of the loss $\mathcal{L}_r$~\eqref{def:L_tilde} is straightforward, we elect to omit the explicit expressions in favor of coarse upper bounds on the expressions.
	To this end, we now collect some uniform bounds.  First, recall the definition of the set $\mathbb{T}$~\eqref{eq:T-def} and note that $\mathbb{T}_{\epsilon}$ is an $\epsilon$-net of $\mathbb{T}$.  We thus note the bounds
	\begin{align}\label{ineq:unif-bound-a}
		\| \bt \|_{\infty} \leq M_2 \cdot \sqrt{\log{n}} \cdot \| \bt_0 \|_{\infty} \quad \text{ and } \quad \| \bZ \bt \|_{\infty} \leq M_3 \cdot \sqrt{\log{n}}, \quad \text{ for all } \quad \bt \in \mathbb{T}_{\epsilon}.
	\end{align}
Additionally, recall the properties of the function $F$~\eqref{prop:F_prop} as well as of the function $\rho$, and note the bounds
\begin{align}\label{ineq:unif-bound-b}
	\sup_{t \in \mathbb{R}}\; \bigl \lvert F^{(i)}(t) \bigr \rvert \leq C' \quad \text{ and } \quad \sup_{t \in \mathbb{R}}\; \bigl \lvert \rho^{(i)}(t) \bigr \rvert \leq 1, \quad \text{ for all } \quad i \in \{1, 2, 3\}.
\end{align}
\end{subequations}
We now apply the uniform bounds~\eqref{ineq:unif-bound-a}--\eqref{ineq:unif-bound-b} to the computation of the derivatives of the loss $\wtln$.  This yields the bounds on the first derivatives
\begin{align*}
	\bigl \lvert \partial_{Z} (\wtln) \bigr \rvert \leq \frac{C \sqrt{\log{n}} \| \bt_0 \|_{\infty}}{n} \qquad \text{ and } \qquad \bigl\lvert \partial_{X}(\wtln)\bigr\rvert \leq \frac{C\| \bt_0 \|_{\infty} \sqrt{\log{n}}}{nr},
\end{align*}
the bounds on the second derivatives
\begin{align*}
	\bigl\lvert \partial_{Z}^2(\wtln)\bigr\rvert \leq \frac{C\log{n} \| \bt_0 \|_{\infty}^2}{n}, \quad \bigl \lvert \partial_{X}^2(\wtln) \bigr\rvert \leq \frac{C \| \bt_0 \|_{\infty}^2 \log{n}}{nr^2}, \quad \text{ and } \quad \bigl \lvert \partial_{Z}\partial_{X} (\wtln) \bigr\rvert \leq \frac{C \| \bt_0 \|_{\infty}^2 \log{n}}{nr},
\end{align*}
and the bounds on the third derivatives
\begin{align*}
	\bigl\lvert\partial_{Z}^3(\wtln) \bigr\rvert &\leq \frac{C(\log{n})^{3/2} \| \bt_0 \|_{\infty}^3}{n}, \qquad \bigl\lvert \pzn \pxn^2 (\wtln) \bigr\rvert \leq \frac{C \| \bt_0 \|_{\infty}^3 (\log{n})^{3/2}}{nr^2}, \\
	\bigl\lvert\partial_{X}^3(\wtln) \bigr\rvert &\leq \frac{C(\log{n})^{3/2} \| \bt_0 \|_{\infty}^3}{nr^3}  \qquad \text{ and } \quad \bigl \lvert \partial_{Z}^2\partial_{X}(\wtln) \bigr \rvert \leq \frac{C \| \bt_0 \|_{\infty}^{3} (\log{n})^{3/2}}{nr}.
\end{align*}
Next, we apply the triangle inequality to obtain the bound
\begin{align*}
	\bigl \lvert \langle g \rangle \bigr \rvert \leq \max_{\bt \in \mathbb{T}_{\epsilon}}\; \bigl \lvert g(\bt) \bigr \rvert, \quad \text{ for all functions } \quad g: \mathbb{T}_{\epsilon} \rightarrow \mathbb{R}.
\end{align*}
The result follows by recalling the assumptions that the smoothing parameters $r$ and $\beta$ satisfy the relation $\beta/n > r$ and putting the pieces together.  \qed

\subsection{Proof of Lemma~\ref{lem:lipschitz-minimizers}}\label{subsec:proof-lipschitz-minimizers}

\paragraph{Proof of part (a):}  Fix two values $\lambda_1$ and $\lambda_2$ and consider the quantities $\psi\bigl(\widehat{\bt}(\lambda_1)\bigr)$ and $\psi\bigl(\widehat{\bt}(\lambda_2)\bigr)$.  Since $\psi: \mathbb{R}^p \rightarrow \mathbb{R}$ is $C/\sqrt{p}$--Lipschitz, we deduce
%
%
\begin{align}\label{ineq:penultimate-lipschitz-psi}
	\Bigl \lvert \psi\bigl(\widehat{\bt}(\lambda_1)\bigr) - \psi\bigl(\widehat{\bt}(\lambda_2)\bigr) \Bigr \rvert \leq  \frac{C}{\sqrt{p}} \cdot \| \widehat{\bt}(\lambda_1) - \widehat{\bt}(\lambda_2) \|_2.
\end{align}
Invoking strong convexity of the loss $\mathcal{L}_n$, we obtain the inequality
\[
\frac{\lambda_1}{2p} \cdot \bigl \| \widehat{\bt}(\lambda_1) - \widehat{\bt}(\lambda_2) \bigr\|_2^2 \leq \mathcal{L}_n\bigl(\widehat{\bt}(\lambda_2); \bZ, \lambda_1\bigr) - \mathcal{L}_n\bigl(\widehat{\bt}(\lambda_1); \bZ, \lambda_1\bigr),
\]
On the other hand, we note the upper bound
\begin{align*}
	\mathcal{L}_n\bigl(\widehat{\bt}(\lambda_2); \bZ, \lambda_1\bigr) = \mathcal{L}_n\bigl(\widehat{\bt}(\lambda_2); \bZ, \lambda_2\bigr) + \frac{\lambda_1 - \lambda_2}{2p} \bigl\| \widehat{\bt}(\lambda_2) \|_2^2 &\leq \mathcal{L}_n\bigl(\widehat{\bt}(\lambda_1); \bZ, \lambda_2\bigr) + \frac{\lambda_1 - \lambda_2}{2p} \bigl\| \widehat{\bt}(\lambda_2) \|_2^2
\end{align*}
where the final inequality follows as $\widehat{\bt}(\lambda_2)$ is a minimizer of the loss $\mathcal{L}_n(\cdot; \bZ, \lambda_2)$.  Moreover, we note the equivalence
\begin{align*}
	\mathcal{L}_n\bigl(\widehat{\bt}(\lambda_1); \bZ, \lambda_2\bigr) + \frac{\lambda_1 - \lambda_2}{2p} \bigl\| \widehat{\bt}(\lambda_2) \|_2^2 &= \mathcal{L}_n\bigl(\widehat{\bt}(\lambda_1); \bZ, \lambda_1\bigr) + \frac{\lambda_1 - \lambda_2}{2p} \cdot \bigl( \bigl\| \widehat{\bt}(\lambda_2) \|_2^2 - \bigl\| \widehat{\bt}(\lambda_1) \|_2^2 \bigr).
\end{align*}
Factoring the difference of squares on the RHS of the preceding display and applying Lemma~\ref{lem:l2_norm_theta_hat}(a) yields the bound
\begin{align*}
	\frac{\lambda_1 - \lambda_2}{2p} \cdot \bigl( \bigl\| \widehat{\bt}(\lambda_2) \|_2^2 - \bigl\| \widehat{\bt}(\lambda_1) \|_2^2 \bigr) &\leq  M_1 \cdot \frac{\lvert \lambda_1 - \lambda_2\rvert}{2\sqrt{p}} \cdot \Bigl\lvert \bigl\| \widehat{\bt}(\lambda_2) \|_2 - \bigl\| \widehat{\bt}(\lambda_1) \|_2 \Bigr\rvert \\
	&\leq M_1 \cdot \frac{\lvert \lambda_1 - \lambda_2 \rvert}{2\sqrt{p}} \cdot \bigl\| \widehat{\bt}(\lambda_2) -  \widehat{\bt}(\lambda_1) \|_2,
\end{align*}
where the first inequality holds with probability at least $1 - 2e^{-n}$ and the second inequality follows by applying the triangle inequality.  Putting the pieces together yields the inequality
\[
\frac{\lambda_1}{2p} \cdot \bigl \| \widehat{\bt}(\lambda_1) - \widehat{\bt}(\lambda_2) \bigr\|_2^2 \leq M_1 \cdot \frac{\lvert \lambda_1 - \lambda_2 \rvert}{2\sqrt{p}} \cdot \bigl\| \widehat{\bt}(\lambda_2) -  \widehat{\bt}(\lambda_1) \|_2.
\]
Re-arranging the above inequality and substituting the resulting bound into the RHS of the inequality~\eqref{ineq:penultimate-lipschitz-psi} yields the desired conclusion. \qed

\paragraph{Proof of part (b):} This proof proceeds in a parallel manner to that of part (a).  Fix two values $\lambda_1$ and $\lambda_2$.  Recall the asymptotic loss $L$~\eqref{eq:asymploss} 
and define the function $\Psi(\cdot; \lambda): \mathbb{R}_{\geq 0} \times \mathbb{R} \rightarrow \mathbb{R}$ as
\[
\Psi(\sigma, \xi; \lambda) = \max_{\gamma \geq 0}\; L(\sigma, \xi, \gamma; \lambda),
\]
where we have made the dependence on the regularization parameter $\lambda$ explicit in both definitions.
Next, introduce the vector valued maps $\bzeta_{\star}: \mathbb{R} \rightarrow \mathbb{R}^2$ and $\bzeta_{\star}^{(R)}: \mathbb{R} \rightarrow \mathbb{R}^2$ as
\[
\bzeta_{\star}(\lambda) = [\sigma_{\star}(\lambda), \;\; \xi_{\star}(\lambda)]^{\top} \qquad \text{ and } \qquad \bzeta_{\star}^{(R)}(\lambda) = [\sigma_{\star}(\lambda), \;\; R \cdot \xi_{\star}(\lambda)]^{\top}.
\]
Now, note that the function $\Psi$ is $\lambda \cdot (1 \wedge R)$-strongly convex, by Lemma~\ref{lem:structural_L}(a).  Consequently, we obtain the inequality
\begin{align}\label{ineq:first-asymp-lipschitz}
	\frac{\lambda_1 \cdot (1 \wedge R)}{2} \cdot \| \bzeta_{\star}(\lambda_1) - \bzeta_{\star}(\lambda_2) \|_2^2 \leq \Psi\bigl(\sigma_{\star}(\lambda_2), \xi_{\star}(\lambda_2); \lambda_1\bigr) - \Psi\bigl(\sigma_{\star}(\lambda_1), \xi_{\star}(\lambda_1); \lambda_1\bigr).
\end{align}
Following similar steps to part (a)---this time using $\Psi\bigl(\sigma_{\star}(\lambda_2), \xi_{\star}(\lambda_2) \cdot; \lambda_1\bigr)$ in place of $\mathcal{L}_n\bigl(\widehat{\bt}(\lambda_2); \bZ, \lambda_1 \bigr)$---yields the upper bound
\begin{align}\label{ineq:second-asymp-lipschitz}
	\Psi\bigl(\sigma_{\star}(\lambda_2), \xi_{\star}(\lambda_2); \lambda_1\bigr) &\leq \Psi\bigl(\sigma_{\star}(\lambda_1), \xi_{\star}(\lambda_1); \lambda_1\bigr) + \frac{\lambda_1 - \lambda_2}{2} \cdot \Bigl(\bigl\| \bzeta_{\star}^{(R)}(\lambda_2) \bigr\|_2^2 - \bigl\| \bzeta_{\star}^{(R)}(\lambda_1) \bigr\|_2^2\Bigr).
\end{align}
Moreover, invoking strong convexity of the function $\Psi$ once more, we deduce the bound
\[
\| \bzeta_{\star}(\lambda) \|_2 \leq \sqrt{\frac{2}{\lambda \cdot (1 \wedge R)} \cdot \Psi(0, 0; \lambda)} \leq \sqrt{\frac{2 \log{2}}{\lambda \cdot (1 \wedge R)}}.
\]
Thus, factoring the difference of squares on the RHS of the inequality~\eqref{ineq:second-asymp-lipschitz} and applying the above inequality, we obtain the bound
\[
\bigl\| \bzeta_{\star}^{(R)}(\lambda_2) \bigr\|_2^2 - \bigl\| \bzeta_{\star}^{(R)}(\lambda_1) \bigr\|_2^2 \leq 2 \sqrt{\frac{2 \log{2}}{\lambda_{\min}}} \cdot  \| \bzeta_{\star}(\lambda_1) - \bzeta_{\star}(\lambda_2)\|_2,
\]
where we have additionally used the inequality $\| \bzeta_{\star}^{(R)}(\lambda) \|_2^2 \leq (1 \vee R) \| \bzeta_{\star} \|_2^2$.
Combining the above inequality with the bounds~\eqref{ineq:first-asymp-lipschitz} and~\eqref{ineq:second-asymp-lipschitz} and re-arranging yields the inequality
\[
\| \bzeta_{\star}(\lambda_1) - \bzeta_{\star}(\lambda_2) \|_2 \leq \frac{2}{\lambda_1 \cdot (1 \wedge R)}\sqrt{\frac{2 \log{2}}{\lambda_{\min}}} \cdot \lvert \lambda_1 - \lambda_2 \rvert.
\]
This concludes the proof since $\phi$ is $1$-Lipschitz. \qed

\section{Geometric properties of the minimizer: proof of Lemma~\ref{lem:l2_norm_theta_hat}}\label{sec:geometric-prop}
This section is organized as follows.  In Section~\ref{sec:proof-thetahat-a}, we prove part (a), in Section~\ref{sec:proof-thetahat-b}, we prove part (b), and in Section~\ref{sec:proof-theta-hat-c}, we prove part (c).

\subsection{Proof of Lemma~\ref{lem:l2_norm_theta_hat}(a)}\label{sec:proof-thetahat-a}
First recall that by Definition~\ref{def:universality}, the entries of the matrix $\bZ$ have Orlicz norm bounded as $\| Z_{ij} \|_{\psi_2} \leq K_1/\sqrt{p}$, for all $i \in \{1, 2, \dots, n\}$ and $j \in \{1, 2 \dots, p\}$.  Since each entry in the matrix $\bZ$ is independent, a straightforward application of~\citet[Theorem 4.4.5]{vershynin2018high} thus furnishes a constant $M_0$ such that $\| \bZ\|_{\mathsf{op}} \leq M_0$, with probability at least $1 - 2e^{-n}$.  Next, note that the estimator $\widehat{\bt}$~\eqref{eq:estimate} satisfies the following KKT condition
\[
0 = -\frac{1}{n} \sum_{i=1}^{n} y_i \bz_i \rho'\bigl(-y_i\langle \bz_i, \widehat{\bt}\rangle\bigr) + \frac{\lambda}{p} \widehat{\bt} = -\frac{1}{n} \bZ^\top \bigl(\by \odot \rho'\bigl(-\by \odot \bZ \widehat{\bt}\bigr)\bigr) + \frac{\lambda}{p} \widehat{\bt}.
\]
Re-arranging, taking the norm of both sides, and applying the high probability bound $\| \bZ \|_{\mathsf{op}} \leq M_0$ yields
Thus, re-arranging and taking the norm of both sides, we obtain the inequality
\begin{align*}
	\|\widehat{\bt}\|_2 = \Bigl\|\frac{p}{\lambda n}\bZ^\top (\by \odot \rho'(-\by \odot \bZ \widehat{\bt}))\Bigr\|_2 \leq \frac{p}{\lambda n} \norm{\bZ}_{\mathsf{op}}\|\by \odot \rho'(\by \odot \bZ \widehat{\bt})\|_2 \leq \frac{\sqrt{n}}{\lambda} \delta M_0.
\end{align*}
The conclusion follows upon setting $M_1 = \lambda \delta M_0$.\qed

\subsection{Proof of Lemma~\ref{lem:l2_norm_theta_hat}(b)}\label{sec:proof-thetahat-b}
Let $\widehat{\theta}_1$ denote the first coordinate of the estimator $\widehat{\bt}$.  Without loss of generality, we will prove 
\begin{align}\label{ineq:ultimate-loco}
	\pr\left\{\lvert \hat{\theta}_1 \rvert \geq C (\log{n})^{3/2} (\| \bt_0 \|_{\infty} \vee \log{n})\right\} \leq \frac{1}{n^2},
\end{align}
and recover the desired result via a union bound.  Towards establishing inequality~\eqref{ineq:ultimate-loco}, we begin with some preliminaries.  First, let $\widetilde{y}_i = (y_i + 1)/2$ and re-parameterize the loss $\mathcal{L}_n$~\eqref{eq:loss} as
\begin{align}
	\label{eq:alt-def-Ln}
	\mathcal{L}_n(\bt; \bZ, \lambda) = -\frac{1}{n}\sum_{i=1}^{n}\widetilde{y}_i \<\bz_i, \bt\> + \frac{1}{n}\sum_{i=1}^{n} \rho(\< \bz_i, \bt\>)  + \frac{\lambda}{2p}\|\bt\|_2^2.
\end{align}
Equipped with this re-parameterization, we isolate the contribution of the coordinate through the function $L_1: \mathbb{R} \rightarrow \mathbb{R}$ defined as 
\begin{align*}
	L_1(\theta_1) = \min_{\theta_2, \theta_3 \dots, \theta_p} \mathcal{L}_n(\bt; \bZ, \lambda), \qquad \text{ where } \qquad \bt = (\theta_1, \theta_2, \dots, \theta_p).
\end{align*}
Additionally, for any scalar $t \in \mathbb{R}$, define $\widehat{\bt}^{(-1)}(t)$ as
\begin{align}
	\label{eq:loco-est}
	\widehat{\bt}^{(-1)}(t) = \argmin_{\theta_2, \theta_3, \dots, \theta_p} \mathcal{L}_n(\bt^{(-1)}(t); \bZ, \lambda), \qquad \text{ where } \qquad \bt^{(-1)}(t) = (t, \theta_2, \dots, \theta_p).
\end{align}
In the sequel, we will drop the argument and refer to $\widehat{\bt}^{(-1)}(0)$ as $\widehat{\bt}^{(-1)}$.  It will additionally be useful to consider the leave one out vectors $\{\bz_i^{(-1)}\}_{i \leq n} = [0 \; Z_{i2} \; \dots \; Z_{ip}] \in \mathbb{R}^p$ as well as the leave one out matrix $\bZ^{(-1)} = [\bz_1^{(-1)}\; \bz_2^{(-1)} \; \dots \; \bz_n^{(-1)}]^{\top} \in \mathbb{R}^{n \times p}$.

Continuing, note that the function $L_1$ is $\lambda/(2p)$--strongly convex~\citep[see, e.g.,][$\S$ 2.4]{HiriartUrrutyLe93} and continuously differentiable, whence
\[
0 \leq L_1(0) - L_1\bigl( \widehat{\theta}_1 \bigr) \leq -L_1'(0) \cdot \widehat{\theta}_1 - \frac{\lambda}{2p} \cdot \bigl(\widehat{\theta}_1\bigr)^2.
\]
Re-arranging yields $\bigl \lvert \widehat{\theta}_1 \bigr \rvert \leq 2p/\lambda \cdot \lvert L_1'(0) \rvert$.  We thus compute
\begin{align*}
	L_1'(t) &= \frac{\lambda}{p} \cdot t  -\frac{1}{n} \sum_{i=1}^{n} \widetilde{y}_i \cdot Z_{i1} + \frac{1}{n}\sum_{i=1}^{n} \rho'\bigl (\bigl\langle \widehat{\bt}^{(-1)}(t), \bz_i \bigr \rangle \bigr) \cdot Z_{i1},
\end{align*}
and subsequently apply the triangle inequality in conjunction with the aforementioned bound $\bigl \lvert \widehat{\theta}_1 \bigr \rvert \leq 2p/\lambda \cdot \lvert L_1'(0) \rvert$ to obtain
\[
\bigl \lvert \widehat{\theta}_1 \bigr \rvert \leq \frac{2}{\lambda \delta} \Bigl \lvert \sum_{i=1}^{n} \widetilde{y}_i Z_{i1} \Bigr \rvert + \frac{2}{\lambda \delta} \Bigl \lvert \sum_{i=1}^{n} Z_{i1} \rho'\bigl(\bigl\langle \bz_{i}^{(-1)}, \widehat{\bt}^{(-1)}\bigr \rangle \bigr) \Bigr \rvert.
\]
The result follows from the following two inequalities, whose proofs are lengthy and deferred to Sections~\ref{sec:proof-empirical-y-tilde} and~\ref{sec:proof-empirical-loco}, respectively. 
\begin{subequations}
	\begin{align}
		\Bigl \lvert \sum_{i=1}^{n} \widetilde{y}_i Z_{i1} \Bigr \rvert & \leq C \cdot \bigl(\| \bt_0 \|_{\infty} \vee \sqrt{\log{n}}\bigr) \quad \text{ with probability } \quad \geq 1 - \frac{1}{n^2},\label{ineq:empirical-y-tilde}\\
		 \Bigl \lvert \sum_{i=1}^{n} Z_{i1} \rho'\bigl(\bigl\langle \bz_{i}^{(-1)}, \widehat{\bt}^{(-1)}\bigr \rangle \bigr) \Bigr \rvert & \leq C  \| \bt_0 \|_{\infty} \sqrt{\log{n}}  \;\;\;\quad \qquad \text{ with probability } \quad \geq 1 - \frac{1}{n^2}. \label{ineq:empirical-loco}
	\end{align}
\end{subequations}
Combining the pieces yields the result. \qed

\subsubsection{Proof of the inequality~\eqref{ineq:empirical-y-tilde}}\label{sec:proof-empirical-y-tilde}
Since $\bZ$ belongs to the $(\alpha_c, \alpha_2)$--universality class and $\widetilde{y}_i$ is bounded, the product $Z_{i1}\widetilde{y}_i$ is sub-exponential, so applying Bernstein's inequality in conjunction with the triangle inequality yields 
\begin{align}
	\label{ineq:prob_linfty_tA}
	\Bigl \lvert \sum_{i=1}^{n} \widetilde{y}_i \cdot Z_{i1} \Bigr \rvert  \leq \Bigl\lvert \E \Bigl\{\sum_{i=1}^{n} \widetilde{y}_i \cdot Z_{i1}\Bigr\} \Bigr \rvert + C \sqrt{\log{n}}, \quad \text{ with probability } \quad \geq 1 - \frac{2}{n^2}.
\end{align}
The remainder of the proof consists of bounding the first term on the RHS.  To this end, note that the distribution of the label $\widetilde{y}_i$ can be characterized as
\begin{align*}
\widetilde{y}_i = \begin{cases} 1, & \text{ if } U_i \leq \rho'(\langle \bx_i, \bt_0 \rangle)\\
	0, & \text{ else},
\end{cases} \qquad \text{ where } \qquad U_i \sim \mathsf{Unif}[0, 1].
\end{align*}
Additionally, let $\bx_{i}^{(-1)} =  [0 \; X_{i2} \; \dots \; X_{ip}]$ denote a row of the data matrix $\bX$ with its first coordinate left out and let $\bt_{0}^{(-1)} = [0 \; \theta_{0, 2} \; \dots \; \theta_{0, p}]$ denote the ground truth with its first coordinate left out.  Use these to define the label
\begin{align}
	\label{def:y_loo}
	\widetilde{y}_i(t) = \begin{cases} 1, & \text{ if } U_i \leq \rho'\bigl(\langle \bx_i^{(-1)}, \bt_0^{(-1)} \rangle + t\cdot \theta_{0, 1}\bigr)\\
		0, & \text{ else},
	\end{cases} \qquad \text{ for any } \qquad t \in \mathbb{R},
\end{align}
and note that the original label is recovered upon considering the quantity $\widetilde{y}_i(X_{i1})$.  Now, let the random variable $\widebar{X}_{i1}$ denote an independent copy of the entry $X_{i1}$ so that the product $Z_{i1} \cdot \widetilde{y}_i\bigl(\widebar{X}_{i1}\bigr)$ is a zero-mean random variable.  Consequently, 
\[
\Bigl \lvert \E \Bigl\{\sum_{i=1}^{n} \widetilde{y}_i \cdot Z_{i1}\Bigr\} \Bigr \rvert = n \bigl \lvert \E  \bigl\{ \bigl(\widetilde{y}_1(X_{11}) - \widetilde{y}_1\bigl(\widebar{X}_{11}\bigr) \bigr) \cdot Z_{11}\bigr\} \bigr \rvert  \leq n\E  \bigl\{\bigl \lvert \bigl(\widetilde{y}_1(X_{11}) - \widetilde{y}_1\bigl(\widebar{X}_{11}\bigr) \bigr) \cdot Z_{i1}\bigr \rvert\bigr\}.
\]
Applying Taylor's theorem with remainder in conjunction with the representation~\eqref{def:y_loo} yields
\begin{align*}
\EE\bigl\{\bigl\lvert\widetilde{y}_1(X_{11}) - \widetilde{y}_1\bigl(\widebar{X}_{11}\bigr) \bigr\rvert \mid \bX, \bZ\bigr\} &= \bigl \lvert  \rho'\bigl(\langle \bx_1^{(-1)}, \bt_0^{(-1)} \rangle + X_{11}\cdot \theta_{0, 1}\bigr) - \rho'\bigl(\langle \bx_1^{(-1)}, \bt_0^{(-1)} \rangle + \widebar{X}_{11}\cdot \theta_{0, 1}\bigr) \bigr\rvert\\
&\leq \| \bt_0 \|_{\infty} \lvert X_{11} - \widebar{X}_{11} \rvert + \frac{\| \bt_0 \|_{\infty}^2}{2}  \lvert X_{11} - \widebar{X}_{11} \rvert^2.
\end{align*}
To obtain the inequality we have additionally upper bounded $\lvert \theta_{0, 1} \rvert \leq \| \bt_0 \|_{\infty}$ and used the fact that the first three derivatives of $\rho$ are uniformly bounded by one. Substituting this inequality into the previous display and applying the Cauchy--Schwarz inequality yields
\[
\Bigl \lvert \E \Bigl\{\sum_{i=1}^{n} \widetilde{y}_i \cdot Z_{i1}\Bigr\} \Bigr \rvert \leq C \| \bt_0 \|_{\infty} + C \| \bt_0 \|_{\infty}^2 \frac{1}{\sqrt{n}} \leq C \| \bt_0 \|_{\infty},
\]
where the final inequality follows since by assumption $\| \bt_0 \|_{\infty} \lesssim n^{1/6}$.  Substituting this into the inequality~\eqref{ineq:prob_linfty_tA} yields the result. 
\qed

\subsubsection{Proof of the inequality~\eqref{ineq:empirical-loco}} \label{sec:proof-empirical-loco}
Note that $\widehat{\bt}^{(-1)}$ still depends on the entry $Z_{i1}$ through the label $\widetilde{y}_i$.  Towards decoupling the two, we use the representation $\widetilde{y}_i(r)$~\eqref{def:y_loo} and define the loss (which depends on the auxiliary variable $\br$)
	\begin{align*}
		\mathfrak{R}_n(\bt; \bZ, \lambda, \br) = -\frac{1}{n}\sum_{i} \widetilde{y}_i(r_i) \bigl\langle \bz_i^{(-1)}, \bt \bigr\rangle + \frac{1}{n}\sum_{i=1}^{n}\rho\Bigl(\bigl\langle \bz_i^{(-1)}, \bt \bigr\rangle\Bigr) + \frac{\lambda}{2p}\| \bt \|_2^2,
	\end{align*}
	as well as its minimizer $\widehat{\btv}(\br) = \argmin_{\btv \in \mathbb{R}^p} \;\mathfrak{R}_n(\bt; \bZ, \lambda, \br)$, noting that $\widehat{\btv}(\bx_{\cdot, 1}) = \widehat{\bt}^{(-1)}$.  By Taylor's theorem with remainder,
	\begin{align*}
	\sum_{i=1}^{n} Z_{i1} \rho'\bigl(\langle \bz_i^{(-1)}, \widehat{\bt}^{(-1)}\rangle\bigr) = \underbrace{\sum_{i=1}^{n} Z_{i1}  \rho'\bigl(\langle \bz_i^{(-1)}, \widehat{\btv}(\boldsymbol{0})\rangle\bigr)}_{=: A} + \underbrace{\Bigl \langle \widehat{\bt}^{(-1)} -  \widehat{\btv}(\boldsymbol{0}), \sum_{i=1}^{n} Z_{i1} \rho''(\zeta_{i}) \bz_{i}^{(-1)}\Bigr \rangle}_{=: B},
	\end{align*}
where the random variable $\zeta_i \in [\langle \bz_i^{(-1)}, \widehat{\btv}(\boldsymbol{0})\rangle, \langle \bz_i^{(-1)}, \widehat{\bt}^{(-1)}\rangle]$.  First bounding term $A$, we note that $(Z_{i1})_{i \leq n}$ are independent and for each $i$, $Z_{i1}$ is independent of $\rho'\bigl(\langle \bz_i^{(-1)}, \widehat{\btv}(\boldsymbol{0})\rangle\bigr)$.  Thus, conditioning and applying Hoeffding's inequality yields 
\begin{align*}
	A = \sum_{i=1}^{n} Z_{i1}  \rho'\bigl(\langle \bz_i^{(-1)}, \widehat{\btv}(\boldsymbol{0})\rangle\bigr) \leq C\sqrt{\log{n}} \qquad \text{ with probability } \qquad \geq 1 - \frac{C}{n^2}.
\end{align*}
Turning to term $B$, we apply the Cauchy--Schwarz inequality to obtain
\begin{align}\label{ineq:bound-B-step1}
B = \Bigl \langle \widehat{\bt}^{(-1)} -  \widehat{\btv}(\boldsymbol{0}), \sum_{i=1}^{n} Z_{i1} \rho''(\zeta_{i}) \bz_{i}^{(-1)}\Bigr \rangle &\leq \bigl \|  \widehat{\bt}^{(-1)} -  \widehat{\btv}(\boldsymbol{0}) \bigr \|_2 \cdot \bigl \| (\bZ^{(-1)})^{\top} \bigl[\bz_{\cdot 1} \odot \rho''(\boldsymbol{\zeta})\bigr] \bigr \|_2 \nonumber\\
&\leq M_0^2 \bigl \|  \widehat{\bt}^{(-1)} -  \widehat{\btv}(\boldsymbol{0}) \bigr \|_2,
\end{align}
where the second inequality follows with probability at least $1 - 2e^{-cn}$ by applying Lemma~\ref{lem:l2_norm_theta_hat}(a) twice.  It remains to bound $\|  \widehat{\bt}^{(-1)} -  \widehat{\btv}(\boldsymbol{0})  \|_2$.  To this end, by strong convexity of $\mathfrak{R}_n(\cdot; \bZ, \lambda, \boldsymbol{0})$, we deduce the sandwich relation
\[
0 \geq \mathfrak{R}_n\bigl(\widehat{\btv}(\boldsymbol{0}); \bZ, \lambda, \boldsymbol{0}\bigr) - \mathfrak{R}_n\bigl(\widehat{\bt}^{(-1)}; \bZ, \lambda, \boldsymbol{0}\bigr) \geq \bigl\langle \nabla \mathfrak{R}_n\bigl(\widehat{\bt}^{(-1)}; \bZ, \lambda, \boldsymbol{0}\bigr), \widehat{\bt}^{(-1)} -  \widehat{\btv}(\boldsymbol{0}) \bigr\rangle + \frac{\lambda}{2p} \bigl \|  \widehat{\bt}^{(-1)} -  \widehat{\btv}(\boldsymbol{0}) \bigr \|_2^2.
\]
Applying Cauchy--Schwarz and re-arranging yields the inequality
\begin{align}\label{ineq:bound-B-step2}
\bigl \|  \widehat{\bt}^{(-1)} -  \widehat{\btv}(\boldsymbol{0}) \bigr \|_2 \leq \frac{2p}{\lambda} \bigl\| \nabla \mathfrak{R}_n\bigl(\widehat{\bt}^{(-1)}; \bZ, \lambda, \boldsymbol{0}\bigr) \bigr\|_2.
\end{align}
From the first order condition $\nabla \mathfrak{R}_n\bigl(\widehat{\bt}^{(-1)}; \bZ, \lambda, \bx_{\cdot, 1}\bigr) = \boldsymbol{0}$, we deduce 
\begin{align} \label{ineq:bound-B-step3}
\bigl\| \nabla \mathfrak{R}_n\bigl(\widehat{\bt}^{(-1)}; \bZ, \lambda, \boldsymbol{0}\bigr) \bigr\|_2 &= \bigl\| \nabla \mathfrak{R}_n\bigl(\widehat{\bt}^{(-1)}; \bZ, \lambda, \boldsymbol{0}\bigr)  - \nabla \mathfrak{R}_n\bigl(\widehat{\bt}^{(-1)}; \bZ, \lambda, \bx_{\cdot, 1}\bigr)\bigr\|_2\nonumber\\
&= \frac{1}{n} \bigl \| \bigl(\bZ^{(-1)}\bigr)^{\top} \bigl(\widetilde{\by}(\bx_{\cdot, 1}) - \widetilde{\by}(\boldsymbol{0})\bigr) \bigr\|_2 \overset{\1}{\leq} \frac{M_0}{n} \bigl \| \widetilde{\by}(\bx_{\cdot, 1}) - \widetilde{\by}(\boldsymbol{0}) \bigr\|_2,
\end{align}
where step $\1$ follows with probability at least $1 - 2e^{-cn}$ upon applying Lemma~\ref{lem:l2_norm_theta_hat}(a) and we have let $\widetilde{\by}(\br) = [\widetilde{y_i}(r_i)]_{i \leq n} \in \{0, 1\}^{n}$.  It remains to bound the norm of the differences in labels $ \| \widetilde{\by}(\bx_{\cdot, 1}) - \widetilde{\by}(\boldsymbol{0}) \|_2$.  From the representation~\eqref{def:y_loo}, we deduce the equivalent expression
\[
\bigl \| \widetilde{\by}(\bx_{\cdot, 1}) - \widetilde{\by}(\boldsymbol{0}) \bigr\|_2^2 = \sum_{i=1}^{n} \mathbbm{1}\Bigl\{U_i \in \bigl[\rho'\bigl(\langle \bx_i^{(-1)}, \bt_0^{(-1)}\rangle\bigr), \rho'\bigl(\langle \bx_i^{(-1)}, \bt_0^{(-1)}\rangle + X_{i1} \theta_{0,1}\bigr)\bigr]\Bigr\} =: \sum_{i=1}^{n} V_i
\]
Note that the summands $V_i$ are bounded and mutually independent.  Thus, we apply Bernstein's inequality for bounded random variables~\citep[][Theorem 2.8.4]{vershynin2018high} to obtain
\begin{align}\label{ineq:bernstein-bounded-Vi}
\pr\Bigl\{ \Bigl \lvert \sum_{i=1}^{n} V_i - \EE V_i \Bigr \rvert \geq t \Bigr\} \leq 2 \exp\Bigl\{ -\frac{t^2/2}{\sum_{i=1}^{n} \EE\{(V_i - \EE[V_i])^2\} + t/3}\Bigr\}.
\end{align}
Note that $\EE\{(V_i - \EE[V_i])^2\} \leq \EE\{V_i^2\} = \EE V_i$ so that it suffices to bound $\EE V_i$.  Applying Taylor's theorem with remainder, we deduce
\begin{align*}
\EE  V_i = \EE\Bigl\{X_{i1} \theta_{01} \rho''\bigl(\langle \bx_i^{(-1)}, \bt_0^{(-1)}\rangle\bigr) + \frac{X_{i1}^2 \theta_{01}^2}{2} \rho'''(\zeta_i)\Bigr\} \overset{\1}{\leq} \frac{C\| \bt_0 \|_{\infty}^2}{n},
\end{align*}
where in the first equality, $\zeta_i$ denotes some element in the interval $\bigl[\rho'\bigl(\langle \bx_i^{(-1)}, \bt_0^{(-1)}\rangle\bigr), \rho'\bigl(\langle \bx_i^{(-1)}, \bt_0^{(-1)}\rangle + X_{i1} \theta_{0,1}\bigr)\bigr]$ and step $\1$ follows from independence of $X_{i1}$ and $\bx_i^{(-1)}$ as well as the second moment condition of Definition~\ref{def:universality}(iv).  Substituting this bound into the inequality~\eqref{ineq:bernstein-bounded-Vi} with $t= C\| \bt_0 \|_{\infty} \log{n}$ and recalling that $\bigl \| \widetilde{\by}(\bx_{\cdot, 1}) - \widetilde{\by}(\boldsymbol{0}) \bigr\|_2^2 = \sum_{i=1}^{n} V_i$, we deduce the inequality
\[
\bigl \| \widetilde{\by}(\bx_{\cdot, 1}) - \widetilde{\by}(\boldsymbol{0}) \bigr\|_2^2 \leq C \| \bt_{0}\|_{\infty}^2 \log{n} \qquad \text{ with probability } \qquad \geq 1 - \frac{C}{n^2}.
\]
Combining this inequality with the inequalities~\eqref{ineq:bound-B-step1}--\eqref{ineq:bound-B-step3}, we deduce that with probability at least $1 - C/n^2$, $B \leq C \| \bt_0 \|_{\infty} \sqrt{\log{n}}$, which concludes the proof.  \qed

\subsection{Proof of Lemma~\ref{lem:l2_norm_theta_hat}(c)}\label{sec:proof-theta-hat-c}
Without loss of generality, we will prove that for a large enough constant $C$,
\begin{align*}
	\pr\left\{\lvert \< \bz_1, \widehat{\bt}\> \rvert \geq C\sqrt{\log{n}}\right\} \leq \frac{C}{n^2},
\end{align*}
and conclude via an application of the union bound.  To this end, we define the leave-one-sample-out loss as
\begin{align}\label{eq:loso-loss}
\mathcal{L}^{(\backslash 1)}_n(\bt; \bZ, \lambda) = \frac{1}{n} \sum_{i=2}^{n} \rho(-y_i \< \bz_i, \bt\>) + \frac{\lambda}{2p} \| \bt \|_2^2,
\end{align}
and its minimizer as $
\hbtloos = \argmin_{\bt \in \reals^{p}} \mathcal{L}^{(\backslash 1)}_n(\bt; \bZ, \lambda)$.  Decomposing and applying the Cauchy--Schwarz inequality yields
\begin{align}
	\label{ineq:first_ineq_sample_linfty}
	\< \bz_1, \widehat{\bt}\> = \< \bz_1, \hbtloos \> + \< \bz_1, \widehat{\bt} - \hbtloos\> \leq \langle \bz_1, \hbtloos \> + \| \bz_1 \|_2 \| \hbt - \hbtloos \|_2.
\end{align}
We bound each term in turn.  First, note that $\bz_1$ and $\hbtloos$ are independent, whence applying Lemma~\ref{lem:l2_norm_theta_hat}(a) in conjunction with Hoeffding's inequality yields $\langle \bz_1, \hbtloos \rangle \leq C \sqrt{\log{n}}$ with probability at least $1 - C/n^2$.

Turning to the second term, we note that by strong convexity of the original objective $\mathcal{L}_n$~\eqref{eq:loss}, 
\begin{align*}
\| \hbt - \hbtloos \|_2 \leq \frac{2p}{\lambda} \| \nabla \mathcal{L}_n(\hbtloos) \|_2 \overset{\1}{=} \frac{2p}{\lambda} \| \nabla \mathcal{L}_n(\hbtloos) - \nabla \mathcal{L}_n^{(\backslash 1)}(\hbtloos)\|_2 &= \frac{2p}{\lambda n} \| \bz_1 \rho'(-y_i\langle \bz_1, \hbtloos \rangle) \|_2,
\end{align*}
where step $\1$ follows from the first order condition for the loss $\mathcal{L}_n^{(\backslash 1)}$~\eqref{eq:loso-loss}.  By boundedness of $\rho'$, it follows from the above display that $\| \hbt - \hbtloos \|_2 \leq C \| \bz_1 \|_2$.  

Combining these bounds with the inequality~\eqref{ineq:first_ineq_sample_linfty} yields
\[
\< \bz_1, \widehat{\bt}\> \leq C\sqrt{\log{n}} + C\| \bz_1 \|_2^2 \overset{\1}{\leq} C\sqrt{\log{n}} \qquad \text{ with probability } \qquad \geq 1 - \frac{C}{n^2},
\]
where in step $\1$, we utilized the sub-Gaussian nature of the vector $\bz_1$ and applied Hoeffding's inequality. The conclusion follows immediately from the above inequality used in conjunction with the union bound. \qed

\section{Auxiliary lemmas}
\begin{lemma}
	\label{lem:derivative_prox}
	Let $f: \mathbb{R} \rightarrow \mathbb{R}$ be convex, bounded below, and smooth.  Additionally, for\\ $x, c \in \mathbb{R} \times \mathbb{R}_{\geq 0}$, let
	\[
	M_f(x; c) = \min_{u \in \mathbb{R}} \left\{f(u) + \frac{c}{2}(u - x)^2\right\}, \;\;\;\;\; \prox_f(x; c) = \argmin_{u \in \mathbb{R}}\left\{f(u) + \frac{c}{2}(u - x)^2\right\}.
	\]
	Then, we have
	\begin{align*}
		&(a)\;f'(\prox_f(x; c)) + c (\prox_f(x; c) - x) = 0,\\
		(b)\;\frac{\partial}{\partial x} M_f(x; c) &= f'(\prox_f(x;c)), \;\;\;\; (c)\; \frac{\partial}{\partial c} M_f(x; c) = \frac{1}{2}\left(\prox_{f}(x; c) - x\right)^2.
	\end{align*}
\end{lemma}
The proof of the above lemma is standard and we omit it for brevity.

\begin{lemma}
	\label{lem:lbgamma}
	Let $M > 0$ be a positive scalar.  There exists a constant $\ubar{\gamma}_M$ such that for all $\lvert x \rvert \leq M$ and $\gamma \in (0, \ubar{\gamma}_M]$, the following hold.
	\begin{itemize}
		\item[(a)] The quantity $
		\prox^{+}_{\rho}(x; \gamma) := \argmin_{u \in \mathbb{R}} \bigl\{ \rho(u) + \frac{\gamma}{2} (u - x)^2\bigr\}$,
		satisfies the upper bound $
		\prox^{+}_{\rho}(x; \gamma) \leq \frac{1}{4} \log{\gamma}$.
		\item[(b)] 
		The quantity $
		\prox^{-}_{\rho}(x; \gamma) := \argmin_{u \in \mathbb{R}} \bigl\{ \rho(-u) + \frac{\gamma}{2} (u - x)^2\bigr\}$
		satisfies the lower bound $
		\prox^{-}_{\rho}(x; \gamma) \geq \frac{1}{4}\log{\frac{1}{\gamma}}$.
	\end{itemize}
\end{lemma}
\begin{proof}
We restrict ourselves to the proof of part (a), noting that part (b) follows from nearly identical steps.  

Consider setting $
\ubar{\gamma}_M = \bigl[8 \cdot (1 \vee M) \cdot \bigl(1 + e^M\bigr)\bigr]^{-2}$ and suppose towards a contradiction that 
\begin{align}\label{ineq:contradiction-hypothesis-small-gamma-growth}
	\prox_{\rho}^{+}(x; \gamma) \geq (1/4) \cdot \log{\gamma}, \qquad \text{ for all } \gamma \leq \ubar{\gamma}_M.
\end{align}
Applying the first order conditions, we deduce 
\begin{align}
	\label{exp:stat_prox}
	\rho'(\prox^{+}_{\rho}(x; \gamma)) + \gamma \prox^{+}_{\rho}(x; \gamma) - \gamma x = 0.
\end{align}
Next, we study the LHS of the display above under the hypothesis~\eqref{ineq:contradiction-hypothesis-small-gamma-growth}.  To this end, straightforward calculation yields the chain of inequalities
\begin{align*}
	\rho'(\prox^{+}_{\rho}(x; \gamma)) + \gamma \prox^{+}_{\rho}(x; \gamma) - \gamma x &\geq \frac{e^{\prox^{+}_{\rho}(x; \gamma)}}{1 + e^M} + \gamma \prox^{+}_{\rho}(x; \gamma) - \gamma x\\
	&\overset{\mathsf{(i)}}{\geq} \frac{\gamma^{1/4}}{1 + e^M} + \frac{1}{4}\gamma \log{\gamma} - \gamma x =: f(\gamma),
\end{align*}
where step $\1$ follows under the hypothesis~\eqref{ineq:contradiction-hypothesis-small-gamma-growth}.
Now, note that the function $f$ satisfies $f(0) = 0$; we claim that moreover $f$ is strictly increasing on the interval $(0, \ubar{\gamma}_M]$.  Taking this claim as given, we complete the proof upon noting that for all $0 < \gamma \leq \ubar{\gamma}_M$, $f(\gamma) > 0$, thus contradicting the stationary condition~\eqref{exp:stat_prox}.   

It remains to prove that $f$ is strictly increasing.  To this end, we compute the derivative
\begin{align}\label{eq:derivative-f-numeric}
	f'(\gamma) = \frac{\gamma^{-3/4}}{4(1 + e^M)} + \frac{1}{4} \log{\gamma}  + \frac{1}{4} - x,
\end{align}
and note the following numeric inequality, whose justification is provided at the end of the proof,
	\begin{align}\label{ineq:ubar-gamma-numeric}
	\frac{\gamma^{-3/4}}{8 \cdot (1 + e^M)} > M \vee \frac{1}{4} \cdot \log\Bigl(\frac{1}{\gamma}\Bigr), \qquad \text{ for all } 0 < \gamma \leq \ubar{\gamma}_M.
\end{align}
Substituting the above inequality into the RHS of the derivative computation~\eqref{eq:derivative-f-numeric} yields $f'(\gamma) > 0$ for $0 < \gamma \leq \ubar{\gamma}_M$, which completes the proof.

\bigskip
\noindent \underline{Proof of the inequality~\eqref{ineq:ubar-gamma-numeric}:}  We first show that for all $\gamma \in (0, \ubar{\gamma}_M]$, 
$
\frac{\gamma^{-3/4}}{8 \cdot (1 + e^M)} >\log\Bigl(\frac{1}{\gamma}\Bigr)$.
To this end, consider the map $\Psi: \gamma\mapsto \gamma^{-3/4}/\log(1/\gamma)$, and note that $\Psi$ is strictly decreasing on $\mathbb{R}_{+}$.  Additionally, note that
\[
\Psi(\ubar{\gamma}_M) = \frac{\Bigl[8 \cdot (1 \vee M) \cdot \bigl(1 + e^M\bigr)\Bigr]^{3/2}}{2\log\Bigr(8 \cdot (1 \vee M) \cdot \bigl(1 + e^M\bigr)\Bigr)} \overset{\1}{>} 8 \cdot (1 \vee M) \cdot \bigl(1 + e^M\bigr),
\]
where step $\1$ follows by noting that the map $t \mapsto t^{1/2}/(2 \cdot \log(t))$ is increasing on the domain $[e^2, \infty)$ and is greater than $1$ when evaluated at $e^2$.  

We next show that for all $\gamma \in (0, \ubar{\gamma}_M]$, $
\frac{\gamma^{-3/4}}{8 \cdot (1 + e^M)} > M$.
The function $\gamma \mapsto \gamma^{-3/4}$ is clearly decreasing on $\mathbb{R}_{+}$, whence it suffices to verify the inequality at the right endpoint $\ubar{\gamma}_M$, which is immediate.
\end{proof}

\section{Additional numerical experiments}
\label{sec:additional-numerical-experiments}
In Section~\ref{sec:numerical-illustration}, we provided contour plots in which we fixed the probability with which an entry is missing as $\alpha = 0.704$ and compared the Bayes test error to that of single imputation over a large swath of parameters $\delta$ and $R$.  Here, we do the same, fixing the parameters $\delta$ and $R$ in turn.  Figure~\ref{fig:fixdelta} fixes the parameter $\delta$ and Figure~\ref{fig:fixR} fixes the parameter $R$.
\begin{figure*}[!h]
\begin{subfigure}[b]{0.2\textwidth}
	\centerline{\includegraphics[scale=0.5]{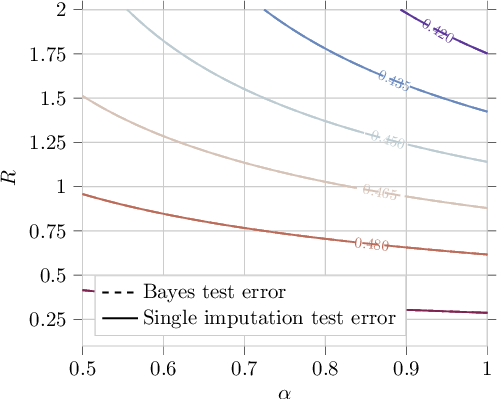}}
	\caption{Overlay of the two test errors, $\delta = 0.4$.}    
	\label{subfig:overlay-fix-deltasmall}
\end{subfigure}
\qquad
\begin{subfigure}[b]{0.2\textwidth}  
	\centerline{\includegraphics[scale=0.5]{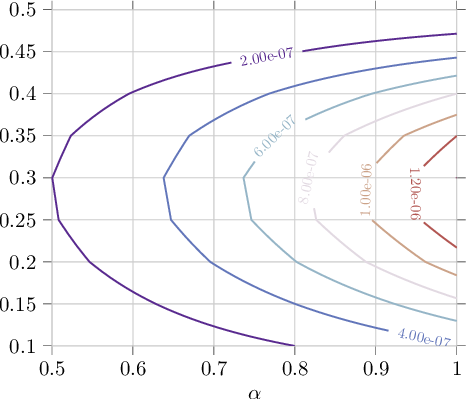}}
	\caption{Difference between test errors, $\delta = 0.4$.}
	\label{subfig:difference-fix-deltasmall}
\end{subfigure}
\qquad
\begin{subfigure}[b]{0.2\textwidth}   
	\centerline{\includegraphics[scale=0.5]{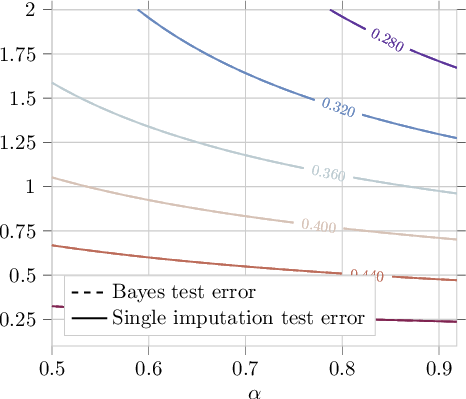}}
	\caption{Overlay of the two test errors, $\delta=20$.}
	\label{subfig:overlay-fix-deltalarge}
\end{subfigure}
\qquad
\begin{subfigure}[b]{0.2\textwidth}   
	\centerline{\includegraphics[scale=0.5]{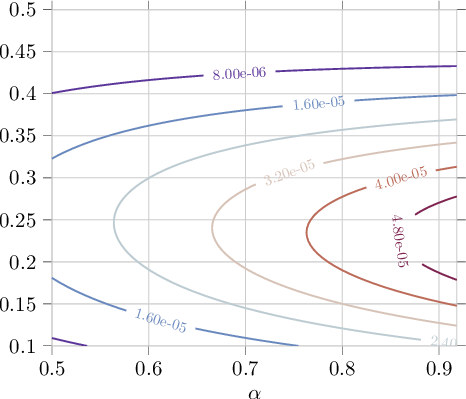}}
	\caption{Difference between test errors, $\delta=20$.}    
	\label{subfig:difference-fix-deltalarge}
\end{subfigure}
	\caption{A comparison of the Bayes test error to the test error of optimally regularized single imputation when the ratio of samples to dimension $\delta = n/p$ is fixed.  Subfigure~\ref{subfig:overlay-fix-deltasmall} plots contours of both test errors on the same plot for a small value of $\delta$ ($\delta  = 0.4$) and Subfigure~\ref{subfig:difference-fix-deltasmall} zooms in to show contours of the difference between the two test errors for this setting.  Subfigures~\ref{subfig:overlay-fix-deltalarge} and~\ref{subfig:difference-fix-deltalarge} paint a similar picture for a much larger value of $\delta$ ($\delta = 20$).}
	\label{fig:fixdelta}
\end{figure*}

\begin{figure*}[!h]
	\begin{subfigure}[b]{0.2\textwidth}
	\centerline{\includegraphics[scale=0.5]{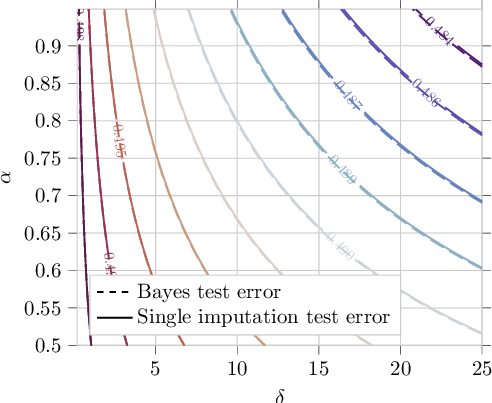}}
	\caption{Overlay of the two test errors, $R = 0.2$.}    
	\label{subfig:overlay-fix-Rsmall}
\end{subfigure}
\qquad
\begin{subfigure}[b]{0.2\textwidth}  
	\centerline{\includegraphics[scale=0.5]{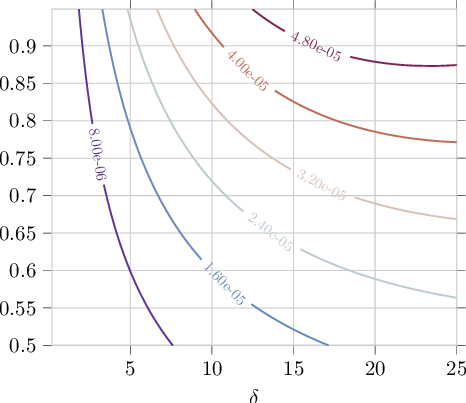}}
	\caption{Difference between test errors, $R = 0.2$.}
	\label{subfig:difference-fix-Rsmall}
\end{subfigure}
\qquad
\begin{subfigure}[b]{0.2\textwidth}   
	\centerline{\includegraphics[scale=0.5]{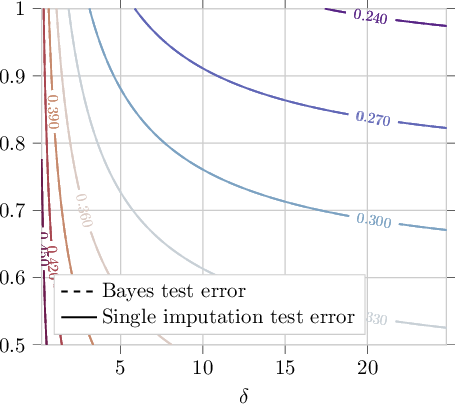}}
	\caption{Overlay of the two test errors, $R = 2$.}
	\label{subfig:overlay-fix-Rlarge}
\end{subfigure}
\qquad
\begin{subfigure}[b]{0.2\textwidth}   
	\centerline{\includegraphics[scale=0.5]{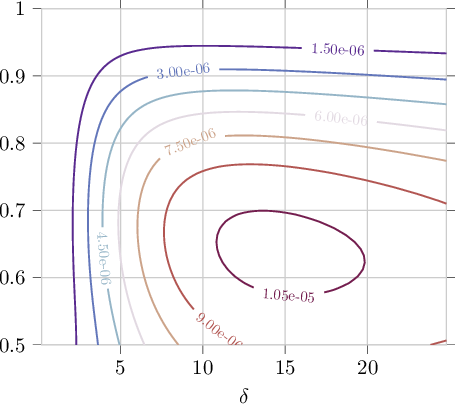}}
	\caption{Difference between test errors, $R=2$.}    
	\label{subfig:difference-fix-Rlarge}
\end{subfigure}
	\caption{A comparison of the Bayes test error to the test error of optimally regularized single imputation when the radius of the problem $R = \| \bt_0 \|_2/\sqrt{p}$ is fixed.  Subfigure~\ref{subfig:overlay-fix-Rsmall} plots contours of both test errors on the same plot for a small value of $R$ ($R = 0.2$) and Subfigure~\ref{subfig:difference-fix-Rsmall} zooms in to show contours of the difference between the two test errors when $R=0.2$.  In a similar vein, Subfigures~\ref{subfig:overlay-fix-Rlarge} and~\ref{subfig:difference-fix-Rlarge} repeat the same experiment for a larger value of $R$ ($R = 2$).}
	\label{fig:fixR}
\end{figure*}
\end{document}


